\documentclass[10pt,a4paper]{article}
\usepackage{amsfonts}
\usepackage[latin9]{inputenc}
\usepackage{amsmath}
\usepackage{amssymb}
\usepackage{breqn}
\usepackage{url}
\usepackage{graphicx}
\usepackage[pdfstartview=FitH]{hyperref}
\usepackage{amsthm}
\usepackage{authblk}
\usepackage{hyperref}
\usepackage{url}
\usepackage{appendix}
\usepackage{minitoc}
\makeatletter
\begin{document}
\newtheorem{theorem}{Theorem}[section]
\newtheorem*{theorem*}{Theorem}
\newtheorem{lemma}[theorem]{Lemma}
\newtheorem{definition}[theorem]{Definition}
\newtheorem{claim}[theorem]{Claim}
\newtheorem{example}[theorem]{Example}
\newtheorem{remark}[theorem]{Remark}
\newtheorem{proposition}[theorem]{Proposition}
\newtheorem{corollary}[theorem]{Corollary}

\title{Local spectral expansion approach to high dimensional expanders}
\author{Izhar Oppenheim}

\newcommand{\Addresses}{{
  \bigskip
  \footnotesize

  IZHAR OPPENHEIM, \textsc{Department of Mathematics, The Ohio State University, Columbus, OH 43210, USA}\par\nopagebreak
  \textit{E-mail address:} \texttt{izharo@gmail.com}
}}

\maketitle
\textbf{Abstract}. This paper introduces the notion of local spectral expansion of a simplicial complex as a possible analogue of spectral expansion defined for graphs. We show the condition of local spectral expansion has several nice implications. For example, for a simplicial complex with local spectral expansion we show vanishing of cohomology with real coefficients, Cheeger type inequalities and mixing type results and geometric overlap results.   

\section{Introduction}

Let  $G=(V,E)$ be a finite graph without loops or multiple edges. For a vertex $u \in V$, denote by $m(u)$ the valency of $u$, i.e.,
$$m(u) = \vert \lbrace (u,v) \in E \rbrace \vert . $$
The Cheeger constant of the graph defined as follows: for $\emptyset \neq U \subseteq V$, denote 
$$m(U) = \sum_{u \in U} m(u).$$
The Cheeger constant of $G$, $h(G)$, is
$$h(G) = \min_{m(U) \leq \frac{1}{2} m(V)} \dfrac{\vert \lbrace (u,v) \in E : u \in U, v \in V \setminus U \rbrace \vert}{m(U)} .$$
Note that if $G$ is connected, then $h(G) > 0$ and that for any graph $G$, $h(G) \leq 1$. For $\varepsilon >0$, a graph $G$ is called a $\varepsilon$-expander if $h(G) \geq \varepsilon$. The intuition behind this definition is that the larger the $\varepsilon$, the more connected the graph. Next, we'll recall the notion of an family of  expanders. A family of graphs, $\lbrace G_j \rbrace_{j \in \mathbb{N}}$ is called a family of expanders if there is a $\varepsilon >0$ such that
$$\forall j \in \mathbb{N}, h(G_j) \geq \varepsilon .$$
For applications, one is usually interested in a family of expanders with constant valency (i.e., $\exists k, \forall j, \forall u \in V_j, m(u) =k$)  or at least uniformly bounded valency (i.e., $\exists k, \forall j, \forall u \in V_j, m(u) \leq k$). \\ \\

An equivalent definition of a family of expanders is relies on the graph Laplacian.  Recall that the (normalized) Laplacian on $G$ is a positive operator $\mathcal{L}$ on $L^2 (V,\mathbb{R})$ defined by the matrix
$$\mathcal{L} (u,v)= \begin{cases}
1 & u=v \\
-\dfrac{1}{\sqrt{m(u)m(v)}} & (u,v) \in E \\
0 & \text{otherwise}
\end{cases}$$
If $G$ is connected then $\mathcal{L}$ has the eigenvalue $0$ with multiplicity $1$ (the eigenvector is the constant function) and all other the eigenvalues are positive. Denote by $\lambda (G)$ the smallest positive eigenvalue of $\mathcal{L}$ of $G$. $\lambda (G)$ is often referred to as the spectral gap of $G$. spectral gap of $G$ and its Cheeger constant are connected through the inequalities:
$$\dfrac{h(G)^2}{2} \leq \lambda (G) \leq 2 h(G).$$
(for proof, see for instance \cite{Chung}[Lemma 2], \cite{Chung}[Theorem 1]).
Thus an equivalent definition of a family of expanders is as follows: a family of graphs $\lbrace G_j \rbrace_{j \in \mathbb{N}}$ is a family of expanders if all the graphs $G_j$ are connected and there is $\lambda >0$ such that 
$$\forall j \in \mathbb{N}, \lambda (G_j) \geq \lambda .$$
For some applications one is interested not just $\lambda (G)$ but also in the largest eigenvalue of $\mathcal{L}$, denoted here as $\kappa (G)$. For $\lambda >0, \kappa <2$, we shall call $G$ a two-sided $(\lambda, \kappa)$ expander if $$\lambda (G) \geq \lambda, \kappa (G) \leq \kappa.$$
A $(\lambda, \kappa)$ expander has "nice" properties (such as mixing) when $\lambda, \kappa$ are both close to $1$. \\ \\

In recent years, expanders had vast applications in pure and applied mathematics (see \cite{LubExpGr}). This fruitfulness of the theory of expander graph, raises the question - what should be the high dimensional analogue of expanders?, i.e., what is the analogous definition of an expander complex when one considers a $n$-dimensional simplicial complex, $X$, instead of a graph. In \cite{LubHighDim} two main approaches are suggested: \\
The first is through the $\mathbb{F}_2$-coboundary expansion of $X$ originated in \cite{LM}, \cite{MW} and \cite{Grom}. The second is through studying the spectral gap of the $(n-1)$-Laplacian of $X$ (where $n$ is the dimension of $X$)  or the spectral gaps of all $0,..,(n-1)$-Laplacians of $X$ (see \cite{PRT}, \cite{P}). One of the difficulties with both approached are that both the $\mathbb{F}_2$-coboundary expansion and the spectral gap of the $n-1$-Laplacian are usually hard to calculate or even bound in examples. \\ \\

This paper suggests a new approach that we call "local spectral expansion" (or $1$-dimensional spectral expansion). Recall that for a simplicial complex $X$ of dimension $n$ and a simplex $\lbrace u_0,...,u_k \rbrace \in X^{(k)}$, the link of $\lbrace u_0,...,u_k \rbrace$ denoted $X_{\lbrace u_0,...,u_k \rbrace}$ is a simplicial complex of dimension $\leq n-k-1$ defined as:
$$X_{\lbrace u_0,...,u_k \rbrace}^{(j)} = \lbrace \lbrace v_0,...,v_j \rbrace \in X^{(j)} : \lbrace u_0,...,u_k,v_0,...,v_j \rbrace \in X^{(k+j+1)} \rbrace .$$
Note that if $X$ is pure $n$ dimensional (i.e., every simplex of $X$ is a face of a simplex of dimension $n$), then $X_{\lbrace u_0,...,u_k \rbrace}$ is of dimension exactly $n-k-1$. Next, we can turn to define local spectral expansion:
\begin{definition}
For $\lambda >\frac{n-1}{n}$, a pure $n$-dimensional simplicial complex will be said to have $\lambda$-local spectral expansion if:
\begin{itemize}
\item $X$ and all its links (in all dimensions $>0$) are connected. 
\item Every $1$-dimensional link of $X$ has a spectral gap $\geq \lambda$, i.e.,
$$\forall \lbrace u_0,...,u_{n-2} \rbrace \in X^{(n-2)}, \lambda (X_{\lbrace u_0,...,u_{n-2} \rbrace} ) \geq \lambda .$$
\end{itemize} 
For $\lambda >\frac{n-1}{n} ,\kappa <2$, a pure $n$-dimensional simplicial complex will be said to have two sided $(\lambda, \kappa)$-local spectral expansion if:
\begin{itemize}
\item $X$ and all its links (in all dimensions $>0$) are connected. 
\item The non zero spectrum of every $1$-dimensional link is contained in the interval $[\lambda , \kappa]$, i.e.,
$$\forall \lbrace u_0,...,u_{n-2} \rbrace \in X^{(n-2)}, \lambda (X_{\lbrace u_0,...,u_{n-2} \rbrace} ) \geq \lambda, \kappa (X_{\lbrace u_0,...,u_{n-2} \rbrace} ) \leq \kappa .$$
\end{itemize} 
\end{definition}

We remark that for $n=1$, both of the above definitions coincide with the usual definitions for graphs, using the convention $X^{(-1)} = \lbrace \emptyset \rbrace$ and therefore $X_\emptyset = X$.  \\ 

A main advantage of the above definition is that the spectrum of the $1$-dimensional links is usually easy to bound or even calculate explicitly in examples. In this paper we shall show that the local spectral expansion has interesting implications, specifically, we shall show that local spectral expansion implies
\begin{enumerate}
\item  Vanishing of cohomology with real coefficients.
\item Spectral gaps of various Laplacians.
\item Cheeger-type inequalities.
\item Mixing type results and geometric overlap in the case of partite complexes (see definitions below).
\end{enumerate}
We shall also show that two-sided local spectral expansion implies mixing-type results and geometric overlap. All these implications require extra terminology and therefore we shall overview them in the next section. 
\begin{remark}
Theorem \ref{cohomology and spectral gaps - section 2} below actually shows that the above assumption of local spectral expansion is actually more restrictive than the assumption of spectral gaps in all the Laplacians assumed in \cite{P}, given the one normalizes the Laplacians appropriately (see below). Indeed, all of our results stand if one replaces the assumptions on the spectra of all $1$-dimensional links to suitable assumptions on the spectra of all the Laplacians (this should be done carefully in some cases, such as in the partite simplcial complex case). However, we still find the assumption suggested above appealing because of its compact nature. One may think of the notion of local spectral expansion suggeted above as a Zuk-type criterion for high dimensional expansion. 
\end{remark}

\textbf{Structure of this paper.} Section 2 is devoted to an overview of the main results of this paper. Section 3 lays out the framework and notations. Section 4 discusses links of simplcial complexes and the concepts of localization and restriction. Section 5 gives results about spectral gaps of Laplacians. Section 6 contains definitions about graphs which can be derived for a simplicial complex and random walks on these graph. Section 7 is devoted to stating and proving Cheeger-type inequalities for simplicial complexes with local spectral expansion. Section 8 is devoted to stating and proving mixing-type results for simplicial complexes with two-sided local spectral expansion and partite complexes with local spectral expansion. Section 9 provides a proof of geometric overlapping property based on local spectral expansion. Section 10 includes some examples of (families of) complexes with local spectral expansion. The appendix is devoted to a slight generalization of a result by Pach needed in the proof of geometric overlap.

\section{Overview of main results}
Throughout this section, let $X$ be a pure $n$-dimensional simplicial complex such that all the links of $X$ (including $X$ itself, excluding $0$-dimensional links) are connected. To state our results we need to introduce the following function, which we call the homogeneous weight function:
$$m : \bigcup_{k=0}^{n} X^{(k)} \rightarrow \mathbb{R}^+,$$
$$\forall 0 \leq k \leq n, \forall \tau \in X^{(k)},   m(\tau) = (n-k)! \vert \lbrace \sigma \in X^{(n)} : \tau \subseteq \sigma \rbrace \vert .$$
Up to a normalization by a factor, $m$ is just the function counting for every simplex $\tau$ how many $n$-dimensional simplexes contain $\tau$ as a face. Note that since $X$ is pure $n$-dimensional, we get that $m (\tau) >0$ for every $\tau$. Also note that when $X$ is $1$-dimensional, then $m$ is just the function assigning $1$ to each edge and the valency to each vertex. \\
We remark that the function $m$ is used to define the inner product of $k$-forms and therefore our $k$-Laplacians of $X$, $\Delta_k^+,\Delta_k^- , \Delta_k$ differ from those defined in some other papers such as \cite{PRT}. \\
We shall also need the following notation to state some of our results: for $0 \leq k \leq n$ given disjoint, non empty sets $U_0,...,U_k \subset X^{(0)}$, denote
$$m(U_0,...,U_k) = \sum_{\lbrace u_0,...,u_k \rbrace \in X^{(k)}, u_0 \in U_0,...,u_k \in U_k} m(\lbrace u_0,...,u_k \rbrace) .$$

Next, we are ready to review our main results:
\subsection{Cohomology vanishing and Laplacians spectral gaps}
\begin{theorem}
\label{cohomology and spectral gaps - section 2}
Let $X$ a  pure $n$-dimensional simplicial complex with $\lambda$-local spectral expansion (recall $\lambda > \frac{n-1}{n}$). Then for every $0 \leq k \leq n-1$:
\begin{enumerate}
\item The reduced $k$-cohomology with real coefficients vanish, i.e.,$\widetilde{H}^k (X, \mathbb{R})=0$.
\item The space of real $k$-forms (see definitions in section 3) admits a decomposition
$$C^k (X,\mathbb{R}) = ker (\Delta_k^+) \oplus ker (\Delta_k^-).$$
\item There is a constant $a_k = a_k (\lambda)$ such that the non trivial spectrum of $\Delta_k^+$ is in $[a_k, \infty)$, i.e.,
$$Spec (\Delta_k^+) \setminus \lbrace 0 \rbrace \subseteq [a_k, \infty) ,$$
and such that $\lim_{\lambda \rightarrow 1} a_k (\lambda )=1$.
\item If in addition there is $\kappa <2$ such that $X$ has a two-sided $(\lambda, \kappa )$-local spectral expansion, then there is a constant $b_k= b_k (\kappa)$, such that 
$$Spec (\Delta_k^+) \setminus \lbrace 0 \rbrace \subseteq [a_k, b_k] ,$$
and such that $\lim_{\lambda \rightarrow 1} b_k (\lambda )=1$.
\end{enumerate}

\end{theorem} 

\begin{remark}
A version of this theorem can be traced back to the work of Garland in \cite{Gar}. The theorem in \cite{Gar} is less quantitative and does not give estimates on the spectral gaps (it also refers only to Tits-building and not for general simplicial complexes). The interested reader can find a discussion in the introduction section of \cite{Opp} comparing the result stated above to the result in \cite{Gar} and similar results (\cite{BS}, \cite{Zuk}, \cite{DJ2}, \cite{K},...). The version stated here was already proven by the author in \cite{Opp} in a more general setting, but for completeness, we'll repeat the proof below.
\end{remark}
\subsection{Cheeger-type inequalities}
To state the Cheeger-type results, we first redefine the $1$-dimensional case. For a graph $G=(V,E)$ define
$$h^0 (G) = \max \left\lbrace \varepsilon \geq 0: \forall \emptyset \neq U \subseteq V, \varepsilon \dfrac{m(U)}{m(V)} +  \dfrac{\vert \lbrace (u,v) \in E : u \in U, v \in V \setminus U \rbrace \vert}{m(U)} \geq \varepsilon \right\rbrace .$$

It is not hard to show that for every $G$ we have $h^0 (G) \leq 2 h(G)$ (see proposition \ref{h compared to h^0}) and that $\lambda (G) \leq h^0 (G)$ (see proposition \ref{1-dim Cheeger}), therefore 
$$\dfrac{(h^0 (G))^2}{8} \leq \dfrac{h(G)^2}{2} \leq \lambda (G) \leq  h^0 (G).$$
This give justification to use $h^0 (G)$ as the "corrected" Cheeger constant, instead of $h (G)$. Reviewing the definition of $h^0 (G)$ we can see two different measures of regarding a set $U \subset V$: 
\begin{enumerate}
\item The expression $\frac{m(U)}{m(V)}$ which very informally can be described as a measure on "how much the set $U$ is connected within itself with respect to the whole graph".
\item The expression $\frac{\vert \lbrace (u,v) \in E : u \in U, v \in V \setminus U \rbrace \vert}{m(U)}$  which very informally can be described as a measure on "how much the set $U$ is connected to the outside of it".
\end{enumerate}

Using the above reasoning, for $0 \leq k \leq n-1$ and non empty disjoint sets $U_0,...,U_k \subset X^{(0)}$ we shall define 
$$h_{out}^k (U_0,...,U_k) = \begin{cases}
0 & X^{(0)} \setminus \bigcup_{i=0}^k U_i = \emptyset \\
\dfrac{m(U_0,...,U_k, X^{(0)} \setminus \bigcup_{i=0}^k U_i )}{m (U_0,...,U_k)} & \text{otherwise}
\end{cases} , $$
as the $k$-dimensional analogue of $\frac{\vert \lbrace (u,v) \in E : u \in U, v \in V \setminus U \rbrace \vert}{m(U)}$. Note that for every $U_0,...,U_k \subset X^{(0)}$, $h_{out}^k (U_0,...,U_k) \in [0,1]$.  We shall also define a $k$-dimensional analogue of $\frac{m(U)}{m(V)}$ denoted as $h_{inner}^k (U_0,...,U_k) \in [0,1]$. \\
Alas, the definition of $h_{inner}^k (U_0,...,U_k)$ is not straightforward: define the (hyper) graph $X_{k-1}$ as a graph with the vertex set $X^{(k-1)}$ and two vertices are connected by an edge if their corresponding $(k-1)$-dimensional simplices are contained in a single $k$-simplex. The sets $U_0,...,U_k$ defines a subgraph of $X_{k-1}$ denoted $X_k (U_0,...,U_k)$. $X_k (U_0,...,U_k)$ is defined in the following way - a vertex is in $X_k (U_0,...,U_k)$ if it matches a $(k-1)$-simplex $\lbrace u_0,..., u_{k-1} \rbrace$ such there is some $0 \leq i \leq k$ such that
$$u_0 \in U_0,...,u_{i-1} \in U_{i-1}, u_i \in U_{i+1},...,u_{k-1} \in U_k.$$
An edge is in $X_k (U_0,...,U_k)$ if it matches a $k$ simplex $\lbrace u_0,...,u_k \rbrace$ such that $u_0 \in U_0,...,u_k \in U_k $. Now define the following random walk: pick a vertex in $X_k (U_0,...,U_k)$ is random with respect to its weight (under the homogeneous weight function) and preform a random walk of $X_{k-1}$ with respect to the homogeneous weight function (more detailed description can be found in definitions \ref{Inner connectivity for U_i's}, \ref{InnerCon} below). $h^{inner} (U_0,...,U_k)$ is the conditional probability that the $2$-step random walk described above stays in $X_k (U_0,...,U_k)$ given that the $1$-step random walk described above stayed in $X_k (U_0,...,U_k)$. \\  
Under these definitions we define for an $n$ dimensional simplicial complex
$$h^k (X) = \max \lbrace \varepsilon \geq 0 : \forall \emptyset \neq U_0,..., \emptyset \neq U_k \subseteq X^{(0)} \text{ pairwise disjoint}$$
$$ \left(\dfrac{k}{k+1} +\varepsilon \right) h_{inner}^k (U_0,...,U_k) +\dfrac{1}{k+1} h_{out}^k (U_0,...,U_k) \geq \varepsilon  \rbrace.$$ 
After this set up, we are finally ready to state our Cheeger-type inequality:
\begin{theorem}
\label{Cheeger inequality - section 2}
Let $X$ a  pure $n$-dimensional simplicial complex with $\lambda$-local spectral expansion. Then for every $0 \leq k \leq n-1$ there is a $\varepsilon_k = \varepsilon_k (\lambda)$ such that $h^k (X) \geq \varepsilon_k$ and such that $\lim_{\lambda \rightarrow 1} \varepsilon_k (\lambda) = \frac{1}{k+1}$.
\end{theorem}

\begin{remark}
Cheeger-type inequalities for simplicial complexes where already considered with respect to the $(n-1)$-dimensional Laplacian - see for instance \cite{PRT} and \cite{GS}. However, our treatment passing to $h^k (X)$ defined above is, as far as we know, new. 
\end{remark}

\subsection{Mixing and geometric overlap}

The expander mixing lemma is usually stated as: 
\begin{lemma}[Expander mixing lemma]
Let $G=(V,E)$ be a $d$-regular graph of $N$ vertices. For disjoint, non empty sets $U_0,U_1 \subset V$ denote $E(U_0,U_1)$ to be the number of edges between $U_0$ and $U_1$. Then for every disjoint, non empty sets $U_0,U_1 \subset V$ one has
\begin{dmath*}
\left\vert E(U_0,U_1) - \dfrac{d \vert U_0 \vert \vert U_1 \vert}{n} \right\vert \leq d \max \lbrace 1- \lambda, \kappa - 1 \rbrace \sqrt{\vert U_0 \vert \vert U_1 \vert} .
\end{dmath*}
\end{lemma}
Note that since $G$ is assumed to be $d$-regular, $\vert U_0 \vert = \frac{m(U_0)}{d}, \vert U_1 \vert=\frac{m(U_1)}{d}$. \\
Our version of the above lemma reads as follows:
\begin{theorem}
\label{Mixing in the general case - section 2}
Let $X$ a  pure $n$-dimensional simplicial complex with two sided $(\lambda, \kappa)$-local spectral expansion. Then for every $1 \leq l \leq n$, there are continuous functions $\mathcal{E}_l (\lambda , \kappa)$ and $\mathcal{A}_l (\lambda , \kappa)$ (that can be worked out explicitly as a function of $\lambda, \kappa, l$) such that
$$\lim_{(\lambda, \kappa) \rightarrow (1,1)} \mathcal{A}_l (\lambda , \kappa) = 1, \lim_{(\lambda, \kappa) \rightarrow (1,1)} \mathcal{E}_l (\lambda , \kappa) =0 ,$$
and such that for any non empty, disjoint sets $U_0,...,U_l \subset X^{(0)}$ the following holds:
\begin{dmath*}
\left\vert m(U_0,...,U_l) -  \mathcal{A}_l (\lambda , \kappa) \dfrac{m(U_0) ... m (U_l)}{m (X^{(0)})^{l} } \right\vert \leq 
  \mathcal{E}_l (\lambda , \kappa)  \min_{0 \leq i < j \leq l} \sqrt{m(U_i) m (U_j)}.
\end{dmath*} 
and 
\begin{dmath*}
\left\vert m(U_0,...,U_l) -  \mathcal{A}_l (\lambda , \kappa) \dfrac{m(U_0) ... m (U_l)}{m (X^{(0)})^{l} } \right\vert \leq 
  \mathcal{E}_l (\lambda , \kappa) \left( m(U_0) ... m (U_l) \right)^{\frac{1}{l+1}}.
\end{dmath*} 
\end{theorem}

\begin{remark}
The above result  is very much inspired by the work in \cite{P}, in which the author assumes spectral gaps for all $0,...,(n-1)$-Laplacians and deduces a mixing analogue of the mixing lemma for an $n$-dimensional simplicial complex. Our treatment is very similar to the one taken in \cite{P}, since as stated in theorem \ref{cohomology and spectral gaps - section 2} above, our assumption on the links implies spcetral gaps in all $0,...,(n-1)$-Laplacians when those are normalizes according to the weight function $m$. However, one should note that there are major differences in the end results emanating from the fact that our Laplacian are normalized with respect to $m$ (for instance, we get tighter bounds on the difference in the absolute value). 
\end{remark}

From the above mixing result one can deduce the property of geometric overlap (see definition \ref{geometric overlap definition}) below:

\begin{theorem}
\label{Geometric overlap in the general case - section 2}
Let $X$ be a pure $n$-dimensional simplicial complex. There is a continuous function $\varepsilon (\lambda , \kappa) : [0,1] \times [1,2] \rightarrow \mathbb{R}$ such that $\varepsilon (1,1) >0$ and such that for a simplicial complex with a "good enough" two sided $(\lambda, \kappa)$-local spectral expansion ("good enough" means that $\lambda$ and $\kappa$ are close enough to $1$), then $\varepsilon (\lambda,\kappa) >0$ and $X$ has $\varepsilon (\lambda , \kappa)$-geometric overlap.
\end{theorem}

\begin{remark}
\label{remark - for mixing to geometric overlap}
The method on passing from a mixing type result to geometric overlap is taken from \cite{P} and \cite{FGLNP} (we claim no originality here). The main idea is to use a theorem of Pach in \cite{Pach}. We had do slightly adapt the result in \cite{Pach} to our weighted setting and this was done in the appendix.  
\end{remark}

\subsection{Mixing and geometric overlap for partite simplicial complexes}

Recall that a graph $(V,E)$ is called bipartite if the vertex set $V$ can be partitioned into two sides $S_0,S_1$ such that $E \subseteq \lbrace \lbrace u,v \rbrace : u \in S_0, v \in S_1 \rbrace$. The spectrum of a bipatite graph is symmetric around $1$ and the version of the mixing lemma for bipartite graphs uses this property and can be deduced only from the spectral expansion (and not the two-sided spectral expansion). Generalizing to higher dimension, we shall say that a pure $n$-dimensional simplicial complex $X$ is $(n+1)$-partite, if $X^{(0)}$ can be partitioned into $n+1$ sets $S_0,...,S_n$ such that
$$X^{(n)} \subseteq \lbrace \lbrace u_0,...,u_n \rbrace : u_0 \in S_0,...,u_n \in S_n \rbrace .$$
Our version of mixing for $(n+1)$-partite simplicial complexes reads as follows:
\begin{theorem}
\label{Mixing for partite case - section 2}
Let $X$ be a pure $n$-dimensional, $(n+1)$-partite simplicial complex such that all the links of $X$ of dimension $>0$ are connected. Denote by $S_0,...,S_n$ the sides of $X$. If $X$ has $\lambda$-local spectral expansion then for every $1 \leq l \leq n$, there is a continuous function $\mathcal{E}_l (\lambda)$ such that
$$\lim_{\lambda \rightarrow 1} \mathcal{E}_l (\lambda) = 0 , $$
and such that  every non empty disjoint sets $U_0 \subseteq S_0,...,U_l \subseteq S_l$ the following inequalities holds:
\begin{dmath*}
\left\vert \dfrac{m(U_0,...,U_l)}{m(X^{(0)})} -  \dfrac{1}{(n+1)n(n-1)...(n-l+1)} \dfrac{m(U_0) ... m (U_l)}{m(S_0)...m(S_l)} \right\vert \leq 
\mathcal{E}_l (\lambda ) \min_{0 \leq i < j \leq l} \sqrt{\dfrac{m(U_i) m(U_j)}{m(S_i) m(S_j)}},
\end{dmath*}
and
\begin{dmath*}
\left\vert \dfrac{m(U_0,...,U_l)}{m(X^{(0)})} -  \dfrac{1}{(n+1)n(n-1)...(n-l+1)} \dfrac{m(U_0) ... m (U_l)}{m(S_0)...m(S_l)} \right\vert \leq 
\mathcal{E}_l (\lambda) \left( \dfrac{ m (U_0)...m (U_l) }{m(S_0)...m(S_l)}) \right)^{\frac{1}{l+1}} .
\end{dmath*}
\end{theorem}

\begin{remark}
Mixing results for partite Ramanujan complexes were already proven in \cite{FGLNP} and \cite{EGL}. The treatment in those papers was very different and relied on quantitative estimate for
Kazhdan property (T) of $PGL_n (F)$. Our treatment relays only on spectral gap estimates and therefore applies to any partite simplicial complex.     
\end{remark}

From the above mixing result one can deduce the property of geometric overlap (see remark \ref{remark - for mixing to geometric overlap} above):

\begin{theorem}
\label{Geometric overlap for partite case - section 2}
Let $X$ be a pure $n$-dimensional, $(n+1)$-partite simplicial complex. There is a continuous function $\varepsilon (\lambda) : [0,1] \rightarrow \mathbb{R}$ such that $\varepsilon (1) >0$ and such that for a simplicial complex with a "good enough"  $\lambda$-local spectral expansion ("good enough" means that $\lambda$ is close enough to $1$), we have that $\varepsilon (\lambda) >0$ and that $X$ has $\varepsilon (\lambda)$-geometric overlap.
\end{theorem}

This theorem can be used to prove that partite quotients of affine buildings of type $\widetilde{A}_n$ have geometric overlap property, given that the thickness of the building is large enough (a different proof of this fact was already given in  in \cite{FGLNP}). It can also be used to prove that partite quotients of affine buildings of any type have geometric overlap property given that the building thickness is large enough (this was conjectured in \cite{LubHighDim}, but as far as we know, we are the first to provide a proof).        

\section{framework}
The framework suggested here owes its existence to the framework suggested in \cite{BS}. Throughout this paper, $X$ is pure $n$-dimensional finite simplicial complex, i.e., every simplex in $X$ is contained in at least one $n$-dimensional simplex. 
\subsection{Weighted simplicial complexes}
Our results in the previous section were stated for a specific function $m$. However, the function $m$ define above is only one example of a weight function on $X$. Since our results extend to any weight function, we shall work with the general definition of a weighted simplicial complex defined below and we shall refer to the specific function $m$ we used in the previous section the homogeneous weight  function. \\ 
For $-1\leq k\leq n$, denote:
\begin{itemize}
\item $X^{(k)}$ is the set of all $k$-simplices in $X$.
\item $\Sigma(k)$ the set of ordered $k$-simplices, i.e., $\sigma \in \Sigma(k)$ is an ordered $(k+1)$-tuple of vertices that form a $k$-simplex in $X$.
\end{itemize} 
Note the  $\Sigma (-1) = X^{(-1)}$ is just the singleton $\lbrace \emptyset \rbrace$. 
\begin{definition}
\label{weightedDef}
A simplicial complex $X$ is called weighted if there is a strictly positive function $m : \bigcup_{-1 \leq k \leq n} X^{(k)} \rightarrow \mathbb{R}^+$ (called the weight function) such that for every $-1 \leq k \leq n-1$,  we have the following equality
$$ \sum_{\sigma \in X^{(k+1)}, \tau \subset \sigma} m( \sigma ) =  m (\tau ) ,$$
where $\tau \subset \sigma$ means that $\tau$ is a face of $\sigma$. 
\end{definition}

Given a weight function $m$ we can define it on ordered simplices (denoting it again as $m$) as
$$m( (v_0,...,v_k)) = m( \lbrace v_0,...,v_k \rbrace), \forall  (v_0,...,v_k) \in \bigcup_{-1 \leq k \leq n} \Sigma (k).$$
By the definition of $m$, we have the following equality:
$$\forall  \tau \in \bigcup_{-1 \leq k \leq n-1} \Sigma (k), \sum_{\sigma \in \Sigma (k+1), \tau \subset \sigma} m( \sigma ) = (k+2)!  m (\tau ) ,$$
where $\tau \subset \sigma$ means that all the vertices of $\tau$ are contained in $\sigma$ (with no regard to the ordering). We note that under this equality one can start with a strictly positive function $m: \bigcup_{-1 \leq k \leq n} \Sigma (k)  \rightarrow \mathbb{R}^+$ and get a weight function $m: \bigcup_{-1 \leq k \leq n} X^{(k)} \rightarrow \mathbb{R}^+$:

\begin{proposition}
\label{weightedOrderDef}
Let $m: \bigcup_{-1 \leq k \leq n} \Sigma (k)  \rightarrow \mathbb{R}^+$ be a strictly positive function such that:
\begin{enumerate}
\item For every $1 \leq k \leq n$, and every permutation $\pi \in Sym (\lbrace 0,..,k \rbrace)$ we have 
$$m( (v_0,...,v_k) ) = m( (v_{\pi (0)},...,v_{\pi (k)} )), \forall (v_0,...,v_k) \in \Sigma (k).$$
\item 
$$\forall  \tau \in \bigcup_{-1 \leq k \leq n-1} \Sigma (k), \sum_{\sigma \in \Sigma (k+1), \tau \subset \sigma} m( \sigma ) = (k+2)!  m (\tau ) .$$
\end{enumerate}
Then $m : \bigcup_{-1 \leq k \leq n} X^{(k)} \rightarrow \mathbb{R}^+$ defined as 
$$ m( \lbrace v_0,...,v_k \rbrace) = m( (v_0,...,v_k)), \forall   \lbrace v_0,...,v_k \rbrace \in \bigcup_{-1 \leq k \leq n} X^{(k)},$$
is a weight function.
\end{proposition}

\begin{proof}
Trivial.
\end{proof}

\begin{remark}
From the definition of the weight function $m$, it should be clear that every map  $m : X^{(n)} \rightarrow \mathbb{R}^+$ can be extended in a unique way to a weight function $m: \bigcup_{-1 \leq k \leq n} X^{(k)} \rightarrow \mathbb{R}^+$. 
\end{remark}

\begin{definition}
$m$ is called the homogeneous weight on $X$ if for every $\sigma \in X^{(n)}$, we have $m( \sigma )=1$.
\end{definition}

\begin{proposition}
\label{weight in n dim simplices}
For every $-1 \leq k \leq n$ and every $\tau \in X^{(k)}$ we have that
$$\dfrac{1}{(n-k)!}  m (\tau) =\sum_{\sigma \in X^{(n)}, \tau \subseteq \sigma} m (\sigma ),$$
where $\tau \subseteq \sigma$ means that $\tau$ is a face of $\sigma$. 
\end{proposition}

\begin{proof}
The proof is by induction. For $k=n$ this is obvious. Assume the equality is true for $k+1$, then for $\tau \in X^{(k)}$ we have
\begin{dmath*}
m( \tau ) = \sum_{\sigma \in X^{(k+1)}, \tau \subset \sigma} m(\sigma ) = \sum_{\sigma \in X^{(k+1)}, \tau \subset \sigma} (n-k-1)! \sum_{\eta \in X^{(n)}, \sigma \subset \eta} m(\eta) = (n-k) (n-k-1)! \sum_{\eta \in X^{(n)}, \tau \subset \eta} m(\eta)=   (n-k)! \sum_{\eta \in X^{(n)}, \tau \subset \eta} m(\eta) .
\end{dmath*}
\end{proof}

\begin{corollary}
\label{weight in l dim simplices}
For every $-1 \leq k < l \leq n$ and every $\tau \in X^{(k)}$ we have
$$\dfrac{1}{(l-k)!} m(\tau) = \sum_{\sigma \in X^{(l)}, \tau \subset \sigma} m(\sigma) .$$
\end{corollary}

\begin{proof}
For every $\sigma \in X^{(l)}$ we have
$$\dfrac{1}{(n-l)!} m(\sigma ) = \sum_{\eta \in X^{(n)}, \sigma \subseteq \eta} m(\eta) .$$
Therefore
\begin{dmath*}
\sum_{\sigma \in X^{(l)}, \tau \subset \sigma} m(\sigma) = \sum_{\sigma \in X^{(l)}, \tau \subset \sigma} (n-l)! \sum_{\eta \in X^{(n)}, \sigma \subseteq \eta} m(\eta) = \dfrac{(n-k)!}{(l-k)! (n-k - (l-k) )! } (n-l)! \sum_{\eta \in X^{(n)}, \tau \subseteq \eta} m(\eta) = \dfrac{(n-k)!}{(l-k)!}  \sum_{\eta \in X^{(n)}, \tau \subseteq \eta} m(\eta) = \dfrac{1}{(l-k)!} m (\tau ) .
\end{dmath*}
\end{proof}

From now on, we shall always assume that $X$ is weighted. 

\subsection{Cohomology with real coefficients}

For $-1 \leq k\leq n$, denote 
$$C^{k}(X, \mathbb{R}) = \lbrace \phi : \Sigma (k) \rightarrow \mathbb{R} : \phi \text{ is antisymmetric} \rbrace.$$
We recall that $\phi : \Sigma (k) \rightarrow \mathbb{R}$ is called antisymmetric, if for every $(v_0,...,v_k) \in  \Sigma (k)$ and every permutation $\pi \in Sym (\lbrace 0,...,k \rbrace)$, we have 
$$\phi ((v_{\pi (0)},...,v_{\pi (k)})) = sgn (\pi) \phi ((v_0,...,v_k)).$$
Every $\phi \in C^k (X,\mathbb{R})$ is called a \textit{$k$-form}, and $C^k (X, \mathbb{R})$ is called the \textit{space of $k$-forms}. \\
For $-1 \leq k \leq n$ define an inner product on $C^{k}(X,\mathbb{R})$ as:
$$\forall \phi, \psi \in C^{k}(X,\mathbb{R}), \left\langle \phi , \psi \right\rangle = \sum_{\tau \in \Sigma (k)} \dfrac{m(\tau)}{(k+1)!} \phi (\tau ) \psi (\tau ).$$
Note that with this inner product $C^{k}(X,\mathbb{R})$ is a (finite dimensional) Hilbert space. Denote the norm induced by this inner product as $\Vert . \Vert$. 
For $-1 \leq k \leq n-1$ define the differential $d_k : C^k (X,\mathbb{R}) \rightarrow C^{k+1} (X,\mathbb{R})$ in the usual way, i.e., for every $\phi \in C^k (X,\mathbb{R})$ and every $(v_0,...,v_{k+1})$,
$$(d_k \phi ) ((v_0,...,v_{k+1} )) = \sum_{i=0}^{k+1} (-1)^i \phi ((v_0,..., \widehat{v_i},...,v_{k+1})).$$
One can easily check that for every $0 \leq k \leq n-1$ we have that $d_{k+1} d_k = 0$ and therefore we can define the cohomology in the usual way:
$$H^k (X, \mathbb{R} ) = \dfrac{ker (d_k)}{im (d_{k-1})} .$$
Next, we describe the discrete Hodge theory in our setting. Define $\delta_k : C^{k+1} (X, \mathbb{R} ) \rightarrow C^{k} (X,\mathbb{R})$ as the adjoint operator of $d_{k}$ (with respect to the inner product we defined earlier on $C^{k} (X, \mathbb{R} ), C^{k-1} (X,\mathbb{R})$). Define further operators $\Delta_k^+, \Delta_k^-, \Delta_k :  C^{k} (X, \mathbb{R} ) \rightarrow C^{k} (X, \mathbb{R} )$ as 
$$\Delta_k^+ = \delta_{k} d_k, \Delta_k^- = d_{k-1} \delta_{k-1}, \Delta_k = \Delta_k^+ + \Delta_k^- .$$
The operators $\Delta_k^+, \Delta_k^-, \Delta_k$ are called the upper Laplacian, the lower Laplacian and the full Laplacian. The reader should note that by definition, all these operators are positive (i.e., self-adjoint with a non negative spectrum). 
\begin{proposition}
\label{HodgeConsid}
For every $1 \leq k \leq n-1$ we have that
$$H^k (X, \mathbb{R} ) = ker (\Delta_k ), $$
and
$$Spec (\Delta^+_{k-1}) \setminus \lbrace 0 \rbrace \subseteq [a,b] \Leftrightarrow Spec (\Delta^-_k) \setminus \lbrace 0 \rbrace \subseteq [a,b],$$
where $Spec (\Delta^+_{k-1}), Spec (\Delta^-_k) $ are the spectrum of $\Delta^+_{k-1}, \Delta^-_k$. 
\end{proposition}

\begin{proof}
Notice that since $d_k^* = \delta_k$ we have the following:
$$im (\Delta_k^+)=  (ker ( \Delta_k^+))^\perp =(ker (d_k))^\perp = im (\delta_k),$$
$$im (\Delta_k^-)=  (ker ( \Delta_k^-))^\perp =(ker (\delta_{k-1}))^\perp = im (d_{k-1}).$$
Therefore, we have an orthogonal decomposition
$$ker (d_k) = ker ( \Delta_k^+) = \left( ker ( \Delta_k^+) \cap ker ( \Delta_k^-) \right) \oplus im (\Delta_k^-) = ker (\Delta_k ) \oplus im (d_{k-1}).$$
Which yields that $H^k (X, \mathbb{R} ) = ker (\Delta_k ) $. $ker (\Delta_k )$ is called the space of harmonic $k$-forms on $X$. Also notice that due to the fact that $\Delta^+_{k-1} = \delta_{k-1} d_{k-1}, \Delta^-_{k} =  d_{k-1} \delta_{k-1}$, we have
$$Spec (\Delta^+_{k-1}) \setminus \lbrace 0 \rbrace \subseteq [a,b] \Leftrightarrow Spec (\Delta^-_k) \setminus \lbrace 0 \rbrace \subseteq [a,b],$$
\end{proof}
The next proposition gives an explicit formula for $\delta_k, \Delta_k^+, \Delta_k^-$:

\begin{proposition}

\begin{enumerate}
\item Let $-1 \leq k \leq n-1$ then: $\delta_k :C^{k+1}(X,\mathbb{R})\rightarrow C^{k}(X,\mathbb{R})$ is \[
\delta_k \phi(\tau)=\sum_{\begin{array}{c}
{\scriptstyle v \in\Sigma(0)}\\
{\scriptstyle v \tau\in \Sigma (k+1)}\end{array}}\frac{m(v\tau)}{m(\tau)}\phi(v\tau),\;\tau\in\Sigma(k)\]
where $v\tau=(v,v_{0},...,v_{k})$ for $\tau=(v_{0},...,v_{k})$. 
\item For $0 \leq k \leq n-1$, $\phi\in C^{k}(X,\mathbb{R})$ and $\sigma\in\Sigma(k)$,
\[
\Delta_k^+ \phi (\sigma) = \phi(\sigma)-\sum_{\begin{array}{c}
{\scriptstyle v \in\Sigma(0)}\\
{\scriptstyle v \sigma \in \Sigma (k+1)}\end{array} }\sum_{0\leq i\leq k}(-1)^{i}\frac{m(v\sigma)}{m(\sigma)}\phi(v\sigma_{i}) .\]
\item For  $0 \leq k \leq n$, $\phi\in C^{k}(X,\mathbb{R})$ and $\sigma\in\Sigma(k)$,
$$ \Delta_k^- \phi (\sigma ) = \sum_{i=0}^k (-1)^i  \sum_{ v \in\Sigma(0),  v\sigma_i \in\Sigma(k) }\frac{m(v\sigma_i)}{m(\sigma_i)}\phi(v\sigma_i) .$$ 

\end{enumerate}
\end{proposition}

\begin{proof}
\begin{enumerate}

\item For $\sigma \in \Sigma (k+1)$ and $\tau \subset \sigma, \tau \in \Sigma (k)$ denote by $[\sigma : \tau ]$ the incidence coefficient of $\tau$ with respect to $\sigma$, i.e., if $\sigma_i$ has the same vertices as $\tau$ then for every $\psi \in C^k (X, \mathbb{R})$ we have $[ \sigma : \tau ] \psi (\tau ) = (-1)^i \psi (\sigma_i)$. Take $\phi \in C^{k+1} (X, \mathbb{R})$ and $\psi \in C^k (X, \mathbb{R})$, then we have
\begin{dmath*}
\left\langle d \psi , \phi \right\rangle = \sum_{\sigma \in \Sigma (k+1)} \dfrac{m(\sigma )}{(k+2)!} \left( \sum_{i=0}^{k+1}(-1)^{i}\psi(\sigma_{i}) \right) \phi ( \sigma )  = 
{ \sum_{\sigma \in \Sigma (k+1)} \dfrac{m(\sigma )}{(k+1)! (k+2)! } \left( \sum_{ \tau \in\Sigma(k), \tau \subset \sigma} [\sigma : \tau] \psi(\tau) \right)  \phi ( \sigma )}  = 
 \sum_{\sigma \in \Sigma (k+1)} \dfrac{m(\tau )}{(k+1)!} \sum_{ \tau \in\Sigma(k) , \tau \subset \sigma}    \psi(\tau) \left( \dfrac{[\sigma : \tau] m(\sigma )}{m(\tau ) (k+2)!} \phi ( \sigma ) \right) = 
\sum_{\tau \in \Sigma (k)} \dfrac{m(\tau )}{(k+1)!} \sum_{ \sigma \in\Sigma(k+1), \tau \subset \sigma }   \psi(\tau)\left( \dfrac{[\sigma : \tau] m(\sigma )}{m(\tau ) (k+2)!} \phi ( \sigma ) \right) = 
\sum_{\tau \in \Sigma (k)} \dfrac{m(\tau )}{(k+1)! }    \psi(\tau) \left( \sum_{ \sigma \in\Sigma(k+1), \tau \subset \sigma} \dfrac{[\sigma : \tau] m(\sigma )}{m(\tau ) (k+2)!} \phi ( \sigma ) \right) = 
\sum_{\tau \in \Sigma (k)} \dfrac{m(\tau )}{(k+1)! }   \psi(\tau) \left( \sum_{ v \in\Sigma(0),  v\tau\in\Sigma(k+1) }\frac{m(v\tau)}{m(\tau)}\phi(v\tau) \right) =
 \left\langle \psi, \sum_{ v \in\Sigma(0),  v\tau\in\Sigma(k+1) }\frac{m(v\tau)}{m(\tau)}\phi(v\tau) \right\rangle .
 \end{dmath*}

\item For every $\phi \in C^k (X, \mathbb{R})$ and every $\sigma \in \Sigma (k)$ we have:
\begin{dmath*}
\delta d \phi (\sigma ) = \sum_{\begin{array}{c}
{\scriptstyle v \in\Sigma(0)}\\
{\scriptstyle v\sigma \in\Sigma(k+1) }\end{array} }\frac{m(v\sigma)}{m(\sigma)} d \phi(v\sigma) = 
\sum_{\begin{array}{c}
{\scriptstyle v \in\Sigma(0)}\\
{\scriptstyle v\sigma \in\Sigma(k+1) }\end{array}}\frac{m(v\sigma)}{m(\sigma)} \phi (\sigma ) - \sum_{\begin{array}{c}
{\scriptstyle v \in\Sigma(0)}\\
{\scriptstyle v\sigma \in\Sigma(k+1) }\end{array} }\sum_{0\leq i\leq k}(-1)^{i}\frac{m(v\sigma)}{m(\sigma)}\phi(v\sigma_{i}) =
\sum_{\begin{array}{c}
{\scriptstyle \gamma \in\Sigma(k+1)}\\
{\scriptstyle \sigma \subset \gamma }\end{array}}\frac{m(\gamma)}{(k+2)! m(\sigma)} \phi (\sigma ) - \sum_{\begin{array}{c}
{\scriptstyle v \in\Sigma(0)}\\
{\scriptstyle v\sigma \in\Sigma(k+1) }\end{array}}\sum_{0\leq i\leq k}(-1)^{i}\frac{m(v\sigma)}{m(\sigma)}\phi(v\sigma_{i}) =
\phi (\sigma ) - \sum_{\begin{array}{c}
{\scriptstyle v \in\Sigma(0)}\\
{\scriptstyle v\sigma \in\Sigma(k+1) }\end{array}}\sum_{0\leq i\leq k}(-1)^{i}\frac{m(v\sigma)}{m(\sigma)}\phi(v\sigma_{i}).
\end{dmath*}
\item For every $\phi \in C^k (X, \mathbb{R})$ and every $\sigma \in \Sigma (k)$ we have: 
$$d \delta \phi (\sigma) = \sum_{i=0}^k (-1)^i \delta \phi (\sigma_i) = \sum_{i=0}^k (-1)^i  \sum_{ v \in\Sigma(0),  v\sigma_i \in\Sigma(k) }\frac{m(v\sigma_i)}{m(\sigma_i)}\phi(v\sigma_i) .$$

\end{enumerate}
\end{proof}

Note that by the above proposition, we have for every $\phi \in  C^{0} (X, \mathbb{R} )$ that
$$\delta_{-1} \phi (\emptyset ) = \sum_{v \in \Sigma (0)} \dfrac{m(v)}{m (\emptyset )} \phi (v),$$
and
$$\forall u \in \Sigma (0), \Delta_0^- \phi (u) = \delta_0 \phi (\emptyset).$$ 

\begin{proposition}
\label{Delta0Norm}
For every $\phi \in  C^{0} (X, \mathbb{R} )$, $\left\langle \Delta_0^- \phi , \phi \right\rangle = \Vert \delta_{-1} \phi \Vert^2 =  \Vert \Delta_0^- \phi \Vert^2$.
\end{proposition}

\begin{proof}
For every $\phi \in  C^{0} (X, \mathbb{R} )$ we have
\begin{dmath*}
\left\langle \Delta_0^- \phi , \phi \right\rangle = \sum_{u \in \Sigma (0)} m(u) \left( \sum_{v \in \Sigma (0)} \dfrac{m(v)}{m (\emptyset )} \phi (v) \right) \phi (u)  =  \left( \sum_{v \in \Sigma (0)} \dfrac{m(v)}{m (\emptyset )} \phi (v) \right) \sum_{u \in \Sigma (0)} m(u) \phi (u) = m( \emptyset ) \left( \sum_{v \in \Sigma (0)} \dfrac{m(v)}{m (\emptyset )} \phi (v) \right)^2 = \Vert \delta_{-1} \phi \Vert^2.
\end{dmath*}
Also note that 
$$ \Vert \delta_{-1} \phi \Vert^2 = m( \emptyset ) \left( \sum_{v \in \Sigma (0)} \dfrac{m(v)}{m (\emptyset )} \phi (v) \right)^2  = \sum_{u \in \Sigma (0)} m(u)  \left( \sum_{v \in \Sigma (0)} \dfrac{m(v)}{m (\emptyset )} \phi (v) \right)^2 = \Vert \Delta_0^- \phi \Vert^2.$$
\end{proof}

\begin{proposition}
For every $\phi \in  C^{0} (X, \mathbb{R} )$, $\Delta_0^- \phi$ is the projection of $\phi$ on the space of constant $0$-forms.
\end{proposition}

\begin{proof}
Let $\textbf{1} \in C^0 (X, \mathbb{R})$ be the constant function $\textbf{1} (u) = 0, \forall u \in \Sigma (0)$. Then the projection of $\phi$ on the space of constant $0$-forms is given by $\frac{\left\langle \phi , \textbf{1} \right\rangle}{\Vert \textbf{1} \Vert^2} \textbf{1}$. Note that
$$ \Vert \textbf{1} \Vert^2 = \sum_{v \in \Sigma (0)} m(v) = m (\emptyset ),$$
$$ \left\langle \phi , \textbf{1} \right\rangle = \sum_{v \in \Sigma (0)} m(v) \phi (v).$$
Therefore for every $u \in \Sigma (0)$,
$$ \dfrac{\left\langle \phi , \textbf{1} \right\rangle}{\Vert \textbf{1} \Vert^2} \textbf{1} (u) = \sum_{v \in \Sigma (0)} \dfrac{m(v)}{m (\emptyset )} \phi (v) = \Delta_0^- (u).$$
\end{proof}

Define the reduced cohomology $\widetilde{H}^k (X, \mathbb{R})$ as $\widetilde{H}^k (X, \mathbb{R}) = H^k (X, \mathbb{R})$ for $k>0$ and $\widetilde{H}^k (X, \mathbb{R}) = ker (\Delta_0^-) \cap ker (\Delta_0^+)$.

\begin{remark}
Note that since $m$ is positive on every simplex, we get that if $X$ is connected, then $ker (\Delta^+_0)$ is the space of constant functions and $\widetilde{H}^k (X, \mathbb{R}) = 0$. 
Also note that in the case that $X$ is a graph (i.e., a $1$-dimensional simplicial complex) and $m$ is homogeneous, then $\Delta^+_0$ is just the usual graph Laplacian.   
\end{remark}

\begin{remark}
We remark that for $\Delta_0^+$ one always have $\Vert \Delta_0^+ \Vert \leq 2$, where $\Vert . \Vert$ here denotes the operator norm (we leave this calculation to the reader). We also remark that the largest eigenvalue of $\Delta_0^+$ is always larger than $1$. This can be seen easily when thinking about  $\Delta_0^+$ in matrix form: it is a matrix with $1$ along the diagonal and $0$ as an eigenvalue. Since the trace of $\Delta_0^+$ as a matrix is equal to the sum of eigenvalues, we can deduce it must have at least one eigenvalue strictly larger than $1$.  
\end{remark}

From now on, when there is no chance of confusion, we will omit the index of $d_k, \delta_k, \Delta^+_k, \Delta^-_k, \Delta_k $ and just refer to them as $d, \delta, \Delta^+, \Delta^-, \Delta $.

\subsection{partite simplicial complexes}

In important source of examples of simplicial complexes which act like bipartite expander graphs comes from $(n+1)$-partite simplicial complexes:
\begin{definition}
An $n$-dimensional simplicial complex $X$ will be called $(n+1)$-partite, if there is a disjoint partition $X^{(0)} = S_0 \cup ... \cup  S_n$ such that for every $u, v \in X^{(0)}$, 
$$\lbrace u ,v \rbrace \in X^{(1)} \Rightarrow \exists 0 \leq i, j \leq n, i \neq j, u \in S_i, v \in S_j.$$
If $X$ is pure $n$-dimensional, the above condition is equivalent to the following condition:
$$ \lbrace u_0,...,u_n \rbrace \in X^{(n)} \Rightarrow \exists \pi \in Sym (\lbrace 0,...,n \rbrace ), \forall 0 \leq i \leq n, u_i \in  S_{\pi (i)} .$$
We shall call $S_0,...,S_n$ the sides of $X$.
\end{definition}

Let $X$ be a pure $n$-dimensional, weighted, $(n+1)$-partite simplicial complex with sides $S_0,...,   S_n$ as in the above definition. We shall define the following operators: \\
For $0 \leq j \leq n$ and $-1 \leq k \leq n-1$, define
$$d_{(k,j)} : C^k (X, \mathbb{R}) \rightarrow   C^{k+1} (X, \mathbb{R}) ,$$
as following:
$$d_{(k,j)} \phi ((v_0,...,v_{k+1})) = \begin{cases}
0 &  v_0 \notin S_j,...,v_{k+1} \notin S_j \\
(-1)^i \phi ((v_0,...,\widehat{v_i},...,v_{k+1})) & v_i \in S_j
\end{cases} .$$
Denote by $\delta_{(k,j)} :  C^{k+1} (X, \mathbb{R}) \rightarrow   C^{k} (X, \mathbb{R})$ the adjoint operator to $d_{(k,j)}$ and $\Delta^{-}_{(k,j)} = d_{(k-1,j)} \delta_{(k-1,j)}$.

\begin{proposition}
Let $-1 \leq k \leq n, 0 \leq j \leq n$, then for every $\phi \in C^{k+1} (X, \mathbb{R})$ 
$$\delta_{(k,j)} \phi (\tau) = \sum_{v \in S_j, v \tau \in \Sigma (k+1)} \dfrac{m(v \tau)}{m(\tau)} \phi (v \tau).$$
\end{proposition}

\begin{proof}
Let $\phi \in C^{k+1} (X, \mathbb{R}) , \psi \in C^{k} $, then 
\begin{dmath*}
\left\langle d_{(k,j)} \psi, \phi \right\rangle = \sum_{\sigma \in \Sigma (k+1)} \dfrac{m(\sigma )}{(k+2)!} d_{(k,j)} \psi (\sigma) \phi (\sigma) = \sum_{\sigma = (v_0,...,v_{k+1}) \in \Sigma (k+1), v_i \in S_j} \dfrac{m(\sigma )}{(k+2)!} (-1)^i \psi (\sigma_i) \phi (\sigma) = \sum_{\sigma = (v_0,...,v_{k+1}) \in \Sigma (k+1), v_i \in S_j} \dfrac{m(v_i \sigma_i )}{(k+2)!}  \psi (\sigma_i) \phi (v_i \sigma_i) =
\sum_{\tau \in \Sigma (k)} \sum_{v \in S_j} \dfrac{m(v \tau )}{(k+1)!}  \psi (\tau) \phi (v \tau) = 
\sum_{\tau \in \Sigma (k)} \dfrac{m( \tau )}{(k+1)!}  \psi (\tau) \left( \sum_{v \in S_j}  \dfrac{m(v \tau )}{m (\tau)} \phi (v \tau) \right) .
\end{dmath*}
\end{proof}

A straightforward computation gives raise to:
\begin{corollary}
\label{side average operators}
For every $0 \leq k \leq n$, $0 \leq j \leq n$ and every $\phi \in C^{k} (X,\mathbb{R})$ we have that
$$\Delta^{-}_{(k,j)} \phi (\sigma) =
\begin{cases}
0 & \sigma = (v_0,...,v_k), \forall i, v_i \notin S_j \\
(-1)^i \sum_{v \in S_j, v \sigma_i \in \Sigma (k)} \dfrac{m(v \sigma_i)}{m(\sigma_i)} \phi (v \sigma_i) & \sigma = (v_0,...,v_k), v_i \in S_j 
\end{cases}.$$
\end{corollary}


\section{Links of $X$}

Let $\lbrace v_{0},...,v_{j} \rbrace=\tau\in X^{(j)}$, denote by $X_{\tau}$
the \emph{link} of $\tau$ in $X$, that is, the (pure) complex of dimension
$n-j-1$ consisting on simplices $\sigma=\lbrace  w_{0},...,w_{k} \rbrace$
such that $\lbrace v_{0},...,v_{j} \rbrace, \lbrace w_{0},...,w_{k} \rbrace$ are disjoint as sets and $\lbrace v_{0},...,v_{j} \rbrace \cup \lbrace w_{0},...,w_{k} \rbrace \in X^{(j+k+1)}$. Note that for $\emptyset \in \Sigma (-1)$, $X_\emptyset =X$. \\
For an ordered simplex $( v_{0},...,v_{j} ) \in \Sigma (k)$ define $X_{( v_{0},...,v_{j} ) } = X_{\lbrace v_{0},...,v_{j} \rbrace }$. \\
Throughout this article we shall assume that all the links of $X$ of dimension $>0$ are connected . \\
Next, we'll basically repeat the definitions that we gave above for $X$: \\  
For $0\leq k\leq n-j-1$, denote by $\Sigma_{\tau}(k)$ the set of
ordered $k$-simplices in $X_\tau$. \\
Define the function $m_\tau : \bigcup_{0 \leq k \leq n-j-1} \Sigma_\tau (k) \rightarrow \mathbb{R}^+$ as
$$\forall \sigma \in \Sigma_\tau (k), m_{\tau}(\sigma) = m (\tau \sigma ),$$
where $\tau \sigma$ is the concatenation of $\tau$ and $\sigma$, i.e., if $\tau = (v_{0},...,v_{j} ), \sigma = (w_{0},...,w_{k} )$ then  $ \tau \sigma = (v_{0},...,v_{j},w_{0},...,w_{k} )$.
\begin{proposition}
The function $m_\tau$ defined above follows the conditions stated in proposition \ref{weightedOrderDef}, i.e., $m_\tau$ is a weight function of $X_\tau$.
\end{proposition}

\begin{proof}
The fact that $m_\tau$ is invariant under permutation is obvious, therefore we are left to check that for every $\eta \in \Sigma_\tau (k)$ we have
$$ \sum_{\sigma \in \Sigma_\tau (k+1), \eta \subset \sigma} m_\tau ( \sigma ) = (k+2)!  m_\tau (\eta ).$$
For $\eta \in \Sigma_\tau (k)$ we have by definition
\begin{dmath*}
 \sum_{\sigma \in \Sigma_\tau (k+1), \eta \subset \sigma } m_\tau (\sigma)= \sum_{ \sigma \in \Sigma_\tau (k+1), \eta \subset \sigma } m (\tau \sigma) = \sum_{ \gamma \in \Sigma (j+k+2), \tau \eta \subset \gamma } \dfrac{ (k+2)!}{(j+k+3)!} m ( \gamma ) =
= { (k+2)! m (\tau \eta ) =  (k+2)!  m_\tau ( \eta ).}
\end{dmath*}

\end{proof}

We showed that $X_\tau$ is a weighted simplicial complex with the weight function $m_\tau$ and therefore we can repeat all the definitions given before for $X$. Therefore we have:
\begin{enumerate}
\item $C^{k}(X_{\tau},\mathbb{R})$ with the inner product on it.
\item Differential $d_{\tau,k} : C^{k}(X_{\tau},\mathbb{R}) \rightarrow C^{k+1}(X_{\tau},\mathbb{R})$, $\delta_{\tau,k} = (d_{\tau,k})^*. \delta_{\tau,0}$.
\item $\Delta_{\tau,k}^+ = \delta_{\tau,k} d_{\tau,k},  \Delta_{\tau,k}^- = d_{\tau ,k-1} \delta_{\tau,k-1}, \Delta_{\tau,k} = \Delta_{\tau,k}^+ + \Delta_{\tau,k}^-$.
\end{enumerate} 

From now on, when there is no chance of confusion, we will omit the index of $d_{\tau,k}, \delta_{\tau,k}, \Delta^+_{\tau,k}, \Delta^-_{\tau,k}, \Delta_{\tau,k} $ and just refer to them as $d_\tau, \delta_\tau, \Delta^+_\tau, \Delta^-_\tau, \Delta_\tau $.
\begin{remark}
Notice that for an $n$-dimensional simplicial complex $X$, if $m$ is homogeneous, then for every $\tau \in \Sigma (n-2)$, $X_\tau$ is a graph such that $m_\tau$ assigns the value $1$ for every edge. In this case, $\Delta_{\tau,0}^+$ is the usual graph Laplacian.
\end{remark}
We now turn to describe how maps $C^{k}(X,\mathbb{R})$ induce maps on the links of $X$. This is done in two different ways described below: localization and restriction. 

\subsection{Localization}
\begin{definition}
For $\tau \in \Sigma (j)$ and $j-1 \leq k \leq n$ define the \emph{localization map} \[
C^{k}(X,\mathbb{R})\rightarrow C^{k-j-1}(X_{\tau},\mathbb{R}),\;\phi\rightarrow\phi_{\tau} ,\]
where $\phi_\tau$ is defined by $\phi_{\tau}(\sigma)=\phi(\tau\sigma)$. 
\end{definition}
When $\phi \in C^k (X, \mathbb{R} ), k>0$, one can compute $\Vert \phi \Vert^2, \Vert \delta \phi \Vert^2, \Vert d \phi \Vert^2$ by using all the localizations of the form $\phi_\tau, \tau \in \Sigma (k-1)$. This is described in the following lemmas:

\begin{lemma}
\label{LocalizNorm1}
For every $0 \leq k \leq n$ and every $\phi, \psi \in C^k (X, \mathbb{R} )$, one has:
\begin{enumerate}
\item $$(k+1)! \left\langle \phi , \psi \right\rangle= \sum_{\tau \in \Sigma(k-1)} \left\langle \phi_\tau , \psi_\tau \right\rangle .$$
\item For $\tau \in \Sigma (k-1)$, 
$$k! \left\langle \delta \phi , \delta \psi \right\rangle= \sum_{\tau \in \Sigma(k-1)} \left\langle \delta_\tau \phi_\tau , \delta_\tau \psi_\tau \right\rangle .$$

\end{enumerate}
\end{lemma}

\begin{proof}
\begin{enumerate}
\item 
\begin{dmath*}
\sum_{\tau \in \Sigma(k-1)} \left\langle \phi_\tau , \psi_\tau \right\rangle = \sum_{\tau \in \Sigma(k-1)} \sum_{u \in \Sigma_\tau (0)}  m_\tau (u) \phi_\tau (u)\psi_\tau (u)
 = \\
 \sum_{\tau \in \Sigma(k-1)} \sum_{u \in \Sigma_\tau (0)}  m (\tau u) \phi (\tau u) \psi (\tau u) = \sum_{\tau \in \Sigma(k-1)} \sum_{\sigma \in \Sigma (k), \tau \subset \sigma} \dfrac{m (\sigma) }{(k+1)!}  \phi (\sigma) \psi (\sigma) 
 = \\
 \sum_{\sigma \in \Sigma(k)} \dfrac{m (\sigma)}{(k+1)!}   \phi (\sigma) \psi (\sigma) \sum_{\tau \in \Sigma (k-1), \tau \subset \sigma} 1  = (k+1)!  \left\langle \phi , \psi \right\rangle.
 \end{dmath*}
\item For every $\phi, \psi \in C^k (X, \mathbb{R} )$ and every $\tau \in \Sigma (k-1)$,
\begin{dmath*}
\left\langle \delta_\tau \phi_\tau , \delta_\tau \psi_\tau \right\rangle = m_\tau (\emptyset ) \left( \dfrac{1}{m_\tau (\emptyset )} \sum_{v \in \Sigma_\tau (0)} m_\tau (v) \phi_\tau (v) \right) \left( \dfrac{1}{m_\tau (\emptyset )} \sum_{v \in \Sigma_\tau (0)} m_\tau (v) \psi_\tau (v) \right) = 
 m( \tau ) \left ( \dfrac{1}{m (\tau)} \sum_{v \in \Sigma (0), \tau v \in \Sigma (k)} m (\tau v) \phi (\tau v) \right)\left ( \dfrac{1}{m (\tau)} \sum_{v \in \Sigma (0), \tau v \in \Sigma (k)} m (\tau v) \psi (\tau v) \right) = m(\tau)  ((-1)^k \delta \phi (\tau)) ((-1)^k \delta \psi (\tau)) =m(\tau)  ( \delta \phi (\tau))( \delta \psi (\tau)) ,
 \end{dmath*}
and the equality in the lemma follows.

\end{enumerate}
 
\end{proof}

\begin{lemma}
\label{LocalizNorm2}
For every $0 \leq k \leq n-1$ and every $\phi, \psi \in C^k (X, \mathbb{R} )$, one has:
 $$  k! \left\langle d \phi , d \psi \right\rangle =  \sum_{\tau \in \Sigma (k-1)} \left( \left\langle d_\tau \phi_\tau , d_\tau \psi_\tau \right\rangle -  \dfrac{ k}{k+1} \left\langle \phi_\tau , \psi_\tau \right\rangle \right).$$
 \end{lemma}
 
\begin{proof}
For $k=0$, there is only $\emptyset \in \Sigma (-1)$ and $\phi_\emptyset = \phi, \psi_\emptyset = \psi$, $d_\emptyset = d$ and the lemma is trivial. Assume that $1 \leq k \leq n-1$.  For $(v_0,...,v_{k+1}) = \sigma \in \Sigma (k+1)$ and $ 0 \leq i < j \leq k+1$ denote 
$$\sigma_{ij} = (v_0,...\hat{v_{i}},...,\hat{v_{j}},...,v_{k+1}).$$ 
By this notation we can write:
\begin{dmath*}
 ( d \phi (\sigma) ) ( d \psi (\sigma) )  = \sum_{0 \leq i < j \leq k+1} ( \phi_{\sigma_{ij}} (v_i) - \phi_{\sigma_{ij}} (v_j) )  ( \psi_{\sigma_{ij}} (v_i) - \psi_{\sigma_{ij}} (v_j) )   - k \sum_{0 \leq i \leq k+1}  \phi (\sigma_i) \psi (\sigma_i)  = \sum_{0 \leq i < j \leq k+1} \left( ( \phi_{\sigma_{ij}} (v_i) - \phi_{\sigma_{ij}} (v_j)) ( \psi_{\sigma_{ij}} (v_i) - \psi_{\sigma_{ij}} (v_j) ) - \dfrac{k}{k+1} (  \phi_{ \sigma_{ij}} (v_i)\psi_{ \sigma_{ij}} (v_i)+  \phi_{ \sigma_{ij}} (v_j)\psi_{ \sigma_{ij}} (v_j) ) \right) =
  \dfrac{1}{k!} \sum_{\tau \in \Sigma (k-1), \tau \subset \sigma} \left( ( d_\tau \phi_{\tau} (\sigma - \tau) )( d_\tau \psi_{\tau} (\sigma - \tau) ) - \dfrac{k}{k+1} \sum_{v \in \Sigma_\tau (0), v \subset \sigma - \tau} \phi_{\tau} (v)\psi_{\tau} (v)  \right),
  \end{dmath*}
where $\sigma - \tau$ is the $1$-dimensional simplex obtained by deleting the the vertices of $\tau$ from $\sigma$. \\
We can use this equality to connect $\langle  d \phi, d \psi \rangle$ to $\langle d_\tau \phi_\tau,d_\tau \psi_\tau  \rangle$ and $\langle \phi_\tau, \psi_\tau  \rangle$:
\begin{dmath*}
 k! \langle d \phi, d \psi \rangle = \sum_{\sigma \in \Sigma (k+1)} \dfrac{m( \sigma )}{(k+2)! } k! d \phi (\sigma) d \psi (\sigma)  = 
 \sum_{\sigma \in \Sigma (k+1)} \dfrac{1}{(k+2)!} \sum_{\tau \in \Sigma (k-1), \tau \subset \sigma} m_\tau ( \sigma - \tau ) \left( ( d_\tau \phi_{\tau} (\sigma - \tau) ) ( d_\tau \psi_{\tau} (\sigma - \tau) )  - \dfrac{k}{k+1} \sum_{v \in \Sigma_\tau (0), v \subset \sigma - \tau} \phi_{\tau} (v)\psi_{\tau} (v) \right) = 
 \sum_{\tau \in \Sigma (k-1)} \dfrac{1}{(k+2)! } \sum_{\sigma \in \Sigma (k+1), \tau \subset \sigma} m_\tau ( \sigma - \tau ) \left( ( d_\tau \phi_{\tau} (\sigma - \tau) )( d_\tau \psi_{\tau} (\sigma - \tau) ) - \dfrac{k}{k+1} \sum_{v \in \Sigma_\tau (0), v \subset \sigma - \tau}  \phi_{\tau} (v)\psi_{\tau} (v) \right) = 
\sum_{\tau \in \Sigma (k-1)} \dfrac{1}{(k+2)! } \dfrac{(k+2)!}{2} \sum_{\eta \in \Sigma_\tau (1)} m_\tau ( \eta ) \left( ( d_\tau \phi_{\tau} (\eta) )( d_\tau \psi_{\tau} (\eta) ) - \dfrac{k}{k+1} \sum_{v \in \Sigma_\tau (0), v \subset \eta}  \phi_{\tau} (v)\psi_{\tau} (v) \right)=
  {\sum_{\tau \in \Sigma (k-1)}   \sum_{\eta \in \Sigma_\tau (1)} \dfrac{m_\tau ( \eta ) }{2} ( d_\tau \phi_{\tau} (\eta) )( d_\tau \psi_{\tau} (\eta) ) - \dfrac{k}{k+1} \sum_{\tau \in \Sigma (k-1)}  \sum_{\eta \in \Sigma_\tau (1)}  \sum_{v \in \Sigma_\tau (0), v \subset \eta} \dfrac{m_\tau (\eta )}{2} \phi_{\tau} (v)\psi_{\tau} (v).}
\end{dmath*}
Note that 
$$ \sum_{\tau \in \Sigma (k-1)}  \sum_{\eta \in \Sigma_\tau (1)} \dfrac{m_\tau ( \eta ) }{2} ( d_\tau \phi_{\tau} (\eta) )( d_\tau \psi_{\tau} (\eta) ) =  \sum_{\tau \in \Sigma (k-1)} \langle d_\tau \phi_\tau, d_\tau \psi_\tau \rangle.$$
Also note that
\begin{align*}
\dfrac{k}{k+1} \sum_{\tau \in \Sigma (k-1)} \sum_{\eta \in \Sigma_\tau (1)}  \sum_{v \in \Sigma_\tau (0), v \subset \eta} \dfrac{m_\tau (\eta )}{2} \phi_{\tau} (v)\psi_{\tau} (v) = \\
\dfrac{k}{k+1} \sum_{\tau \in \Sigma (k-1)}  \sum_{v \in \Sigma_\tau (0)}  \phi_{\tau} (v)\psi_{\tau} (v)  \sum_{\eta \in \Sigma_\tau (1), v \subset \eta} \dfrac{m_\tau (\eta )}{2} =\\
 \dfrac{k}{k+1} \sum_{\tau \in \Sigma (k-1)} \sum_{v \in \Sigma_\tau (0)} \phi_{\tau} (v)\psi_{\tau} (v) m_\tau (v)  =  \sum_{\tau \in \Sigma (k-1)}  \dfrac{k}{k+1}  \langle \phi_\tau, \psi_\tau \rangle  .
 \end{align*}
Therefore we get the desired equality.
 \end{proof}
 
\begin{corollary}
\label{LocalizNorm3}
For every $1 \leq k \leq n$ and every $\phi, \psi \in C^k (X, \mathbb{R} )$, one has:
 $$  k! \langle d \phi , d \psi \rangle + k!k \langle \phi , \psi \rangle =  \sum_{\tau \in \Sigma (k-1)} \langle d_\tau \phi_\tau, d_\tau \psi_\tau \rangle .$$
In particular, for $\phi = \psi$, one has:
 $$  k! \Vert d \phi \Vert^2 + k!k \Vert \phi \Vert^2 =  \sum_{\tau \in \Sigma (k-1)} \Vert d_\tau \phi_\tau \Vert^2 .$$
\end{corollary} 

\begin{proof}
Combine the equality of lemma \ref{LocalizNorm1} 1. with the equality of lemma \ref{LocalizNorm2}.
\end{proof}
 
Next, we'll discuss localization for $(n+1)$-partite complexes. Let $X$ be a pure $n$-dimensional, weighted, $(n+1)$-partite simplicial complex with sides $S_0,...,S_n$. Notice that for any $-1 \leq k \leq n-1$, $X_\tau$ is a $(n-k)$-partite complex. In order to keep the indexing of the sides consistent, we shall denote as follows: for $\tau = (v_0,...,v_k), v_i \in S_{j_i} $, the sides of $X_\tau$ will be denoted by $S_{\tau,j}$, where $j \neq j_0,...,j_k$ and $S_{\tau, j} \subseteq S_{j}$.\\
This will allow us to define $d_{\tau, (l,j)}, \delta_{\tau, (l,j)}$ on $X_\tau$ for $-1 \leq l \leq n-k-1$ in the following way: if $\tau = (v_0,...,v_k), v_i \in S_{j_i} $, then for $j \neq  j_0,...,j_k$, define $d_{\tau, (l,j)}, \delta_{\tau, (l,j)}$ as above (using the indexing on $X_\tau$). If $j = j_i$ for some $0 \leq i \leq k$, then define $d_{\tau, (l,j)} \equiv 0, \delta_{\tau, (l,j)} \equiv 0$. Denote $\Delta^-_{\tau, (l,j)} = d_{\tau, (l-1,j)} \delta_{\tau, (l-1,j)}$.  \\
After setting these conventions, we can show the following: 
\begin{proposition}
\label{partite localization of Delta^-}
Let $X$ be a pure $n$-dimensional, weighted, $(n+1)$-partite simplicial complex. Then for every $\phi \in  C^{k} (X,\mathbb{R} )$ and every $0 \leq j \leq n$, we have that
$$k! \langle \Delta^-_{(k,j)} \phi, \phi \rangle = \sum_{\tau \in \Sigma (k-1)} \langle \Delta^-_{\tau, (0,j)} \phi_\tau, \phi_\tau \rangle .$$ 
\end{proposition} 
 
\begin{proof}
Let $\phi \in  C^{k} (X,\mathbb{R} )$, then by definition
\begin{dmath*}
k! \langle \Delta^-_{(k,j)} \phi, \phi \rangle = k! \langle \delta^-_{(k,j)} \phi, \delta^-_{(k,j)} \phi \rangle = \sum_{\tau \in \Sigma (k-1)} m (\tau) \left( \sum_{v \in S_j, v \tau \in \Sigma (k+1)} \dfrac{m(v \tau)}{m(\tau)} \phi (v \tau) \right)^2 = \sum_{\tau \in \Sigma (k-1)} m_\tau (\emptyset) \left( \sum_{v \in S_j, v  \in \Sigma_\tau (0)} \dfrac{m_{\tau} (v)}{m_{\tau} (\emptyset)} \phi_\tau (v) \right)^2 = \sum_{\tau \in \Sigma (k-1)} \Vert \delta_{\tau, (0,j)} \phi_\tau \Vert^2 = \sum_{\tau \in \Sigma (k-1)} \langle \Delta^-_{\tau, (0,j)} \phi_\tau, \phi_\tau \rangle .
\end{dmath*}
\end{proof} 
 
 \subsection{Restriction}

\begin{definition}
For $\phi \in C^k (X, \mathbb{R})$ and $\tau \in \Sigma (l)$ s.t. $k+l+1 \leq n$, the restriction of $\phi$ to $X_\tau$  is a function $\phi^\tau \in C^k (X_\tau, \rho_\tau)$ defined as follows: 
$$ \forall \sigma \in \Sigma_\tau (k), \phi^\tau (\sigma) = \phi (\sigma).$$
\end{definition}

For $\phi, \psi \in C^k (X, \mathbb{R} )$, one can compute $\langle \phi , \psi \rangle$ using all the localizations of the form $\phi^\tau, \psi^\tau$. This is described in the following lemma:

\begin{lemma}
\label{restNorm1}
For every $0 \leq k \leq n-1$ let $\phi, \psi \in C^k (X, \mathbb{R} )$ and $0 \leq l \leq n-k-1$. Then
$$  \langle \phi , \psi  \rangle = \sum_{\tau \in \Sigma (l)} \langle \phi^\tau, \psi^\tau \rangle.$$
\end{lemma}

\begin{proof}
\begin{align*}
 \sum_{\tau \in \Sigma (l)} \langle \phi^\tau, \psi^\tau \rangle = \sum_{\tau \in \Sigma (l)} \sum_{\sigma \in \Sigma_\tau (k)} \dfrac{m_\tau (\sigma)}{(k+1)!}  \phi^\tau (\sigma)\psi^\tau (\sigma) = \\
 \sum_{\tau \in \Sigma (l)} \dfrac{1}{(k+1)!  }\sum_{\sigma \in \Sigma_\tau (k)} m (\tau \sigma) \phi (\sigma)\psi (\sigma) = \\
  \sum_{\tau \in \Sigma (l)} \dfrac{1}{(k+1)! }\sum_{\gamma \in \Sigma (l+k+1),\tau \subset \gamma} \dfrac{(k+1)!}{(l+k+2)! } m (\gamma) \phi (\gamma - \tau)\psi (\gamma - \tau),
\end{align*}
where $\gamma - \tau$ means deleting the vertices of $\tau$ from $\gamma$. Changing the order of summation gives
\begin{align*}
\sum_{\gamma \in \Sigma (l+k+1)} \dfrac{m(\gamma)}{(l+k+2)!  }\sum_{\tau \in \Sigma (l),\tau \subset \gamma}  \phi (\gamma - \tau) \psi (\gamma - \tau)  = \\
\sum_{\gamma \in \Sigma (l+k+1)} \dfrac{m(\gamma)}{(l+k+2)!  }\sum_{\sigma \in \Sigma (k),\sigma \subset \gamma} \dfrac{(l+1)!}{(k+1)!} \phi (\sigma)\psi (\sigma) = \\
 \sum_{\sigma \in \Sigma (k)} \dfrac{(l+1)! \phi (\sigma)\psi (\sigma)}{(l+k+2)!(k+1)! } \sum_{\gamma \in \Sigma (l+k+1),\sigma \subset \gamma} m(\gamma) 
 \end{align*}
 Recall that by corollary \ref{weight in l dim simplices} we have that 
 $$\sum_{\gamma \in \Sigma (l+k+1),\sigma \subset \gamma} m(\gamma)  = \dfrac{(l+k+2)!}{(l+1)!}  m (\sigma ) .$$
 Therefore we get
\begin{dmath*}
\sum_{\sigma \in \Sigma (k)} \dfrac{(l+1)! \phi (\sigma) \psi (\sigma)}{(l+k+2)!(k+1)! } \sum_{\gamma \in \Sigma (l+k+1),\sigma \subset \gamma} m(\gamma)  = 
 \sum_{\sigma \in \Sigma (k)} \dfrac{ m (\sigma ) }{(k+1)! } \phi (\sigma) \psi (\sigma) = \langle \phi, \psi \rangle .
\end{dmath*}
\end{proof}

\begin{lemma}
\label{restNorm2}
Assume that $X$ is of dimension $>1$. Let $\phi, \psi \in C^0 (X,\mathbb{R})$ and $0 \leq l \leq n-1$, then
$$  \langle d \phi, d \psi  \rangle = \sum_{\tau \in \Sigma (l)} \langle d_\tau \phi^\tau , d_\tau \psi^\tau \rangle,$$
where $d_\tau$ is the restriction of $d$ to the link of $\tau$.
\end{lemma}

\begin{proof}
Note that 
$$ \forall (v_0,v_1) \in \Sigma_\tau (1), d_\tau \phi^\tau ((v_0,v_1)) = \phi (v_0)-\phi (v_1) = d \phi ((v_0,v_1)) = (d \phi)^\tau ((v_0,v_1)),  $$
and similarly
$$ \forall (v_0,v_1) \in \Sigma_\tau (1), d_\tau \psi^\tau ((v_0,v_1)) = (d \phi)^\tau ((v_0,v_1)) .$$
Therefore $d_\tau (\phi^\tau) = (d \phi)^\tau, d_\tau (\psi^\tau) = (d \psi)^\tau$  and the lemma follows from the previous one. 
\end{proof}

\subsection{Connectivity of links}

Throughout this paper, we'll assume $X$ and all its links of dimension $>0$ are connected . We show that this implies that $X$ has strong connectivity properties, namely we shall show that $X$ is gallery connected (see definition below).

\begin{definition}
A pure $n$-dimensional simplicial complex is called gallery connected, if for every two vertices $u ,v \in X^{(0)}$ there is a sequence of simplexes $\sigma_0 ,...,\sigma_l \in X^{(n)}$ such that $u \in \sigma_0, v \in \sigma_l$ and for every $0 \leq i \leq l-1$, we have that $\sigma_i \cap \sigma_{i+1} \in X^{(n-1)}$.
\end{definition}

\begin{proposition}
\label{gallery connectedness}
Let $X$ be  a connected pure $n$-dimensional simplicial complex. If all the links of $X$ of dimension $>0$ are connected, then $X$ is gallery connected. 
\end{proposition}

\begin{proof}
We shall prove the by induction on $n$. If $n =1$ then gallery connected is the same as connected and there is nothing to prove. Assume the proposition holds for $n-1$. Let  $X$ be  a connected pure $n$-dimensional simplicial complex such that all the links of $X$ are connected. Then the $(n-1)$-skeleton of $X$ is a pure $(n-1)$-simplicial complex with connected links. Therefore, by the induction assumption, for every $u,v \in  X^{(0)}$, there are $\tau_0,...,\tau_l \in X^{(n-1)}$  such that $u \in \tau_0, v \in \tau_l$ and for every $0 \leq i \leq l-1$, $\tau_i \cap \tau_{i+1} \in X^{(n-2)}$. $X$ is pure $n$-dimensional, therefore we can take $\sigma_i \in X^{(n)}$ such that for every $0 \leq i \leq l$, $\tau_i \subset \sigma_i$. If $l=0$ there is nothing to prove. Assume $l>0$,  to finish, we shall show that for every $0 \leq i \leq l-1$ there is a gallery connecting $\sigma_i$ and $\sigma_{i+1}$ (and therefore one can take a concatenation of those galleries). Fix $0 \leq i \leq l-1$. Denote $\eta = \tau_i \cap \tau_{i+1} \in X^{(n-2)}, v' = \tau_i \setminus \eta, v'' = \tau_{i+1} \setminus \eta$. By our assumptions $X_\eta$ is connected, therefore there are $v_1,...,v_k \in X_\eta^{(0)}$ such that 
$$\lbrace v' , v_1 \rbrace, \lbrace v_1,v_2 \rbrace,...,\lbrace v_k,v'' \rbrace \in X_\eta^{(1)} .$$
Denote 
$$\sigma_0' = \eta \cup \lbrace v', v_1 \rbrace, \sigma_1' = \eta \cup \lbrace v_1, v_2 \rbrace, ...,\sigma_k' =  \eta \cup \lbrace v_k, v'' \rbrace .$$
Note that $\sigma_0',...,\sigma_k' \in X^{(n)}$ and that 
$$\forall 0 \leq i \leq k-1, \sigma_i' \cap \sigma_{i+1}' = \eta \cup \lbrace v_i \rbrace \in X^{(n-1)}.$$
Also note that
$$\tau_i \subseteq \sigma_i \cap  \sigma_0',  \tau_{i+1} \subseteq \sigma_{i+1} \cap  \sigma_k'   .$$
Therefore there is a gallery connecting $\sigma_i$ and $\sigma_{i+1}$ and we are done.

\end{proof}

\section{Laplacian spectral gaps}

In this section we will show that a large spectral gap on the upper Laplacian on all the $1$ dimensional links induces spectral gaps in all the other Laplacians (in the Laplacians of all the other links and in the Laplacians of $X$). The exact formulation appears in theorem \ref{SpecGapThm}. The results of this section were already worked-out by the author in a more general setting in \cite{Opp}. We chose to include all the proofs and not just refer to \cite{Opp} in order to keep this paper self-contained.

\subsection{Descent in links}

We shall show that spectral gaps of the $0$ upper Laplacian "trickle down" through links of simplices of different dimension. Specifically, we shall show the following:
\begin{lemma}
\label{SpectralGapDescent1}
Let $X$ as before, i.e., a pure $n$-dimensional weighted simplicial complex such that all the links of $X$ of dimension $>0$ are connected. Also, assume that $n>1$. For $0 \leq k \leq n-2$, if there are $\kappa \geq \lambda >0$ such that
$$\bigcup_{\sigma \in \Sigma (k)} Spec (\Delta_{\sigma, 0}^+) \setminus \lbrace 0 \rbrace \subseteq [\lambda, \kappa],$$
then
$$\bigcup_{\tau \in \Sigma (k-1)} Spec (\Delta_{\tau, 0}^+) \setminus \lbrace 0 \rbrace  \subseteq \left[2 - \dfrac{1}{\lambda}, 2 - \dfrac{1}{\kappa} \right].$$ 
 
\end{lemma}

\begin{proof}

Fix some $\tau \in \Sigma (k-1)$. First note that
$$\bigcup_{v \in \Sigma_\tau (0)} Spec (\Delta_{\tau v, 0}^+) \setminus \lbrace 0 \rbrace \subseteq \bigcup_{\sigma \in \Sigma (k)} Spec (\Delta_{\sigma, 0}^+) \setminus \lbrace 0 \rbrace \subseteq [\lambda, \kappa].$$
For every $v \in \Sigma_\tau (0)$ and recall that $\Delta_{\tau v}^- \phi^v$ is the projection of $\phi^v$ to the space of constant maps on $X_{\tau v}$. Denote by $( \phi^v)^1$ the orthogonal compliment of that projection. \\
Since $X_{\tau v}$ is connected for every $v \in \Sigma_\tau (0)$, the kernel of $\Delta^+_{\tau v}$ is the space of constant maps. Therefore for every  $v \in \Sigma_\tau (0)$ we have that
$$\kappa \Vert (\phi^v)^1 \Vert^2 \geq \Vert d_{\tau v} \phi^v \Vert^2 \geq \lambda \Vert (\phi^v)^1 \Vert^2 .$$

Take $\phi \in  C^0 (X_\tau, \mathbb{R})$ to be a non constant eigenfunction of $\Delta^+_\tau$ with the eigenvalue $\mu >0$ (recall that $X_\tau$ is connected so the kernel of $\Delta^+_\tau$ is the space of constant functions) , i.e., 
$$\Delta_\tau^+ \phi (u) = \mu \phi (u).$$ 
By lemma \ref{restNorm2} we have
$$\mu \Vert \phi \Vert^2 =  \Vert d_\tau \phi \Vert^2 = \sum_{v \in \Sigma_\tau (0)} \Vert d_{\tau v} \phi^v \Vert^2.$$
Combined with the above inequalities this yields:
\begin{equation}
\label{ineq}  
\kappa \sum_{v \in \Sigma_\tau (0)}  \Vert (\phi^v)^1 \Vert^2 \geq  \mu \Vert \phi \Vert^2 \geq \lambda \sum_{v \in \Sigma_\tau (0)}  \Vert (\phi^v)^1 \Vert^2.\end{equation}
Next, we shall compute $\sum_{v \in \Sigma_\tau (0)}  \Vert (\phi^v)^1 \Vert^2 $. Note that
$$\Vert (\phi^v)^1 \Vert^2 = \Vert (\phi^v) \Vert^2 - \Vert \Delta_{\tau v}^- \phi^v \Vert^2.$$  
By lemma \ref{restNorm1} we have that 
$$\sum_{v \in \Sigma_\tau (0)}  \Vert (\phi^v) \Vert^2 = \Vert \phi \Vert^2,$$
and therefore we need only to compute $\sum_{v \in \Sigma_\tau (0)}  \Vert \Delta_{\tau v}^- \phi^v \Vert^2 $. First, let us write $\Delta_{\tau v}^- \phi^v$ explicitly: 
$$\Delta_{\tau v}^- \phi^v \equiv \dfrac{1}{m_{\tau v} (\emptyset )} \sum_{u \in \Sigma_{\tau v} (0) } m_{\tau v} (u) \phi^v (u) = \dfrac{1}{m_{\tau} (v )} \sum_{(v,u)  \in \Sigma_{\tau} (1) } m_{\tau} ((v,u)) \phi (u).$$
Notice that since $\Delta_\tau^+ \phi = \mu \phi$, we get 
$$\mu \phi (v) = \Delta_\tau^+ \phi (v) = \phi (v) - \dfrac{1}{m_\tau (v)} \sum_{(v,u) \in \Sigma_\tau (1)} m_\tau((v,u)) \phi (u) = \phi (v)-  \Delta_{\tau v}^- \phi^v.$$
Therefore
$$\Delta_{\tau v}^- \phi^v = (1 - \mu ) \phi (v).$$
This yields
$$\sum_{v \in \Sigma_\tau (0)}  \Vert \Delta_{\tau v}^- \phi^v \Vert^2 = \sum_{v \in \Sigma_\tau (0)} m_\tau (v) (1- \mu)^2 \phi (v)^2 = (1- \mu)^2 \Vert \phi \Vert^2.$$ 
Therefore
$$\sum_{v \in \Sigma_\tau (0)}  \Vert (\phi^v)^1 \Vert^2 = \sum_{v \in \Sigma_\tau (0)} \Vert (\phi^v) \Vert^2- \sum_{v \in \Sigma_\tau (0)}  \Vert \Delta_{\tau v}^- \phi^v \Vert^2 = \Vert \phi \Vert^2 (1 - (1- \mu)^2) =\Vert \phi \Vert^2 \mu (2-\mu) .$$
Combine with the inequality in \eqref{ineq} to get
$$ \kappa  \Vert \phi \Vert^2 \mu (2-\mu) \geq  \mu \Vert \phi \Vert^2  \geq \lambda \Vert \phi \Vert^2 \mu (2-\mu) .$$
Dividing by $\Vert \phi \Vert^2 \mu$ yields
$$ \kappa   (2-\mu) \geq  1    \geq \lambda  (2-\mu) .$$
And this in turns yields 
$$ 2- \dfrac{1}{\kappa} \geq \mu \geq 2- \dfrac{1}{\lambda}.$$
Since $\mu$ was any positive eigenvalue of $\Delta_{\tau,0}^+$ we get that
$$Spec (\Delta_{\tau, 0}^+) \setminus \lbrace 0 \rbrace  \subseteq \left[2 - \dfrac{1}{\lambda}, 2 - \dfrac{1}{\kappa} \right].$$

\end{proof}
Our next step is to iterate the above lemma. Consider the function $f(x) = 2 - \frac{1}{x}$. One can easily verify that this function is strictly monotone increasing and well defined on $(0, \infty)$. Denote $f^2 = f \circ f, f^j = f \circ ... \circ f$. Simple calculations show the following:
$$\forall m \in \mathbb{N}, f(\dfrac{m}{m+1}) = \dfrac{m-1}{m}, f(1)=1$$
$$ \forall a >1, \lbrace f^j (a) \rbrace_{j \in \mathbb{N}} \text{ is a decreasing sequence and } \lim_{j \rightarrow \infty}  f^j (a) = 1.$$ 

\begin{corollary}
\label{SpectralGapDescent2}
Let $X$ be as in the lemma and $f$ as above. Assume that there are $\kappa \geq \lambda >\frac{n-1}{n}$ such that
$$\bigcup_{\sigma \in \Sigma (n-2)} Spec (\Delta_{\sigma, 0}^+) \setminus \lbrace 0 \rbrace \subseteq [\lambda, \kappa],$$
then for every $-1 \leq k \leq n-3$ we have
$$\bigcup_{\tau \in \Sigma (k)} Spec (\Delta_{\tau, 0}^+) \setminus \lbrace 0 \rbrace  \subseteq \left[ f^{n-k-2} (\lambda) , f^{n-k-2} (\kappa) \right] \subseteq \left( \dfrac{k+1}{k+2} , f^{n-k-2} (\kappa) \right].$$ 
\end{corollary}

\begin{proof}
The proof is a straightforward induction using lemma \ref{SpectralGapDescent1}. One only needs to verify that for every $-1 \leq k \leq n-3$ we have $f^{n-k-2} (\lambda) >0$, but this is guaranteed by the condition $\lambda >\frac{n-1}{n}$.
\end{proof}

\begin{corollary}
\label{kappa = 2 bound}
Let $X$ be as above. Then for every $-1 \leq k \leq n-2$ we have 
$$\bigcup_{\tau \in \Sigma (k)} Spec (\Delta_{\tau, 0}^+)   \subseteq \left[ 0 , \dfrac{n-k}{n-k-1} \right] .$$
Moreover, for every $-1 \leq k \leq n-3$, every $\tau \in X^{(k)}$ and every $\phi \in C^0 (X_\tau, \mathbb{R})$, 
$$\Delta_{\tau, 0}^+ \phi = \dfrac{n-k}{n-k-3} \phi \Rightarrow \forall \sigma \in X_\tau^{(n-k-1)}, \Delta_{\sigma, 0}^+ \phi^\sigma = 2 \phi^\sigma .$$
\end{corollary}

\begin{proof}
Notice that 
$$f^{n-k-2} (2) =  \dfrac{n-k}{n-k-1} .$$
Recall that $\Vert \Delta_0^+ \Vert \leq 2$ and therefore, one can always take $\kappa =2$. By corollary \ref{SpectralGapDescent2} we get that for every $-1 \leq k \leq n-2$ we have 
$$\bigcup_{\tau \in \Sigma (k)} Spec (\Delta_{\tau, 0}^+) \setminus \lbrace 0 \rbrace  \subseteq \left[ 0 , \dfrac{n-k}{n-k-1} \right] .$$
Let $-1 \leq k \leq n-2$, $\tau \in X^{(k)}$ and $\phi \in C^0 (X_\tau, \mathbb{R})$ such that $\Delta_0^+ \phi = \mu \phi$.  Assume there is a single $v \in X_\tau^{(0)}$ such that (in the notations of the proof of lemma \ref{SpecGapLocalToGlobal1} )
$$ \dfrac{n-k-1}{n-k-2} \Vert (\phi^v)^1 \Vert^2 > \Vert d_{\tau v} \phi^v \Vert^2.$$
By the fact proven above, for any other $v \in X_\tau^{(0)}$, we have 
$$ \dfrac{n-k-1}{n-k-2} \Vert (\phi^v)^1 \Vert^2 \geq \Vert d_{\tau v} \phi^v \Vert^2.$$
Therefore we can repeat the proof of lemma \ref{SpecGapLocalToGlobal1}, with strict inequalities. Namely, instead of inequality \eqref{ineq}, we can take
$$ \dfrac{n-k-1}{n-k-2} \sum_{v \in \Sigma_\tau (0)}  \Vert (\phi^v)^1 \Vert^2 > \mu \Vert \phi \Vert^2,$$
and complete the rest of the proof with strict inequalities and get
$$2 - \dfrac{1}{\frac{n-k-1}{n-k-2}} > \mu ,$$
which yields 
$$\dfrac{n-k}{n-k-1} > \mu.$$
Therefore 
$$\Delta_{\tau,0}^+ \phi = \dfrac{n-k}{n-k-1}  \phi \Rightarrow \forall v \in X_\tau^{(0)}, \Delta_{\tau v,0}^+ \phi^v = \dfrac{n-k-1}{n-k-2}  \phi^v.$$
Finish by induction on $k$, starting with $k=n-3$ and descending.
\end{proof}

\subsection{Local to global}

We'll show that large enough spectral gaps of the upper Laplacian $\Delta_{\tau,0}^+$ for all $\tau \in \Sigma (k-1)$ implies spectral gaps for $\Delta_k^+$ . 

\begin{lemma}
\label{SpecGapLocalToGlobal1}
Let $X$ as before, i.e., a pure $n$-dimensional weighted simplicial complex such that all the links of $X$ of dimension $>0$ are connected. Also, assume that $n>1$. For $0 \leq k \leq n-1$, if there are $\kappa \geq \lambda >0$ such that
$$\bigcup_{\tau \in \Sigma (k-1)} Spec (\Delta_{\tau,0}^+) \setminus \lbrace 0 \rbrace \subseteq [\lambda, \kappa],$$
then for every $\phi \in C^k (X, \mathbb{R})$ we have
$$(k+1) \Vert \phi \Vert^2 \left( \kappa - \dfrac{k}{k+1} \right) - \kappa  \Vert \delta \phi \Vert^2  \geq    \Vert d \phi \Vert^2  \geq (k+1) \Vert \phi \Vert^2 \left( \lambda - \dfrac{k}{k+1} \right) - \lambda  \Vert \delta \phi \Vert^2.$$
\end{lemma}

\begin{proof}
Let $0 \leq k \leq n-1$. Fix some $\tau \in \Sigma (k-1)$ and some $\phi \in C^k (X, \mathbb{R})$. For $\phi_\tau$ recall that $\Delta_{\tau,0}^- \phi_\tau$ is the projection of $\phi_\tau$ on the space of constant functions. Denote by $(\phi_\tau)^1$ the orthogonal complement of this projection. Since $X_\tau$ is connected we have that $ker (\Delta_{\tau,0}^+)$ is exactly the space of constant functions and therefore 
$$\kappa \Vert (\phi_\tau)^1 \Vert^2 \geq \left\langle \Delta_{\tau, 0}^+ \phi_\tau , \phi_\tau \right\rangle \geq \lambda \Vert (\phi_\tau)^1 \Vert^2.$$
Note that $\Vert (\phi_\tau)^1 \Vert^2 =\Vert \phi_\tau \Vert^2 - \Vert \Delta_{\tau,0}^- \phi_\tau \Vert^2$ and that $\left\langle \Delta_{\tau, 0}^+ \phi_\tau , \phi_\tau \right\rangle = \Vert d_\tau \phi_\tau \Vert^2$. Therefore
$$\kappa \Vert \phi_\tau \Vert^2 - \Vert \Delta_{\tau,0}^- \phi_\tau \Vert^2  \geq \Vert d_\tau \phi_\tau \Vert^2 \geq \lambda \Vert \phi_\tau \Vert^2 - \Vert \Delta_{\tau,0}^- \phi_\tau \Vert^2.$$
Since the above inequality is true for every $\tau \in \Sigma (k-1)$ we can sum over all $\tau \in \Sigma (k-1)$ and get
$$\kappa \sum_{\tau \in \Sigma (k-1)} \left( \Vert \phi_\tau \Vert^2 - \Vert \Delta_{\tau,0}^- \phi_\tau \Vert^2 \right) \geq \sum_{\tau \in \Sigma (k-1)}  \Vert d_\tau \phi_\tau \Vert^2 \geq \lambda \sum_{\tau \in \Sigma (k-1)}  \left( \Vert \phi_\tau \Vert^2 - \Vert \Delta_{\tau,0}^- \phi_\tau \Vert^2 \right).$$
By proposition \ref{Delta0Norm} we have that $\Vert  \Delta_{\tau,0}^- \phi_\tau \Vert^2 = \Vert \delta_{\tau,0} \phi_\tau \Vert^2$, therefore we can write 
$$\kappa \sum_{\tau \in \Sigma (k-1)} \left( \Vert \phi_\tau \Vert^2 - \Vert \delta_{\tau,0} \phi_\tau \Vert^2 \right) \geq \sum_{\tau \in \Sigma (k-1)}  \Vert d_\tau \phi_\tau \Vert^2 \geq \lambda \sum_{\tau \in \Sigma (k-1)}  \left( \Vert \phi_\tau \Vert^2 - \Vert \delta_{\tau,0} \phi_\tau \Vert^2 \right).$$
By lemma \ref{LocalizNorm1}, applied for $\phi = \psi$, we have that 
$$\sum_{\tau \in \Sigma (k-1)} \left( \Vert \phi_\tau \Vert^2 - \Vert \delta_{\tau,0} \phi_\tau \Vert^2 \right) = (k+1)! \Vert \phi \Vert^2 - k! \Vert \delta \phi \Vert^2.$$
Therefore
$$\kappa \left( (k+1)! \Vert \phi \Vert^2 - k! \Vert \delta \phi \Vert^2 \right) \geq \sum_{\tau \in \Sigma (k-1)}  \Vert d_\tau \phi_\tau \Vert^2 \geq \lambda \left( (k+1)! \Vert \phi \Vert^2 - k! \Vert \delta \phi \Vert^2 \right).$$
By corollary \ref{LocalizNorm3} we have for every $\phi \in C^k (X, \mathbb{R})$ that
  $$  k! \Vert d \phi \Vert^2 + k!k \Vert \phi \Vert^2 =  \sum_{\tau \in \Sigma (k-1)} \Vert d_\tau \phi_\tau \Vert^2 ,$$ 
and therefore
$$\kappa \left( (k+1)! \Vert \phi \Vert^2 - k! \Vert \delta \phi \Vert^2 \right) \geq   k! \Vert d \phi \Vert^2 + k!k \Vert \phi \Vert^2 \geq \lambda \left( (k+1)! \Vert \phi \Vert^2 - k! \Vert \delta \phi \Vert^2 \right).$$
Dividing by $k!$ and then subtracting $k \Vert \phi \Vert^2$ gives the inequality stated in the lemma.
\end{proof}

\begin{corollary}
\label{norm bound - local to global}
Let $X$ as in the above lemma. For $0 \leq k \leq n-1$, if there are $\kappa \geq \lambda >\frac{k}{k+1}$ such that
$$\bigcup_{\tau \in \Sigma (k-1)} Spec (\Delta_{\tau, 0}^+) \setminus \lbrace 0 \rbrace \subseteq [\lambda, \kappa],$$
then 
$$\left\Vert \Delta_k^+ + \dfrac{\lambda + \kappa}{2} \Delta_k^- - (k+1) (\dfrac{\lambda + \kappa}{2} - \dfrac{k}{k+1}) I \right\Vert \leq (k+1) \dfrac{\kappa - \lambda}{2},$$
where $\Vert. \Vert$ denotes the operator norm.
\end{corollary}

\begin{proof}
From lemma \ref{SpecGapLocalToGlobal1} we have for every $\phi \in C^{k} (X, \mathbb{R})$ that 
$$(k+1) \Vert \phi \Vert^2 \left( \kappa - \dfrac{k}{k+1} \right) - \kappa  \Vert \delta \phi \Vert^2  \geq    \Vert d \phi \Vert^2  \geq (k+1) \Vert \phi \Vert^2 \left( \lambda - \dfrac{k}{k+1} \right) - \lambda  \Vert \delta \phi \Vert^2.$$
This yields
\begin{dmath*}
(k+1) \langle \phi, \phi \rangle\left( \kappa - \dfrac{k}{k+1} \right) - \kappa  \langle \Delta_k^- \phi, \phi \rangle  \geq   \langle \Delta_k^+ \phi, \phi \rangle  \geq (k+1) \langle \phi, \phi \rangle \left( \lambda - \dfrac{k}{k+1} \right) - \lambda  \langle \Delta_k^- \phi, \phi \rangle,
\end{dmath*}
which yields
\begin{dmath*}
\dfrac{\kappa - \lambda}{2}  \left\langle ((k+1)I - \Delta_k^- )\phi, \phi \right\rangle    \geq   \left\langle \left( \Delta_k^+ +\dfrac{\kappa + \lambda}{2}\Delta_k^-  - (k+1)\left( \dfrac{\kappa + \lambda}{2} - \dfrac{k}{k+1} \right) I \right) \phi, \phi \right\rangle  \geq - \dfrac{\kappa - \lambda}{2} \left\langle ((k+1)I - \Delta_k^- )\phi, \phi \right\rangle .
\end{dmath*}
Therefore, we have for every $\phi$ that
\begin{dmath*}
\left\vert \left\langle \left( \Delta_k^+ +\dfrac{\kappa + \lambda}{2}\Delta_k^-  - (k+1)\left( \dfrac{\kappa + \lambda}{2} - \dfrac{k}{k+1} \right) I \right) \phi, \phi \right\rangle \right\vert \leq \dfrac{\kappa - \lambda}{2}  \left\langle ((k+1)I - \Delta_k^- )\phi, \phi \right\rangle \leq \dfrac{\kappa - \lambda}{2} (k+1) \Vert \phi \Vert^2.
\end{dmath*}
Note that $\Delta_k^+ + \frac{\lambda + \kappa}{2} \Delta_k^- - (k+1) (\frac{\lambda + \kappa}{2} - \frac{k}{k+1}) I $ is a self adjoint operator and therefore the above inequality gives the inequality stated in the theorem.
\end{proof}

\begin{corollary}
\label{SpecGapLocalToGlobal2}
Let $X$ as in the above lemma. For $0 \leq k \leq n-1$, if there are $\kappa \geq \lambda >\frac{k}{k+1}$ such that
$$\bigcup_{\tau \in \Sigma (k-1)} Spec (\Delta_{\tau, 0}^+) \setminus \lbrace 0 \rbrace \subseteq [\lambda, \kappa],$$
then $\widetilde{H}^k (X, \mathbb{R} )=0$, there is an orthogonal decomposition $C^k (X,\mathbb{R}) = ker (\Delta_k^+) \oplus ker (\Delta_k^-)$  and 
$$ Spec (\Delta_{k}^+) \setminus \lbrace 0 \rbrace \subseteq [(k+1 )\lambda - k, (k+1) \kappa- k], $$
$$Spec (\Delta_{k+1}^-) \setminus \lbrace 0 \rbrace \subseteq [(k+1 )\lambda - k, (k+1) \kappa- k] ,$$
$$ Spec (\Delta_{k}^-) \setminus \lbrace 0 \rbrace \subseteq \left[(k+1 )- \dfrac{k}{\lambda},(k+1 )- \dfrac{k}{\kappa} \right] ,$$
$$ Spec (\Delta_{k-1}^+) \setminus \lbrace 0 \rbrace \subseteq \left[(k+1 )- \dfrac{k}{\lambda},(k+1 )- \dfrac{k}{\kappa} \right] .$$
\end{corollary}

\begin{proof}
Since $X$ is assumed to be connected all the statements for $k=0$ are trivial. Assume that $1 \leq k \leq n-1$. First notice that if $\lambda >\frac{k}{k+1}$ we get by lemma \ref{SpecGapLocalToGlobal1} that for every $\phi \in C^k (X, \mathbb{R}) \setminus \lbrace 0 \rbrace$ ( ($0$ here is the constant $0$ function in $C^k (X, \mathbb{R})$).
$$ \left\langle \Delta_k^+ \phi, \phi \right\rangle + \lambda  \left\langle \Delta_k^- \phi, \phi \right\rangle \geq (k+1) \Vert \phi \Vert^2 \left( \lambda - \dfrac{k}{k+1} \right)  > 0.$$
Therefore $ker (\Delta_k) = ker (\Delta_k^+) \cap ker (\Delta_k^-) = \lbrace 0 \rbrace$.  Recall that proposition \ref{HodgeConsid} $H^k (X, \mathbb{R} )= ker (\Delta_k)$ and therefore $H^k (X, \mathbb{R} )= 0$.  From $H^k (X, \mathbb{R} )= 0$ we get that $ker (d_k) = im (d_{k-1})$. Recall that (see proof of proposition \ref{HodgeConsid})
$$ker (d_k) = ker (\Delta_k^+), im (d_{k-1}) = im (\Delta_k^-),$$
and therefore $ker (\Delta_k^+) = im (\Delta_k^-)$. Since $\Delta_k^-$ is self adjoint we get that 
$$ (ker (\Delta_k^+))^\perp = (im (\Delta_k^-))^\perp = ker (\Delta_k^-),$$
and as a consequence, $ker (\Delta_k^-)^\perp = ker (\Delta_k^+)$. Therefore there is an orthogonal decomposition $C^k (X,\mathbb{R}) = ker (\Delta_k^+) \oplus ker (\Delta_k^-)$. \\
Also, since $\Delta_k^+$ is self adjoint, we get that
$$Spec (\Delta_{k}^+) \setminus \lbrace 0 \rbrace  = Spec (\Delta_{k}^+ \vert_{(ker (\Delta_k^+))^\perp}) =Spec (\Delta_{k}^+ \vert_{ker (\Delta_k^-)}).$$
For every $\phi \in ker (\Delta_k^-)$ we have by lemma \ref{SpecGapLocalToGlobal1} that 
$$(k+1) \Vert \phi \Vert^2 \left( \kappa - \dfrac{k}{k+1} \right) \geq \left\langle \Delta_k^+ \phi, \phi \right\rangle \geq (k+1) \Vert \phi \Vert^2 \left( \lambda - \dfrac{k}{k+1} \right) .$$
Therefore 
$$ Spec (\Delta_{k}^+) \setminus \lbrace 0 \rbrace \subseteq [(k+1 )\lambda - k, (k+1) \kappa- k].$$
By proposition \ref{HodgeConsid} we get that
$$Spec (\Delta_{k+1}^-) \setminus \lbrace 0 \rbrace \subseteq [(k+1 )\lambda - k, (k+1) \kappa- k] .$$
By the same considerations, 
$$Spec (\Delta_{k}^-) \setminus \lbrace 0 \rbrace = Spec (\Delta_{k}^+ \vert_{ker (\Delta_k^+)}).$$
For every $\phi \in ker (\Delta_k^-)$ we have by lemma \ref{SpecGapLocalToGlobal1} that
$$(k+1) \Vert \phi \Vert^2 \left( \kappa - \dfrac{k}{k+1} \right) - \kappa  \left\langle \Delta_k^- \phi , \phi \right\rangle  \geq   0,$$
$$0  \geq (k+1) \Vert \phi \Vert^2 \left( \lambda - \dfrac{k}{k+1} \right) - \lambda  \left\langle \Delta_k^- \phi , \phi \right\rangle.$$
Therefore 
$$ Spec (\Delta_{k}^-) \setminus \lbrace 0 \rbrace \subseteq [(k+1 )- \dfrac{k}{\lambda},(k+1 )- \dfrac{k}{\kappa}] .$$
By proposition \ref{HodgeConsid} we get that
$$ Spec (\Delta_{k-1}^+) \setminus \lbrace 0 \rbrace \subseteq [(k+1 )- \dfrac{k}{\lambda},(k+1 )- \dfrac{k}{\kappa}] .$$
\end{proof}

\begin{corollary}
\label{Laplacian norm bounds}
Let $X$ as above. Then for every $0 \leq k \leq n-1$, we have that
$$Spec (\Delta^+_k) \subseteq \left[ 0, \dfrac{n+1}{n-k} \right] ,$$
$$Spec (\Delta^-_{k+1}) \subseteq \left[ 0, \dfrac{n+1}{n-k} \right] .$$
\end{corollary}

\begin{proof}
Combine the above corollary with corollary \ref{kappa = 2 bound}, stating that
$$\bigcup_{\tau \in \Sigma (k-1)} Spec (\Delta_{\tau, 0}^+)   \subseteq \left[ 0 , \dfrac{n+1-k}{n-k} \right] ,$$
and therefore $\kappa \leq \frac{n+1-k}{n-k}$. This yields that 
\begin{dmath*}
(k+1) \kappa -k \leq (k+1)\dfrac{n+1-k}{n-k} -k = \dfrac{n+1}{n-k} .
\end{dmath*}
\end{proof}

\subsection{Very local to very global}

Combining lemma \ref{SpecGapLocalToGlobal1} and corollaries \ref{SpecGapLocalToGlobal2}, \ref{SpectralGapDescent2} we'll prove the exact formulation of theorem \ref{cohomology and spectral gaps - section 2} stated above. Namely, we'll show that large spectral gap in all the $1$-dimensional links yield spectral gaps in $\Delta_k^+, \Delta_{k+1}^-$ for every $0 \leq k \leq n-1$ .

\begin{theorem}
\label{SpecGapThm}
Let $X$ be a pure $n$-dimensional weighted simplicial complex such that all the links of $X$  of dimension $>0$ are connected. Also, assume that $n>1$. Denote $f(x) = 2-\frac{1}{x}$ and $f^j$ to be the composition of $f$ with itself $j$ times (where $f^0$ is defined as $f^0 (x) = x$).
If there are $\kappa \geq \lambda > \frac{n-1}{n}$ such that
$$\bigcup_{\tau \in \Sigma (n-2)} Spec (\Delta_{\tau, 0}^+) \setminus \lbrace 0 \rbrace \subseteq [\lambda, \kappa].$$
Then for every $0 \leq k \leq n-1$:
\begin{enumerate}
\item $\widetilde{H}^k (X, \mathbb{R} )=0$ and $C^k (X,\mathbb{R}) = ker (\Delta_k^+) \oplus ker (\Delta_k^-)$.
\item 
$$ Spec (\Delta_{k}^+) \setminus \lbrace 0 \rbrace \subseteq [(k+1 ) f^{n-1-k} (\lambda) - k, (k+1) f^{n-1-k} (\kappa)- k], $$
$$Spec (\Delta_{k+1}^-) \setminus \lbrace 0 \rbrace \subseteq [(k+1 ) f^{n-1-k} (\lambda) - k, (k+1) f^{n-1-k} (\kappa)- k] .$$
\end{enumerate} 
\end{theorem}

\begin{proof}
First apply corollary \ref{SpectralGapDescent2} to get spectral gaps of $\Delta_{\tau, 0}^+$ for every $\tau \in \Sigma (k)$ when $-1 \leq k \leq n-3$ in terms of $f$ and $\lambda, \kappa$ (notice that since $X_\emptyset = X$ this takes care of the case $k=0$ in $3.$ of the theorem). Then apply corollaries \ref{norm bound - local to global} and \ref{SpecGapLocalToGlobal2} to finish the proof.
\end{proof}

\begin{remark}
In the above proof it seems that we are only using two of the estimates given in corollary \ref{SpecGapLocalToGlobal2} and that we have two additional estimates of the spectrum of $\Delta_k^+,  \Delta_{k+1}^-$. We leave it to the reader to check that when using the function $f$, the two estimates given in corollary \ref{SpecGapLocalToGlobal2} coincide. 
\end{remark}

\begin{remark}
As remarked earlier, if $m$ is the homogeneous weight function, then for every $\tau \in \Sigma (n-2)$, $\Delta_{\tau, 0}^+$ is the usual graph Laplacian on the graph $X_\tau$. This means that if one assigns the  homogeneous weight on $X$, then the spectral gap conditions stated in the above theorem are simply spectral gaps conditions of the usual graph Laplacian on each of the $1$-dimensional links. In concrete examples, these spectral gap conditions are easily attainable (see examples below).
\end{remark}

\subsection{partite complexes}
 
 \begin{proposition}
Let $X$  be a pure $n$-dimensional weighted simplicial complex such that all the links of $X$ of dimension $>0$ are connected. Then we have for the spectrum of $\Delta_0^+$ that:
$$Spec (\Delta_0^+) \subseteq \left[0,\dfrac{n+1}{n} \right].$$
If $X$ is also $(n+1)$-partite then the space of eigenfunctions of the eigenvalue $\frac{n+1}{n}$ is spanned by the functions $\varphi_i$, $0 \leq i \leq n$ defined as
$$\varphi_i (u) = \begin{cases}
n & u \in S_i \\
-1 & \text{otherwise}
\end{cases}.$$
\end{proposition}

\begin{proof}
The claim about the non trivial spectrum of $\Delta_0^+$ is due to corollary \ref{Laplacian norm bounds} applied for $k=0$.
Assume that $X$ is $(n+1)$-partite. First we verify that each $\varphi_i$ defined above is indeed an eigenfunction of the eigenvalue $\frac{n+1}{n}$. We check the following cases:
\begin{enumerate}
\item In the case $u \in S_i$, we have that
\begin{dmath*}
\Delta^+_0 \varphi_i (u) = \varphi_i (u) - \sum_{v \in X^{(0)}, (u,v) \in \Sigma(1)} \dfrac{m((u,v))}{m(u)} \varphi_i (v) = 
n - \sum_{v \in X^{(0)}, (u,v) \in \Sigma(1)} \dfrac{m((u,v))}{m(u)} (-1) =  
n + \sum_{v \in X^{(0)}, (u,v) \in \Sigma(1)} \dfrac{m((u,v))}{m(u)} = n+1 = \dfrac{n+1}{n} \varphi_i (u) . 
\end{dmath*}
\item In the case where $u \notin  S_i$, we have that 
\begin{dmath*}
\Delta^+_0 \varphi_i (u) = -1  - \sum_{v \in X^{(0)}, (u,v) \in \Sigma(1)} \dfrac{m((u,v))}{m(u)} \varphi_i (v) = 
-1 - \sum_{v \in X^{(0)} \setminus S_i , (u,v) \in \Sigma(1)} \dfrac{m((u,v))}{m(u)} (-1) - \sum_{v \in S_i , (u,v) \in \Sigma(1)} \dfrac{m((u,v))}{m(u)} n.
\end{dmath*}

Recall that by proposition \ref{weight in n dim simplices} and by the fact that $X$ is pure $n$-dimensional and $(n+1)$-partite, we have that
\begin{dmath*}
m(u) = n! \sum_{\sigma \in X^{(n)}, u \subset \sigma } m(\sigma) = 
n! \sum_{v \in S_i, (u,v) \in \Sigma (1) } \sum_{\sigma \in X^{(n)}, \lbrace u, v \rbrace \subset \sigma } m(\sigma) =
n \sum_{v \in S_i, (u,v) \in \Sigma (1) } m((u,v)) .
\end{dmath*}
Similarly, 
$$ (n-1) m(u) = n \sum_{v \in X^{(0)} \setminus S_i, (u,v) \in \Sigma (1) } m((u,v)) .$$
Therefore we get
\begin{dmath*}
\Delta^+_0 \varphi_i (u) = -1 + \dfrac{n-1}{n} - 1  = - \dfrac{n+1}{n} =\dfrac{n+1}{n} \varphi_i (u) .
\end{dmath*}
\end{enumerate}
Next, we'll prove that $\varphi_i$ span the space of eigenfunctions with eigenvalue $\frac{n+1}{n}$. 
For $n=1$, this is the classical argument for bipartite graphs repeated here for the convenience of the reader. Let $\phi \in C^0 (X,\mathbb{R} )$ such that $\Delta_0^+ \phi = 2 \phi$ and $X$ is a bipartite graph. There is $u_0 \in X^{(0)}$ such that $\forall v \in X^{(0)}, \vert \phi (u_0 ) \vert \geq \vert \phi (v ) \vert$. Without loss of generality $u_0 \in S_0$. One can always normalize $\phi$ such that $\phi (u_0) =1$ and for every other $v \in X^{(0)}$, $\vert \phi (v) \vert \leq 1$ . Then 
\begin{dmath*}
2 = \Delta \phi (u_0) = 1 - \sum_{v \in X^{(0)}, (u_0,v) \in \Sigma(1)} \dfrac{m((u_0,v))}{m(u_0)} \phi (v) = 
1 - \sum_{v \in X^{(0)}_1, (u_0,v) \in \Sigma(1)} \dfrac{m((u_0,v))}{m(u_0)} \phi (v)
\end{dmath*}
Therefore 
$$ \sum_{v \in X^{(0)}_1, (u_0,v) \in \Sigma(1)} \dfrac{m((u_0,v))}{m(u_0)} \phi (v) = -1.$$
Note that $ \sum_{v \in X^{(0)}_1, (u_0,v) \in \Sigma(1)} \dfrac{m((u_0,v))}{m(u_0)}  =1 $ and $\forall v, \phi (v) \geq -1$ and therefore we get that for every $v \in X^{(0)}_1$ with $(u_0,v) \in \Sigma(1)$ we get $\phi (v) = -1$. By the same considerations, for every $v \in X^{(0)}$ with $\phi (v) =-1$, we have 
$$u \in X^{(0)}, (v,u) \in \Sigma (1) \Rightarrow \phi (u) =1.$$ 
Therefore by iterating this argument and using the fact that the graph is connected, we get that 
$$\phi (u) = \begin{cases}
 1 & u \in S_0 \\
 -1 & u \in X^{(0)}_1
\end{cases},$$
and that is exactly $\varphi_0$ in the case $n=1$. Assume that $n>1$. \\
First, for every $0 \leq i \leq n$, note that $\chi_{S_i} = \frac{1}{n+1} (\varphi_i + \chi_{X^{(0)}}) $ (Recall that $\chi_{X^{(0)}}$ denotes the constant $1$ function and $\chi_{S_i}$ denotes the indicator function of $S_i$ ). 
Therefore every function $\phi$ of the form: 
$$\exists c_0,...,c_n \in \mathbb{R}, \forall u \in S_i, \phi (u) = c_i ,$$
is in the space spanned by the functions $\varphi_i$ and the constant functions. Therefore, for $\phi$ such that $\Delta_0^+ \phi = \frac{n+1}{n} \phi$, it is enough to show that $\phi$ is of the form 
$$\exists c_0,...,c_n \in \mathbb{R}, \forall u \in S_i, \phi (u) = c_i .$$
Let $\phi \in C^0 (X,\mathbb{R})$ such that $\Delta_0^+ \phi = \frac{n+1}{n} \phi$. Fix $0 \leq i \leq n$ and $u' \in S_i$. By proposition \ref{gallery connectedness}, $X$ is gallery connected so for every $u \in S_i$ there is a gallery $\sigma_0,...,\sigma_l \in X^{(n)}$ connecting $u'$ and $u$. We'll show by induction on $l$ that $\phi (u) = \phi (u')$. For $l=0$, $u=u'$ and we are done. Assume the claim is true for $l$. Let $u \in  S_i$ such that the shortest gallery connecting $u'$ and $u$ is  $\sigma_0,...,\sigma_{l+1} \in X^{(n)}$. By the fact that $X$ is $(n+1)$-partite, there is $u'' \in \sigma_l  \cap S_i$ therefore $u'', u$ are both in the link of  $\sigma_l  \cap \sigma_{l+1} \in X^{(n-1)}$. Since $n>1$, $\sigma_l  \cap \sigma_{l+1}$ is of dimension $>1$, therefore there is a non empty simplex $\tau \in X^{(n-2)}$ such that $\tau \subset \sigma_l  \cap \sigma_{l+1}$. Note that by the $(n+1)$-partite assumption of $X$, we have that the link of $X_\tau$ is a bipartite graph, containing $u''$ and $u$. From corollary \ref{kappa = 2 bound} we have that 
$$\Delta_0^+ \phi^\tau = 2 \phi^\tau.$$
Therefore, from the case $n=1$, we get that 
$$\phi (u'') = \phi^\tau (u'') = \phi^\tau (u) =\phi (u).$$ 
By our induction assumption, $\phi (u') = \phi (u'') $ and therefore $\phi$ must be of the form stated above and we are done.
\end{proof}

\begin{remark}
The functions $\varphi_i$ defined above are not orthogonal to each other and in fact they don't even form a basis, because they a linearly dependent.
\end{remark}

The above proposition indicates that when dealing with an $(n+1)$-partite simplicial complex, one should think of the non trivial spectrum of $\Delta^+_0$ as $Spec (\Delta^+_0 ) \setminus \lbrace 0 , \frac{n+1}{n} \rbrace$. 
Following this logic, we denote the space of non trivial functions $C^0 (X,\mathbb{R} )_{nt}$ as 
$$C^0 (X,\mathbb{R} )_{nt} = span \lbrace \chi_{X^{(0)}}, \varphi_0,...,\varphi_n \rbrace^\perp .$$

\begin{proposition}
Let $\chi_{S_i}$ be the indicator function of $S_i$, then
$$C^0 (X,\mathbb{R} )_{nt} = span \lbrace \chi_{S_0},...,\chi_{S_n} \rbrace^\perp .$$
Moreover, for every $\phi \in C^0 (X,\mathbb{R} )$, the projection of $\phi$ on $C^0 (X,\mathbb{R} )_{nt}$ is 
$$\phi - (n+1) \sum_{j=0}^n \Delta^-_{(0,j)} \phi .$$
\end{proposition}

\begin{proof}
As noted in the proof of the proposition above, $\chi_{S_i} = \frac{1}{n+1} (\varphi_i + \chi_{X^{(0)}}) $. Also notice that
$$\chi_{X^{(0)}} = \sum_{i=0}^n \chi_{S_i} ,$$
$$\forall i, \varphi_i = \sum_{j=0}^n \chi_{X^{(0)}_j} + (n-1) \chi_{S_i}.$$
Therefore 
$$ span \lbrace \varphi_0,...,\varphi_n, \chi_{X^{(0)}} \rbrace =  span \lbrace \chi_{S_0},...,\chi_{S_n} \rbrace .$$
Notice that for every $j$, 
$$\Vert \chi_{S_j} \Vert^2 = \sum_{v \in S_j} m(v) = \dfrac{1}{n+1} m (\emptyset ),$$
and for every $\phi \in C^0 (X,\mathbb{R} )$
\begin{dmath*}
\langle \phi , \chi_{S_j} \rangle \chi_{S_j} = \left( \sum_{v \in S_j} \phi (v) \right) \chi_{S_j}  =  m(\emptyset ) \Delta^-_{(0,j)} \phi.
\end{dmath*}
Therefore, for every $\phi \in C^0 (X,\mathbb{R} )$, the projection of $\phi$ on $C^0 (X,\mathbb{R} )_{nt}$ is 
\begin{dmath*}
\sum_{j=0}^n \dfrac{1}{\Vert \chi_{S_j}  \Vert^2} \langle \phi, \chi_{S_j}  \rangle \chi_{S_j} = \dfrac{n+1}{m(\emptyset )} \sum_{j=0}^n  m(\emptyset ) \Delta^-{(0,j)} \phi = (n+1) \sum_{j=0}^n \Delta^-_{(0,j)} \phi .
\end{dmath*}
\end{proof}

Next, we have a technical tool to calculate to norm and Laplacian of functions in $C^0 (X, \mathbb{R})_{nt}$:

\begin{proposition}
Let $X$  be a pure $n$-dimensional, $(n+1)$-partite, weighted simplicial complex such that all the links of $X$ of dimension $>0$ are connected. Let $\phi \in C^0 (X, \mathbb{R})$. For every $0 \leq i \leq n$, define $\phi_{i} (u) \in C^0 (X, \mathbb{R})$ as follows:
$$\phi_{i} (u) = \begin{cases}
-n \phi (u) & u \in S_i \\
 \phi (u) & \text{otherwise}
\end{cases} .$$
Then 
\begin{enumerate}
\item If $\phi \in C^0 (X, \mathbb{R})_{nt}$, then for every $0 \leq i \leq n$, we have that $\phi_{i} (u) \in C^0 (X, \mathbb{R})_{nt}$. 
\item For every $\phi \in C^0 (X, \mathbb{R})$,
$$\sum_{i=0}^n \Vert \phi_{i} \Vert^2 = (n^2 + n) \Vert \phi \Vert^2 .$$
\item For every $\phi \in C^0 (X, \mathbb{R})$,
$$\sum_{i=0}^n \langle \phi_{i}, \Delta^+_0 \phi_{i} \rangle = \langle \phi , ((n+1)^2 I - (n+1) \Delta^+_0 ) \phi \rangle.$$
\end{enumerate} 
\end{proposition}

\begin{proof}
\begin{enumerate}
\item Let $\phi \in C^0 (X,\mathbb{R} )_{nt} $. Fix $0 \leq i \leq n$. Note that for every $0 \leq j \leq n$, we have that
$$\langle \phi , \chi_{X^{(0)}_j} \rangle = 0  \Rightarrow \langle \phi_{i} , \chi_{X^{(0)}_j} \rangle = 0, $$
and therefore by the above proposition $\phi_{i} (u) \in C^0 (X, \mathbb{R})_{nt}$. 
\item For every $0 \leq i \leq n$ we have that
\begin{dmath*}
\Vert \phi_{i} \Vert^2 = \sum_{u \in S_i} m(u) n^2 \phi (u)^2 +  \sum_{u \in X^{(0)} \setminus S_i} m(u) \phi (u)^2.
\end{dmath*}
Therefore
\begin{dmath*}
\sum_{i=0}^n \Vert \phi_{i} \Vert^2 = \sum_{u \in X^{(0)}} m(u) (n^2+n) \phi (u)^2 = (n^2 +n ) \Vert \phi \Vert^2 .
\end{dmath*}
\item For every $0 \leq i \leq n$, we'll compute $\Delta^+_0 \phi_{i}$: For $u \in S_i$, we have that
\begin{dmath*}
(\Delta^+_0 \phi_{i}) (u) = -n \phi (u) - \sum_{v \in X^{(0)}, (u,v) \in \Sigma (1)} \dfrac{m((u,v))}{m(u)} \phi (v) = (-n-1) \phi (u) + (\Delta^+_0 \phi) (u). 
\end{dmath*}
For $u \in X^{(0)} \setminus S_i$ we have that
\begin{dmath*}
(\Delta^+_0 \phi_{i}) (u) = \phi (u) - \sum_{v \in X^{(0)} \setminus S_i, (u,v) \in \Sigma (1)} \dfrac{m((u,v))}{m(u)} \phi (v) - \sum_{v \in S_i, (u,v) \in \Sigma (1)} \dfrac{m((u,v))}{m(u)} (-n) \phi (v) = (\Delta^+_0 \phi) (u) + (n+1) \sum_{v \in S_i, (u,v) \in \Sigma (1)} \dfrac{m((u,v))}{m(u)} \phi (v) .
\end{dmath*}
Therefore
\begin{dmath*}
\langle  \phi_{i},\Delta^+_0 \phi_{i} \rangle = \sum_{u \in S_i} m(u) \phi (u) \left( -n (-n-1) \phi (u) -n (\Delta^+_0 \phi) (u) \right) + \sum_{u \in X^{(0)} \setminus S_i} m(u) \phi (u) \left( (\Delta^+_0 \phi) (u) + (n+1) \sum_{v \in S_i, (u,v) \in \Sigma (1)} \dfrac{m((u,v))}{m(u)} \phi (v) \right).
\end{dmath*}
This yields
\begin{dmath*}
\sum_{i=0}^n \langle  \phi_{i},\Delta^+_0 \phi_{i} \rangle = \sum_{u \in X^{(0)}} m(u) \phi (u) \left( n (n+1) \phi (u) -n (\Delta^+_0 \phi) (u) + n (\Delta^+_0 \phi) (u) + (n+1)  \sum_{v \in X^{(0)}, (u,v) \in \Sigma (1)} \dfrac{m((u,v))}{m(u)} \phi (v) \right) = \sum_{u \in X^{(0)}} m(u) \phi (u) \left( (n+1)^2 \phi (u) - (n+1) (\Delta^+_0 \phi) (u)  \right)= \langle \phi , ((n+1)^2 I - (n+1) \Delta^+_0 ) \phi \rangle.
\end{dmath*}
\end{enumerate}
\end{proof}

It is known that for bipartite graph, the spectrum of the Laplacian is symmetric around $1$. For $(n+1)$-partite complexes we have a weaker result that shows that the bounds of the non trivial spectrum have some symmetry around $1$:

\begin{lemma}
\label{spectral bound in the partite case} 
Let $X$  be a pure $n$-dimensional, $(n+1)$-partite, weighted simplicial complex such that all the links of $X$ of dimension $>0$ are connected. Assume that $X$ is non trivial, i.e., assume that $X$ has more than $1$ $n$-dimensional simplex. Denote 
$$\lambda (X) = \min \lbrace \lambda :  \lambda > 0, \exists \phi, \Delta^+_0 \phi = \lambda \phi \rbrace ,$$
$$\kappa (X) = \max \lbrace \lambda :  \lambda < \frac{n+1}{n}, \exists \phi, \Delta^+_0 \phi = \lambda \phi \rbrace .$$ 
Then 
$$1 - \frac{1}{n} (1 -\lambda (X)) \leq \kappa (X) \leq \min \lbrace 1 - n (1-\lambda (X)), \dfrac{n+1}{n} \rbrace .$$

\end{lemma}

\begin{proof}
Let $\phi \in C^0 (X,\mathbb{R} )_{nt} $ by the eigenfunction of $\kappa (X)$. By the above proposition, for every $0 \leq i \leq n$, $\phi_{i} \in C^0 (X,\mathbb{R} )_{nt} $ and therefore
$$ \langle \phi_{i} , \Delta^+_0 \phi_{i} \rangle \geq \lambda (X) \Vert \phi_{i}  \Vert^2 .$$
Summing on $i$ we get
$$  \sum_{i=0}^n \langle \phi_{i} , \Delta^+_0 \phi_{i} \rangle \geq \lambda (X)  \sum_{i=0}^n \Vert \phi_{i}  \Vert^2 .$$
By the equalities proven in the above proposition, this yields
$$   \langle \phi , ((n+1)^2 I - (n+1) \Delta^+_0 ) \phi \rangle \geq \lambda (X)  (n^2+n) \Vert \phi \Vert^2 .$$
Since we took $\phi$ to be the eigenfunction of $\kappa (X)$, this yields 
$$ ((n+1)^2 - (n+1) \kappa (X)) \Vert \phi \Vert^2 \geq  \lambda (X)  (n^2+n) \Vert \phi \Vert^2 .$$
Therefore 
$$1 +n (1- \lambda (X)) \geq \kappa (X).$$
(Also, recall that $\kappa (X) < \frac{n+1}{n}$). 
By the same procedure, when $\phi$ is taken to be the eigenfunction of $\lambda (X)$, we get that
$$ ((n+1)^2 - (n+1) \lambda (X)) \Vert \phi \Vert^2 \leq  \kappa (X)  (n^2+n) \Vert \phi \Vert^2 ,$$
and therefore
$$1 + \dfrac{1}{n} (1- \lambda (X)) \leq \kappa (X) .$$

\end{proof}

The next theorem is the $(n+1)$-partite analogue of corollary \ref{norm bound - local to global}:
\begin{theorem}

Let $X$  be a pure $n$-dimensional, $(n+1)$-partite, weighted simplicial complex such that all the links of $X$ of dimension $>0$ are connected. Fix $0 \leq k \leq n-1$, if there are $\kappa \geq \lambda > \frac{k}{k+1}$ such that 
$$\bigcup_{\tau \in \Sigma (k-1)} Spec (\Delta_{\tau, 0}^+) \setminus \lbrace 0, \frac{n+1-k}{n-k} \rbrace \subseteq [\lambda, \kappa],$$
then 
\begin{dmath*}
\left\Vert \Delta^+_{k} +\frac{n+1-k}{n-k}   \Delta^-_{k} + (k - (k+1)  \dfrac{\lambda + \kappa}{2})  I  - ( \frac{(n+1-k)^2}{n-k} - (n+1-k)^2 \dfrac{\lambda + \kappa}{2} ) \sum_{j=0}^n \Delta^-_{(k,j)}  \right\Vert \leq (k+1) \dfrac{\kappa - \lambda}{2} ,
\end{dmath*}
where $\Vert. \Vert$ denotes the operator norm.
\end{theorem}

\begin{proof}
Let $\phi \in C^{k} (X,\mathbb{R} )$, then for every $\tau \in \Sigma (k-1)$, we have that the projection of $\phi_\tau$ on $C^{0} (X_\tau, \mathbb{R})_{nt}$ is
$$\phi_\tau - (n+1-k) \sum_{j=0}^n \Delta^-_{\tau,(0,j)} \phi_\tau = (I -  (n+1-k) \sum_{j=0}^n \Delta^-_{\tau,(0,j)}) \phi_\tau  .$$
Therefore,
\begin{dmath*}
\left\langle \Delta^+_\tau  \left(I -  (n+1-k) \sum_{j=0}^n \Delta^-_{\tau,(0,j)} \right) \phi_\tau, \phi_\tau \right\rangle \geq \lambda \left\Vert \left( I -  (n+1-k) \sum_{j=0}^n \Delta^-_{\tau,(0,j)} \right) \phi_\tau \right\Vert^2 = \lambda \left( \Vert \phi_\tau \Vert^2 - (n+1-k)^2 \sum_{j=0}^n \Vert \Delta^-_{\tau,(0,j)} \phi_\tau \Vert^2 \right).
\end{dmath*}
Similarly,
\begin{dmath*}
 \kappa \left( \Vert \phi_\tau \Vert^2 - (n+1-k)^2 \sum_{j=0}^n \Vert \Delta^-_{\tau,(0,j)} \phi_\tau \Vert^2 \right) \geq \left\langle \Delta^+_\tau  \left(I -  (n+1-k) \sum_{j=0}^n \Delta^-_{\tau,(0,j)} \right) \phi_\tau, \phi_\tau \right\rangle .
 \end{dmath*}
From the fact that $\Delta^-_{\tau,0}$ is the projection on the constant functions on $X_\tau$ we get that 
$$(n+1-k) \sum_{j=0}^n \Delta^-_{\tau,(0,j)} - \Delta^-_{\tau,0},$$
is the projection of the eigenfunctions with eigenvalue $\frac{n+1-k}{n-k}$. Therefore,
$$\Delta^+_{\tau,0} \left( (n+1-k) \sum_{j=0}^n \Delta^-_{\tau,(0,j)} \right)=\frac{n+1-k}{n-k} \left((n+1-k) \sum_{j=0}^n \Delta^-_{\tau,(0,j)} - \Delta^-_{\tau,0} \right) ,$$
which yields
$$\Delta^+_{\tau,0} \left(I - (n+1-k) \sum_{j=0}^n \Delta^-_{\tau,(0,j)} \right) = \Delta^+_{\tau,0} +\frac{n+1-k}{n-k}   \Delta^-_{\tau,0}  - \frac{(n+1-k)^2}{n-k} \sum_{j=0}^n \Delta^-_{\tau,(0,j)} .$$
Therefore we have that
\begin{dmath*}
 \kappa \left( \Vert \phi_\tau \Vert^2 - (n+1-k)^2 \sum_{j=0}^n \Vert \Delta^-_{\tau,(0,j)} \phi_\tau \Vert^2 \right) \geq \\
 \left\langle \left( \Delta^+_{\tau,0} +\frac{n+1-k}{n-k}   \Delta^-_{\tau,0}  - \frac{(n+1-k)^2}{n-k} \sum_{j=0}^n \Delta^-_{\tau,(0,j)} \right) \phi_\tau, \phi_\tau \right\rangle \geq \\
 \lambda \left( \Vert \phi_\tau \Vert^2 - (n+1-k)^2 \sum_{j=0}^n \Vert \Delta^-_{\tau,(0,j)} \phi_\tau \Vert^2 \right). 
 \end{dmath*}
 Summing the above inequalities on all $\tau \in \Sigma (k-1)$ and using the equalities:
 $$(k+1)! \Vert \phi \Vert^2 = \sum_{\tau \in \Sigma(k-1)} \Vert \phi_\tau \Vert^2 ,$$
$$k! \left\langle\Delta^-_k \phi , \phi \right\rangle= \sum_{\tau \in \Sigma(k-1)} \left\langle \Delta^-_{\tau,0} \phi_\tau , \phi_\tau \right\rangle ,$$
$$  k! \langle \Delta^+_k \phi ,  \phi \rangle + k!k \Vert \phi \Vert^2 =  \sum_{\tau \in \Sigma (k-1)} \langle \Delta^+_{\tau,0} \phi_\tau, \phi_\tau \rangle ,$$
$$k! \langle \Delta^-_{(k,j)} \phi, \phi \rangle = \sum_{\tau \in \Sigma (k-1)} \langle \Delta^-_{\tau, (0,j)} \phi_\tau, \phi_\tau \rangle ,$$ 
(see lemma \ref{LocalizNorm1}, corollary \ref{LocalizNorm3} and proposition \ref{partite localization of Delta^-} ), yields (after dividing by $k!$):
\begin{dmath*}
 \kappa \left\langle \left( (k+1) I -  (n+1-k)^2 \sum_{j=0}^n  \Delta^-_{(k,j)} \right) \phi, \phi \right\rangle  \geq 
 \left\langle \left( \Delta^+_{k} +k I +  \frac{n+1-k}{n-k}   \Delta^-_{k}  - \frac{(n+1-k)^2}{n-k} \sum_{j=0}^n \Delta^-_{(k,j)} \right) \phi, \phi \right\rangle \geq 
 \lambda \left\langle \left( (k+1) I -  (n+1-k)^2 \sum_{j=0}^n  \Delta^-_{(k,j)} \right) \phi, \phi \right\rangle . 
 \end{dmath*}
Subtracting 
$$ \dfrac{\lambda + \kappa}{2}  \left\langle \left( (k+1) I -  (n+1-k)^2 \sum_{j=0}^n  \Delta^-_{(k,j)} \right) \phi, \phi \right\rangle ,$$
from the above inequality yields
\begin{dmath*}
\left\vert \left\langle \left( \Delta^+_{k} +(k - (k+1)  \dfrac{\lambda + \kappa}{2})  I +  \frac{n+1-k}{n-k}   \Delta^-_{k}  - ( \frac{(n+1-k)^2}{n-k} - (n+1-k)^2 \dfrac{\lambda + \kappa}{2} ) \sum_{j=0}^n \Delta^-_{(k,j)} \right) \phi, \phi \right\rangle \right\vert \leq \dfrac{\kappa - \lambda}{2} \left\langle \left( (k+1) I -  (n+1-k)^2 \sum_{j=0}^n  \Delta^-_{(k,j)} \right) \phi, \phi \right\rangle .
\end{dmath*}
This in turn yields 
\begin{dmath*}
\left\Vert \Delta^+_{k} +\frac{n+1-k}{n-k}   \Delta^-_{k} + (k - (k+1)  \dfrac{\lambda + \kappa}{2})  I  - ( \frac{(n+1-k)^2}{n-k} - (n+1-k)^2 \dfrac{\lambda + \kappa}{2} ) \sum_{j=0}^n \Delta^-_{(k,j)}  \right\Vert \leq (k+1) \dfrac{\kappa - \lambda}{2} .
\end{dmath*}
\end{proof}

The above theorem, combined with lemma \ref{spectral bound in the partite case}, yields the following bound based only on the smallest positive eigenvalues of the links.

\begin{corollary}
\label{norm bound - n+1 partite case}
Let $X$  be a pure $n$-dimensional, $(n+1)$-partite, weighted simplicial complex such that all the links of $X$ of dimension $>0$ are connected. Fix $0 \leq k \leq n-1$, if there is $ \lambda > \frac{k}{k+1}$ such that 
$$\bigcup_{\tau \in \Sigma (k-1)} Spec (\Delta_{\tau, 0}^+) \setminus \lbrace 0 \rbrace \subseteq [\lambda, \infty),$$
then 
\begin{dmath*}
\left\Vert \Delta^+_{k} +\frac{n+1-k}{n-k}   \Delta^-_{k} - (  \dfrac{2+(n-k)(1-\lambda)}{2})  I  - ( \frac{(n+1-k)^2}{n-k} - (n+1-k)^2 \dfrac{2+(n-k)(1-\lambda)}{2} ) \sum_{j=0}^n \Delta^-_{(k,j)}  \right\Vert \leq (k+1)(n+1-k) \dfrac{1 - \lambda}{2} ,
\end{dmath*}
where $\Vert. \Vert$ denotes the operator norm.
\end{corollary}

\begin{proof}
Notice that every link is a $(n+1-k)$-partite complex and that by lemma \ref{spectral bound in the partite case} we have that in the notations of the above theorem:
$$\kappa \leq 1+(n+1-k)(1-\lambda), \kappa - \lambda \leq (n+1-k)(1 - \lambda) . $$
Therefore, we get 
\begin{dmath*}
\left\Vert \Delta^+_{k} +\frac{n+1-k}{n-k}   \Delta^-_{k} - (  \dfrac{2+(n-k)(1-\lambda)}{2})  I  - ( \frac{(n+1-k)^2}{n-k} - (n+1-k)^2 \dfrac{2+(n-k)(1-\lambda)}{2} ) \sum_{j=0}^n \Delta^-_{(k,j)}  \right\Vert \leq (k+1)(n+1-k) \dfrac{1 - \lambda}{2} .
\end{dmath*}
\end{proof}

\section{$k$-graph of $X$ and random walks}

In order to make sense of some of the later results regarding Cheeger type inequalities and mixing, we introduce the some terminology regarding graphs the arise from the simplicial complex $X$ and random walks on these graphs.

\subsection{Random walk on finite graphs}

Let $A$ be a finite set. A \textit{random walk} on $A$ is a map $\mu : A \times A \rightarrow [0, 1]$ such that for every $x \in A$ we have $\sum_{y \in A} \mu (x,y) =1$.\\
A \textit{stationary measure} of a random walk $\mu$ is a strictly positive function
$\nu : A \rightarrow  \mathbb{R}^+$,
such that for every $x,y \in A$ we have $\nu (x) \mu (x,y) = \nu (y) \mu (y,x)$. A random walk that has at least one stationary measure is called \textit{reversible}.
For a random walk $\mu$ and $j \in \mathbb{N}$, we can define a random walk $\mu^{*j}$ as
$$\mu^{*j} (x,y) = \sum_{(x,x_1,...,x_{j-1},y) \in A \times A \times ... \times A} \mu (x,x_1) \mu (x_1,x_2) ... \mu (x_{j-1},y).$$
Next, let $G$ be a graph $G=(V,E)$ (possibly with loops and multiple edges) and a let $c$ be a strictly positive function $c : E \rightarrow \mathbb{R}^+$. Such $c$ is called a conductance function on $G$. Define a random walk on $G$ with respected to $c$ as follows: for $v \in V, e \in E$ denote $v \in e$ if $v$ is an end of $e$. Define
$$\forall v \in V, \nu (v) = \sum_{e \in E, v \in e} c (e),$$
$$\mu (u,e) = \begin{cases}
\dfrac{c(e)}{\nu (v)} & v \in e \\
0 & v \notin e
\end{cases} .$$
The idea behind this definition is that $\mu (v,e)$ is the probability to choose $e$ when standing at $v$ and walk on $e$ to the other end of it. This is a somewhat refined version of the random walk on the set $A$ presented earlier, because since $G$ might have multiple edges, we get that an edge (or a loop) is not uniquely defined by its end vertices. Notice that for every $v \in V$ we have
$$\sum_{e \in E, v \in e} \mu (v,e) = 1,$$
and that for every $u,v \in V$ and every $e = (u,v)$ we have that
$$\nu (v) \mu (e) = \nu (u) \mu (e).$$
Thus $\nu$ is a stationary measure of $\mu$. \\

\subsection{Coarse path random walk, Coarse path conductance}
All the definitions regarding random walks above were more or less standard, our next definitions are (as far as we know) new. The main idea here is to break the graph into several pieces and define notions like random walk or conductance with respect to these pieces. \\
Let $G=(V,E)$ a weighted graph. For $e \in E$ and $v \in V$ such that $v \in e$ denote 
$$e \setminus v = \begin{cases}
v & \text{if } e \text{ is a loop} \\
u & e = (u,v), u \neq v
\end{cases}.$$
\begin{definition}
Let $G = (V,E)$ be a graph. For $j \in \mathbb{N}$, a $j+1$-tuple $(v,e_1,e_2,...,e_j) \in V \times E^j$ is called a path with in $G$ if for every $1 \leq i \leq j$ we have that $v_i \in e_i$, where the $v_i$'s are defined recursively as 
$$v_1 = v, \forall 2 \leq i \leq j, v_{i+1} = e_{i} \setminus v_{i}.$$ 
Denote $Path (G)$ to be the set of all paths in $G$.
\end{definition}

\begin{definition}
Let $G = (V,E)$ be a graph and let $U \subseteq V, E_1 \subseteq E,...,E_j \subseteq E$. Define
$$Path (U, E_1,...,E_j ) = \lbrace (u,e_1,...,e_j) \in Path (G) : u \in U, e_1 \in E_1,...,e_j \in E_j \rbrace.$$
\end{definition}

\begin{definition}
\label{CoarseProbAndCond}
Let $G = (V,E)$ and $c: E \rightarrow \mathbb{R}^+$ as above and let $\mu, \nu$ the random walk and stationary measure associated to $c$. Define the coarse path random walk  as follows:
$$path \mu : \bigcup_{j \in \mathbb{N}} \mathcal{P} (V) \times \mathcal{P} (E)^j \rightarrow \mathbb{R}_{\geq 0},$$
For $U \subseteq V, E_1,...,E_j \subseteq E$ define $\mu_{coarse} (U,E_1,...,E_j)$ as follows:
\begin{itemize}
\item  If $Path (U, E_1,...,E_j ) \neq \emptyset$, define
$$path \mu (U,E_1,...,E_j) = \sum_{(v,e_1,...,e_j) \in Path (U, E_1,...,E_j ) } \mu (v, e_1) \mu (v_2,e_2) ... \mu (v_{j}, e_j),$$
where $v_1 = v, \forall 2 \leq i \leq j, v_{i+1} = e_{i} \setminus v_{i}.$
\item If $Path (U, E_1,...,E_j ) = \emptyset$, define $\mu_{coarse} (U,E_1,...,E_j) =0$ (in particular, if $U = \emptyset$ or $E_i = \emptyset$ for some $i$, then $\mu_{coarse} (U,E_1,...,E_j) =0$). 
\end{itemize}  
Also define the coarse path conductance as
 $$path c : \bigcup_{j \in \mathbb{N}} \mathcal{P} (V)  \times (\mathcal{P} (E))^j \rightarrow \mathbb{R}_{\geq 0},$$
$$path c (U,E_1,...,E_j) = \sum_{v \in U} \nu (v) path \mu (\lbrace v \rbrace, E_1,...,E_j ).$$
\end{definition}


\subsection{Inner connectivity of  subgraph}

Let $G=(V,E)$ be a graph and let $c: E \rightarrow \mathbb{R}^+$ be a conductance function of $G$. Denote as before $\mu, \nu$ the random walk and stationary measure induced by $c$. Given a subgraph of $G$, $G' = (V',E')$, we want a measure on how much $G'$ in connected within itself. Define the following:


\begin{definition}
\label{InnerCon}
Let $G = (V,E),c: E \rightarrow \mathbb{R}^+,\mu, \nu$ as above. Let $G' = (V',E')$ be a subgraph of $G$. Define the inner connectivity constant of $G'$ in $G$ as 
$$h_{inner} (G' ; G) = \begin{cases}
\dfrac{path c (V', E',E')}{path c (V',E')} & E' \neq \emptyset \\
0 & E' = \emptyset 
\end{cases} .$$
\end{definition}

By definition for every $G'$ we have that $0 \leq h_{inner} (G' ; G) \leq 1$. It is worth noting the following interpretation to $h_{inner} (G' ; G) $: when $E' \neq \emptyset$ we get by definition of $path c (V',E',E'), path c (V',E')$ that:
$$h_{inner} (G' ; G)  = \dfrac{\sum_{v \in V'} \nu (v) \sum_{e \in E', v \in e} \mu (v,e) \sum_{e' \in E', e \setminus v \in e'} \mu (e \setminus v,e')}{\sum_{v \in V'} \nu (v) \sum_{e \in E', v \in e} \mu (v,e)}.$$
Denote $\nu (V') = \sum_{v \in V'} \nu (v)$ and divide both numerator and denominator by $\nu (V')$:
$$h_{inner} (G' ; G)  = \dfrac{\sum_{v \in V'} \dfrac{\nu (v)}{\nu (V')} \sum_{e \in E', v \in e} \mu (v,e) \sum_{e' \in E', e \setminus v \in e'} \mu (e \setminus v,e')}{\sum_{v \in V'} \dfrac{\nu (v)}{\nu (V')} \sum_{e \in E', v \in e} \mu (v,e)}.$$
Consider the space
$$\Omega =  \lbrace (v,e,e') \in Path (G) : v \in V' \rbrace,$$
with the probability measure
$$ P(v,e,e') = \dfrac{\nu (v)}{\nu (V')} \mu (v, e) \mu (e \setminus v, e').$$
Then for 
$$A_1 = \lbrace (v,e,e') \in \Omega : e \in E' \rbrace,$$
$$A_2 = \lbrace (v,e,e') \in \Omega : e' \in E' \rbrace,$$
we have
$$h_{inner} (G' ; G) = \dfrac{P(A_1 \cap A_2)}{P(A_1)} = P (A_2 \vert A_1).$$
Therefore, $h_{inner} (G' ; G) $ is exactly the conditional probability that a $2$-step random walk with a randomly chosen origin $v \in V'$ stays in $G'$ given that the $1$-step random walk a randomly chosen origin $v \in V'$ stayed in $G'$. 

\subsection{The $k$-graph of $X$}

\begin{definition}
\label{kGraph}
Let $X$ be an $n$-dimensional simplicial complex. For $-1 \leq k \leq n$, define the $k$-graph of $X$ denoted as $X_k = (V_k, E_k)$ as following: 
\begin{itemize}
\item The vertices of $X_k$ are $k$-dimensional (unordered) simplices of X, i.e., $V_k = X^{(k)}$.
\item For $\lbrace u_0,...,u_k \rbrace, \lbrace v_0,...,v_k \rbrace \in V_k$ we have that $(\lbrace u_0,...,u_k \rbrace, \lbrace v_0,...,v_k \rbrace ) \in E_k$, if there is a $k+1$ simplex $ \lbrace w_0,...,w_{k+1} \rbrace \in X^{(k+1)}$ such that $\lbrace w_0,...,w_k \rbrace = \lbrace u_0,...,u_k \rbrace, \lbrace w_1,...,w_{k+1} \rbrace = \lbrace v_0,...,v_k \rbrace$.

\end{itemize}
 
\end{definition}

\begin{remark}
We remark that in the above definition there are three special cases worth noting:
\begin{enumerate}
\item When $k=n$ then $V_n = X^{(n)}$ and $E_n = \emptyset$.
\item When $k=0$ then $X_0$ is just the $1$-skeleton of $X$.
\item When $k=-1$ then $X_{-1}$ is a graph with a single vertex and a loop for every $v \in X^{(0)}$. 
\end{enumerate}
\end{remark}

\begin{definition}
Let $X$ be an $n$-dimensional simplicial complex. For $l \geq 0$ and non empty sets $U_0,...,U_l \subseteq X^{(0)}$, define a simplicial complex $X (U_0,...,U_l)$ on dimension $\min \lbrace l, n \rbrace$ that will be a sub complex of $X$ as follows: for every $0 \leq j \leq n$ we have that $\lbrace u_0,...,u_j \rbrace \in (X (U_0,...,U_l))^{(j)}$ if:
$$\lbrace u_0,...,u_j \rbrace \in X^{(j)}$$ 
and 
$$ \exists \lbrace i_0,...,i_j \rbrace \subseteq \lbrace 0,...,l \rbrace , \vert \lbrace i_0,...,i_j \rbrace \vert = j+1, u_0 \in U_{i_0},...,u_j \in U_{i_j}.$$ 
Note that if $l <n$ then by the above definition $(X (U_0,...,U_l))^{(l+1)} = ... = (X (U_0,...,U_l))^{(n)} = \emptyset$. 
$X (U_0,...,U_l)$ will be called the simplicial complex spanned by $U_0,...,U_l$. 
\begin{definition}
Let $X$ be an $n$-dimensional simplicial complex, $l \geq 0$ and $U_0,...,U_l \subseteq X^{(0)}$ non empty sets.  For $k \leq \min \lbrace l, n \rbrace$ define $X_k (U_0,...,U_l) = (X (U_0,...,U_l))_k$, i.e., $X_k (U_0,...,U_l)$ is the $k$-graph of the simplicial complex spanned by $U_0,...,U_l$. Notice that since $X (U_0,...,U_l)$ is a sub complex of $X$, we get that $X_k (U_0,...,U_l)$ is a subgraph of $X_k$. 
\end{definition}
\end{definition}

\begin{definition}
\label{WeightOfU_i's}
Let $\emptyset \neq A \subseteq \bigcup_{k=-1}^n X^{(k)}$, define
$$m (A) = \sum_{\tau \in A} m (\tau ) .$$
For $0 \leq k \leq n$ and $U_0,...,U_k \subseteq X^{(0)}$ define
$$m (U_0,...,U_k) = m( V_k (U_0,...,U_k)) =  \sum_{\tau \in V_k (U_0,...,U_k)} m (\tau).$$
\end{definition}

\subsection{Random walk on $X_k$}

Next, we'll define a random walk on $X_k$ given a weight function on $X$.
\begin{definition}
\label{RandomWalkXk}
Let $X$ be a pure $n$-dimensional simplicial complex with a weight function $m$. For $-1 \leq k \leq n-1$ we define a conductance function $c_k$ induced by $m$ on $X_k$ in the following way:
\begin{itemize}
\item For $0 \leq k \leq n-1$ define
$$\forall (\tau_1, \tau_2) \in E_k, c_k ( (\tau_1, \tau_2) ) = m( \tau_1 \cup \tau_2 ).$$
\item As noted before $X_{-1}$ is a graph with a single vertex and a loop for every $v \in X^{(0)}$. Define $c_{-1} (v) = m (v)$ for every $v \in X^{(0)}$. 
\end{itemize}
Using the conductance function $c_k$ we can define a random walk $\mu_k$ and a stationary measure $\nu_k$ induced by $m$:
\begin{itemize}
\item For $0 \leq k \leq n-1$,
$$\forall \tau_1, \tau_2 \in V_k, \mu_k ((\tau_1,\tau_2) ) =
 \begin{cases}
0 & (\tau_1, \tau_2) \notin E_k \\
\dfrac{m( \tau_1 \cup \tau_2 )}{ (k+1) m( \tau_1)} & (\tau_1, \tau_2) \in E_k
\end{cases},$$
$$\forall \tau \in V_k, \nu_k (\tau ) = (k+1) m (\tau).$$
\item For $k=-1$, we distinguish between the probability to walk on each loop: for every loop indexed by $v \in X^{(0)}$, the probability to travel trough the loop $v$ is $\mu_{-1} (v) = \frac{m(v)}{m (\emptyset)}$ and we define $\nu (\emptyset ) = m (\emptyset)$. 
\end{itemize} 
\end{definition}

As in definition \ref{CoarseProbAndCond} we can define the coarse path random walk with respect to $\mu_k$, which we'll denote as 
$$path \mu_{k} : \bigcup_{j \in \mathbb{N}} \mathcal{P}(V_k) \times \mathcal{P} (E_k)^j \rightarrow \mathbb{R}_{\geq 0}.$$
We can also define the coarse path conductance with respect to $\mu_k$, which we'll denote as 
$$path c_{k} : \bigcup_{j \in \mathbb{N}}  \mathcal{P}(V_k) \times \mathcal{P} (E_k)^j \rightarrow \mathbb{R}_{\geq 0}.$$

We'll see that $path \mu_{k}, path c_{k}$ is easy to compute in some special cases:

\begin{proposition}
\label{case(-1) For pathc}
Let  $U_0,...,U_l \subseteq X^{(0)}$, then:
$$ path \mu_{-1} (\emptyset, E_{-1} (U_0),...,E_{-1} (U_l )) = \dfrac{m(U_0) ... m(U_l)}{m(\emptyset)^{l+1}},$$
$$path c_{-1} (\emptyset,E_{-1} (U_0),...,E_{-1} (U_l )) = \dfrac{m(U_0) ... m(U_l)}{m(\emptyset)^{l}}.$$
($m(U_i)$ were defined in definition \ref{WeightOfU_i's}).
\end{proposition}

\begin{proof}
Recall that the random walk on $X_{-1}$ is a random walk on a graph with one vertex and a loop for every $u \in X^{(0)}$, where $\mu_{-1} (u) = \frac{m(u)}{m(\emptyset)}$. By this we get
$$ path \mu_{-1} (\emptyset, E_{-1} (U_0),...,E_{-1} (U_l )) = \dfrac{m(U_0) ... m(U_l)}{m(\emptyset)^{l+1}}.$$
Also, recall that by definition
$$path c_{-1} (\emptyset,E_{-1} (U_0),...,E_{-1} (U_l ))  = \nu (\emptyset) path \mu_{-1} (\emptyset, E_{-1} (U_0),...,E_{-1} (U_l )), $$
which yields the second equality.
\end{proof}

\begin{proposition}
\label{pathc_k for U_0,...,U_(k+1)}
For $0 \leq k \leq n-1$ and any $U_0,...,U_{k+1} \subseteq X^{(0)}$ we have that
$$path c_{k} (V_k (U_0,...,U_k), E_k( U_0,...,U_{k+1}) ) =(k+1) m (U_0,...,U_{k+1} ) .$$
($m (U_0,...,U_{k+1} )$ was defined in definition \ref{WeightOfU_i's}).
\end{proposition}

\begin{proof}
Let $\lbrace u_0,...,u_k \rbrace \in V_k (U_0,...,U_k)$ such that $u_0 \in U_0,..., u_k \in U_k$. Recall that
$$\nu_k (\lbrace u_0,...,u_k \rbrace) = (k+1) m ( \lbrace u_0,...,u_k \rbrace ).$$
For every $u \in U_{k+1}$ such that $\lbrace u_0,...,u_k,u \rbrace \in X^{(k)}$ there are $k+1$ elements $\sigma \in V_k (U_0,...,U_{k+1})$ such that $\lbrace u_0,...,u_k \rbrace \cup \sigma =\lbrace u_0,...,u_k,u \rbrace$. For every such $\sigma$, we have that $(\lbrace u_0,...,u_k \rbrace , \sigma) \in E_k (U_0,...,U_{k+1})$ and
$$\mu_k ((\lbrace u_0,...,u_k \rbrace , \sigma) ) = \dfrac{m(\lbrace u_0,...,u_k,u \rbrace)}{(k+1) m(\lbrace u_0,...,u_k \rbrace) }.$$
Therefore \\ \\
$\nu_k (\lbrace u_0,...,u_k \rbrace ) path \mu_{k} (\lbrace u_0,...,u_k \rbrace,E_k (U_0,...,U_{k+1})) =$
\begin{flushright}
$ \sum_{u \in U_{k+1}, \lbrace u_0,...,u_k, u \rbrace \in X^{(k+1)}} (k+1) m (\lbrace u_0,...,u_k, u \rbrace).$
\end{flushright}
Summing on all $\lbrace u_0,...,u_k \rbrace \in V_k (U_0,...,U_k)$ such that $u_0 \in U_0,..., u_k \in U_k$, we get that 
$$path c_{k} (V_k (U_0,...,U_k), E_k( U_0,...,U_{k+1}) ) = (k+1) m (U_0,...,U_{k+1} ) .$$
\end{proof}

Next we'll define the inner connectivity of $U_0,...,U_k$:
\begin{definition}
\label{Inner connectivity for U_i's}
Let $0 \leq k \leq n-1$ and $U_0,...,U_k \subseteq X^{(0)}$. Define 
$$h_{inner}^k (U_0,...,U_k) = \begin{cases}
h_{inner} (X_{k-1} (U_0,...,U_k) ; X_{k-1} ) & U_0 \neq \emptyset,...,U_k \neq \emptyset \\
0 & \text{otherwise}
\end{cases}.$$
where $h_{inner} (X_{k-1} (U_0,...,U_k) ; X_{k-1} ) $ as in definition \ref{InnerCon}.
\end{definition}

\begin{remark}
For the cases $k=0$, $h_{inner}^0 (U_0)$ is easy to compute from proposition \ref{case(-1) For pathc}):
$$pathc_{-1} (V_{-1} (U_0), E_{-1} (U_0) ) =  pathc_{-1} (\emptyset, E_{-1} (U_0) ) = m(U_0),$$
$$pathc_{-1} (V_{-1} (U_0), E_{-1} (U_0), E_{-1} (U_0) ) =  pathc_{-1} (\emptyset, E_{-1} (U_0), E_{-1} (U_0)  ) = \dfrac{m(U_0)^2}{m (\emptyset)}.$$
Therefore
$$h_{inner}^0 (U_0) = \dfrac{m(U_0)}{m(\emptyset)}.$$
\end{remark}

\begin{proposition}
Let $1 \leq k \leq n-1$ and $U_0,...,U_k \subseteq X^{(0)}$ such that $U_0 \neq \emptyset,...,U_k \neq \emptyset$. Then
$$pathc_{k-1} (V_{k-1} (U_0,...,U_k), E_{k-1} (U_0,...,U_k) ) = k(k+1)m(U_0,...,U_k).$$
\end{proposition}

\begin{proof}
For $k=0$, see remark above. Assume that $1 \leq k \leq n-1$. Note that
$$pathc_{k-1} (V_{k-1} (U_0,...,U_k), E_{k-1} (U_0,...,U_k) ) = $$
$$= \sum_{i=0}^k pathc_{k-1} (V_{k-1} (U_0,..., \widehat{U_i},...,U_k), E_{k-1} (U_0,...,U_k) ).$$
By proposition \ref{pathc_k for U_0,...,U_(k+1)} for every $i$ we have 
$$pathc_{k-1} (V_{k-1} (U_0,..., \widehat{U_i},...,U_k), E_{k-1} (U_0,...,U_k) ) = k m (U_0,...,U_k).$$
Therefore we get
$$pathc_{k-1} (V_{k-1} (U_0,...,U_k), E_{k-1} (U_0,...,U_k) ) = k(k+1)m(U_0,...,U_k).$$
\end{proof}

By the definition of $h_{inner}^k (U_0,...,U_k)$ we get the following corollary:
\begin{corollary}
\label{h inner with m}
Let $1 \leq k \leq n-1$ and $U_0,...,U_k \subseteq X^{(0)}$ such that $U_0 \neq \emptyset,...,U_k \neq \emptyset$. Then
$$h_{inner}^k (U_0,...,U_k) = \dfrac{pathc_{k-1} (V_{k-1} (U_0,..,U_k), E_{k-1} (U_0,..,U_k), E_{k-1} (U_0,..,U_k))}{k(k+1)m(U_0,...,U_k)}.$$

\end{corollary}

\section{High order Cheeger-type inequalities}

High order Cheeger-type inequalities were already shown for simplicial complexes in \cite{PRT} and \cite{GS}. Our treatment differs from previous work since we introduce another factor to the definition of the Cheeger-type constant (see $h_{inner}^k$ below), while previous works only concerned Cheeger-type constants similar to our $h_{out}^k$. 
 
\subsection{$1$-dimensional Cheeger inequality from a new perspective}
First, let us rephrase the idea behind the Cheeger inequality in graphs introducing some new terminology. We'll start with recalling the Cheeger inequality. \\

Let $G=(V,E)$ be a graph with no isolated vertices. For every $v \in V$, denote $d(v)$ as the valency of $v$, i.e., $d(v) = \vert \lbrace (v,u) \in E \rbrace \vert$. For a set $\emptyset \neq U \subseteq V$, denote $\vert U \vert = \sum_{u \in U} d(v)$. For $\emptyset \neq V_1 \subseteq V, \emptyset \neq V_2 \subseteq V$, denote 
$$E(V_1, V_2 ) = \lbrace (v_1,v_2 ) \in E : v_1 \in V_1, v_2 \in V_2 \rbrace.$$
The Cheeger constant of $G$ is defined as
$$h (G) = \min \lbrace \dfrac{\vert E(U, V \setminus U) \vert}{\vert U \vert } : \emptyset \neq U \subset V, \vert U \vert \leq \dfrac{1}{2} \vert V \vert \rbrace.$$
The Cheeger inequality states the following: 
\begin{proposition}
If $G$ is connected and $\lambda (G)$ is the smallest positive eigenvalue of the graph Laplacian $\Delta^+$ of $G$, then $h(G) \geq \frac{1}{2} \lambda (G)$.
\end{proposition}

Next, we'll rework the statement in a new terminology. Let $X$ be a $1$-dimensional weighted simplicial complex (i.e., $X$ is a weighted graph) with a weight function $m$. In this case, $X = X_0$ (recall definition \ref{kGraph}) and $m$ induced a random walk on $X$ as in definition \ref{RandomWalkXk}. For $\emptyset \neq U \subset X^{(0)}$ denote
$$m(U) = \sum_{u \in U} m(u).$$
\begin{definition}
Let $X$ be a $1$-dimensional weighted simplicial complex with a weight function $m$ and no isolated vertices. Let $\emptyset \neq U \subseteq X^{(0)}$. Define
$$h_{out}^0 (U) = \dfrac{m (U, X^{(0)} \setminus U)}{m(U)},$$
where $m(U), m(U,X^{(0)} \setminus U)$ as in definition \ref{WeightOfU_i's}.
\end{definition}

If $m$ is the homogeneous weight that assigns each edge in $X$ the weight $1$ we get that 
$$h_{out}^0 (U) = \ \dfrac{\vert E(U, X^{(0)} \setminus U) \vert}{\vert U \vert } .$$
Therefore, when $m$ is the homogeneous weight, $h(X)$ can be written as
$$h(X) = \min \lbrace h_{out}^0 (U) : \emptyset \neq U \subset X^{(0)}, \vert U \vert \leq \dfrac{1}{2} \vert X^{(0)} \vert \rbrace.$$
In general, we'll write
$$h(X) = \min \lbrace h_{out}^0 (U) : \emptyset \neq U \subset X^{(0)}, m(U) \leq \dfrac{1}{2} m (X^{(0)} ) \rbrace.$$
Next, we'll want to get rid of the condition $m(U) \leq \frac{1}{2} m (X^{(0)} ) $ by altering the definition of $h$. Recall that by definition \ref{Inner connectivity for U_i's} and the remark that follows it, we have for every $U \subseteq X^{(0)}$ that
$$h_{inner}^0 (U) = \dfrac{m(U)}{m (\emptyset)} = \dfrac{m(U)}{m (X^{(0)} )}.$$

\begin{definition}
Let $X$ be a $1$-dimensional weighted simplicial complex with a weight function $m$ and no isolated vertices. Define
$$h^0 (X) = \max \lbrace \varepsilon \geq 0 : \forall \emptyset \neq U \subseteq X^{(0)}, \varepsilon (h_{inner}^0 (U)) +h_{out}^0 (U) \geq \varepsilon  \rbrace.$$ 
\end{definition}

The next proposition shows that a lower bound on $h^0 (X)$ is more informative than a lower bound on $h (X)$.

\begin{proposition}
\label{h compared to h^0}
Let $X$ be a $1$-dimensional weighted simplicial complex with a weight function $m$ and no isolated vertices. Then $2h (X) \geq h^0 (X)$.
\end{proposition}

\begin{proof}
Let $\varepsilon >0$ such that 
$$\forall \emptyset \neq U \subseteq X^{(0)}, \varepsilon (h_{inner}^0 (U)) +h_{out}^0 (U) \geq \varepsilon.$$
For every $\emptyset \neq U \subset X^{(0)}$ with $m(U) \leq \frac{1}{2} m (X^{(0)} ) $ we get $h_{inner}^0 (U) \leq \frac{1}{2}$ and therefore $h_{out}^0 (U) \geq \frac{1}{2} \varepsilon$. Since this is true for any such $\varepsilon$ we get that $2h (X) \geq h^0 (X)$.
\end{proof}

Next, we'll prove the Cheeger inequality in terms of $h^0 (X)$ (we basically use the standard proof of the Cheeger inequality, only write it in terms of $h^0 (X)$).

\begin{proposition}
\label{1-dim Cheeger}
Let $X$ be a $1$-dimensional connected weighted simplicial complex with a weight function $m$. Denote by $\lambda$ the smallest positive eigenvalue of $\Delta_0^+$ on $X$, then $h^0 (X) \geq \lambda$.
\end{proposition}

\begin{proof}
Fix $\emptyset \neq U \subseteq X^{(0)}$. Take $\chi_U \in C^{0} (X,\mathbb{R})$ as the indicator function of $U$. Then by the fact that $X$ is connected and that $\Delta^-$ is the projection on the space of constant functions, we get that 
$$\Vert  d \chi_U \Vert^2 \geq \lambda (\Vert \chi_U \Vert^2 -\Vert \Delta^- \chi_U \Vert^2).$$
Easy computations give
$$\Vert  d \chi_U \Vert^2 = m (U, X^{(0)} \setminus U ),$$
$$\Vert \chi_U \Vert^2 = m(U),$$
$$\Vert \Delta^- \chi_U \Vert^2 = \dfrac{m(U)^2}{m(\emptyset)} .$$
Therefore 
$$m (U, X^{(0)} \setminus U ) \geq \lambda m (U) - \lambda \dfrac{m(U)^2}{m(\emptyset)}.$$
Which yields
$$\lambda h_{inner}^0 (U) + h_{out}^0 (U) \geq \lambda,$$
Since this is true for every $\emptyset \neq U \subseteq X^{(0)}$, we get that $h^0 (X) \geq \lambda$.
\end{proof}

\subsection{High order Cheeger-type inequalities}

Let $X$ be a pure $n$-dimensional simplicial complex with a weight function $m$. We shall prove an analogue to the connection between $h^0 (X)$ and the smallest positive eigenvalue of $\Delta_0^+$. We'll start by defining $h_{out}^k (U_0,...,U_k)$ for any $0 \leq k \leq n-1$ and $U_0,...,U_k \subseteq X^{(0)}$ (we recall that $h_{inner}^k (U_0,...,U_k)$ was already defined in definition \ref{Inner connectivity for U_i's}).
\begin{definition}
Let $X$ be a pure $n$-dimensional simplicial complex with a weight function $m$. For $0 \leq k \leq n-1$ and any pairwise disjoint non empty sets $U_0,...,U_k \subseteq X^{(0)}$ define
$$h_{out}^k (U_0,...,U_k) = \begin{cases}
0 & X^{(0)} \setminus \bigcup_{i=0}^k U_i = \emptyset \\
\dfrac{m(U_0,...,U_k, X^{(0)} \setminus \bigcup_{i=0}^k U_i )}{m (U_0,...,U_k)} & \text{otherwise}
\end{cases}.$$
\end{definition}

Next, we will define $h^k (X)$:
\begin{definition}
Let $X$ be a pure $n$-dimensional simplicial complex with a weight function $m$. For $0 \leq k \leq n-1$ define $h^k (X)$ as follows:
$$h^k (X) = \max \lbrace \varepsilon \geq 0 : \forall \emptyset \neq U_0,..., \emptyset \neq U_k \subseteq X^{(0)} \text{ pairwise disjoint}$$
$$ \left(\dfrac{k}{k+1} +\varepsilon \right) h_{inner}^k (U_0,...,U_k) +\dfrac{1}{k+1} h_{out}^k (U_0,...,U_k) \geq \varepsilon  \rbrace.$$ 
\end{definition}

Our goal is to give a lower bound on $h^k$ given a large enough Laplacian spectral gap for the links of $X$. This statement will be made precise in theorem \ref{high dim Cheeger ineq. 1} and corollary \ref{high dim Cheeger ineq. 2} below. \\ \\

\begin{definition}
For $0 \leq k \leq n-1$ and any pairwise disjoint sets $U_0,...,U_k \subseteq X^{(0)}$ define the indicator $k$-form of $(U_0,...,U_k)$, denoted $\chi_{U_0,...,U_k} \in C^k (X,\mathbb{R} )$, as:
$$\chi_{U_0,...,U_k} ((u_0,...,u_k )) = \begin{cases}
sgn (\pi) & \exists \pi \in Sym (\lbrace 0,...,k \rbrace), \forall i, u_{\pi (i)} \in U_i \\
0 & \text{otherwise}
\end{cases}.$$
We remark that if for some $i_0$ we have that $U_{i_0} = \emptyset$ then $\chi_{U_0,...,U_k}$ is the zero $k$-form.
\end{definition}

\begin{lemma}
\label{chi U_0,...,U_k calc1}
For $0 \leq k \leq n-1$ and any pairwise disjoint sets $U_0,...,U_k \subseteq X^{(0)}$ we have that
\begin{enumerate}
\item $$\Vert \chi_{U_0,...,U_k} \Vert^2 =
\begin{cases}
 m(U_0,...,U_k) & U_0 \neq \emptyset,...,U_k \neq \emptyset \\
 0 & \text{otherwise}
\end{cases}
.$$
\item Denote $U_{k+1}= X^{(0)} \setminus \bigcup_{i=0}^k U_i$, then $d \chi_{U_0,...,U_k} = (-1)^{k+1} \chi_{U_0,...,U_{k+1}}$ and
$$\Vert d \chi_{U_0,...,U_k} \Vert^2 = \begin{cases}
 m(U_0,...,U_{k+1}) & U_0 \neq \emptyset,...,U_{k+1} \neq \emptyset \\
 0 & \text{otherwise}
\end{cases}
.$$

\end{enumerate}
\end{lemma}

\begin{proof}
\begin{enumerate}
\item If there is $0 \leq i_0 \leq k$ such that $U_{i_0} = \emptyset$ then $\chi_{U_0,...,U_k} \equiv 0$ and the statement in the lemma is trivial. Assume that for all $0 \leq i \leq k$, $U_i  \neq \emptyset$. Note that for every $\sigma \in \Sigma (k)$ the following holds:
$$\chi_{U_0,...,U_k} (\sigma)^2 = \begin{cases}
1 & \sigma \in \Sigma (k) \cap \left( \bigcup_{\pi \in Sym (\lbrace 0,...,k \rbrace)} U_{\pi (0)} \times ... \times U_{\pi (k)} \right) \\
0 & \text{otherwise}
\end{cases}.$$
Therefore
$$\Vert \chi_{U_0,...,U_k} \Vert^2 = \sum_{\sigma \in \Sigma (k) \cap \left( \bigcup_{\pi \in Sym (\lbrace 0,...,k \rbrace)} U_{\pi (0)} \times ... \times U_{\pi (k)} \right)} \dfrac{m(\sigma)}{(k+1)!} = $$
$$ =  \sum_{\sigma \in \Sigma (k) \cap \left( U_{0} \times ... \times U_{k} \right)} m(\sigma) = m (U_0,...,U_k).$$
\item If there is $0 \leq i_0 \leq k$ such that $U_{i_0} = \emptyset$ then $\chi_{U_0,...,U_k} \equiv 0$ and therefore $d \chi_{U_0,...,U_k} \equiv 0$ and we are done. Assume that for all $0 \leq i \leq k$, $U_i \neq \emptyset$.
Since $d \chi_{U_0,...,U_k}$ is antisymmetric, it is enough to show that given $(u_0,...,u_{k+1}) = \sigma \in \Sigma(k+1)$ such that 
$$\forall 0 \leq i_1 < i_2 \leq k+1, u_{i_1} \in U_{j_1}, u_{i_2} \in U_{j_2} \Rightarrow j_1 \leq j_2,$$
we have that 
$$d \chi_{U_0,...,U_k} (\sigma ) = \begin{cases}
(-1)^{k+1} & \forall 0 \leq i \leq k+1, u_i \in U_i \\
0 & \text{otherwise}
\end{cases} .$$
There are free cases where $$\forall 0 \leq i \leq k+1, \chi_{U_0,...,U_k} (\sigma_i) = 0$$
which yields  $d \chi_{U_0,...,U_k} (\sigma )=0$:
\begin{itemize}
\item If $u_{k}, u_{k+1} \in U_{k+1}$, then for every $0 \leq i \leq k+1$, $\sigma_i$ contains at least one vertex in $U_{k+1}$ and therefore $\chi_{U_0,...,U_k} (\sigma_i) = 0$.
\item If there is $i_0$ such that $u_{i_0}, u_{i_0 +1}, u_{i_0 +2} \in U_{j_0}$, then for every $0 \leq i \leq k+1$, $\sigma_i$ contains at least two vertices in $U_{j_0}$ and therefore $\chi_{U_0,...,U_k} (\sigma_i) = 0$.\\
\item If there are $i_0,i_1$ such that $u_{i_0}, u_{i_0 +1} \in U_{j_0}, u_{i_1}, u_{i_1 +1} \in U_{j_1}$, then for every $0 \leq i \leq k+1$, $\sigma_i$ contains at least two vertices in $U_{j_0}$ or in $U_{j_1}$ and therefore $\chi_{U_0,...,U_k} (\sigma_i) = 0$. 
\end{itemize}
Excluding the above cases we are left with only two options - either
$$\forall 0 \leq i \leq k+1, u_i \in U_i ,$$
or there is some $0 \leq i_0 \leq k$ such that
$$\forall 0 \leq i \leq i_0, u_i \in U_i \text { and } \forall i_0+1 \leq i \leq k+1,u_i \in U_{i-1}. $$
Note that the first case can not occur if $U_{k+1} = \emptyset$. In the first case: 
$$\forall 0 \leq i \leq k+1, u_i \in U_i ,$$
implies that $\forall 0 \leq i \leq k,\chi_{U_0,...,U_k} (\sigma_i) = 0$ and $\chi_{U_0,...,U_k} (\sigma_{k+1}) = 1$. Therefore $d \chi_{U_0,...,U_k} (\sigma ) = (-1)^{k+1}$. \\
In the second case,  
$$\forall 0 \leq i \leq i_0, u_i \in U_i \text { and } \forall i_0+1 \leq i \leq k+1,u_i \in U_{i-1}, $$
implies that $\chi_{U_0,...,U_k} (\sigma_i) = 0$ for $i \neq i_0,i_0+1$ and that $\chi_{U_0,...,U_k} (\sigma_{i_0}) = \chi_{U_0,...,U_k} (\sigma_{i_0+1}) = 1$. Therefore,
$$ d \chi_{U_0,...,U_k} (\sigma) = (-1)^{i_0} + (-1)^{i_0+1} =0.$$
We conclude that $d \chi_{U_0,...,U_k} = \chi_{U_0,...,U_{k+1}}$. Since we showed $d \chi_{U_0,...,U_k} = \chi_{U_0,...,U_{k+1}}$, we have that $\Vert d \chi_{U_0,...,U_k} \Vert^2 = \Vert \chi_{U_0,...,U_{k+1}} \Vert^2$ and therefore the equality for $\Vert d \chi_{U_0,...,U_k} \Vert^2$ is true by 1. .
\end{enumerate}
\end{proof}

\begin{lemma}
\label{chi U_0,...,U_k calc2}
For $1 \leq k \leq n-1$ and any pairwise disjoint sets $U_0,...,U_k \subseteq X^{(0)}$ we have for every $0 \leq i_0 \leq k$ and every $\tau \in \Sigma(k-1) \cap \left( U_{0} \times ... \times \widehat{U}_{i_0} \times ... \times  U_{k} \right)$ that
$$ \delta \chi_{U_0,...,U_k} (\tau ) = \sum_{\sigma \in \Sigma (k) \cap \left( U_{0} \times ... \times U_k \right), \tau \subset \sigma} \dfrac{m(\sigma)}{m (\tau)} (-1)^{i_0}.$$
We also have that
$$\Vert \delta \chi_{U_0,...,U_k} \Vert^2 = \dfrac{1}{k} pathc_{k-1} (V_{k-1} (U_0,..,U_k), E_{k-1} (U_0,..,U_k), E_{k-1} (U_0,..,U_k)).$$
\end{lemma}

\begin{proof}
Fix $0 \leq i_0 \leq k$ and let $\tau = (u_0,...,u_{k-1} ) \in \Sigma (k)$ such that 
$$\forall 0 \leq i < i_0, u_i \in U_i, \forall i_0 \leq i \leq k-1, u_{i} \in U_{i+1}.$$
Then 
\begin{align*}
\delta \chi_{U_0,...,U_k} (\tau) = \sum_{v \in \Sigma (0), v \tau \in \Sigma (k)} \dfrac{m(v \tau)}{m (\tau)} \chi_{U_0,...,U_k} (v \tau ) = \\
 \sum_{v \in \Sigma (0) \cap U_{i_0}, v \tau \in \Sigma (k)} \dfrac{m(v \tau)}{m (\tau)} \chi_{U_0,...,U_k} (v \tau ) = \\
 \sum_{\sigma \in \Sigma (k) \cap \left( U_{0} \times ... \times U_k \right), \tau \subset \sigma} \dfrac{m(\sigma)}{m (\tau)} (-1)^{i_0}.
 \end{align*}
 It is easy to see that the support of $\delta \chi_{U_0,...,U_k}$ is contained in $\Sigma (k-1) \cap \bigcup_{i=0}^k \bigcup_{\pi \in Sym (\lbrace 0,...,k \rbrace)} U_{\pi (0)} \times ... \times \widehat{U}_{\pi (i)} \times ... \times  U_{\pi (k)}$.
Therefore, by the computation carried above we get that:
\begin{dmath*}
\Vert \delta \chi_{U_0,...,U_k} \Vert^2 = \sum_{\tau \in \Sigma (k-1)} \dfrac{m(\tau )}{k!} \delta \chi_{U_0,...,U_k} (\tau )^2 = \\
  \sum_{i=0}^k \sum_{\tau \in \Sigma (k-1) \cap \left( U_0 \times ... \times \widehat{U_i} \times ... \times U_k \right)} m( \tau ) \left(\sum_{\sigma \in \Sigma (k) \cap \left( U_{0} \times ... \times U_k \right), \tau \subset \sigma} \dfrac{m(\sigma)}{m (\tau)} \right)^2. 
  \end{dmath*}
Note that in the above sum can be rewritten as
$$ \sum_{\tau \in V_{k-1} (U_0,...,U_k)}  m( \tau ) \left(\sum_{\sigma \in V_k ( U_{0} ,..., U_k ), \tau \subset \sigma} \dfrac{m(\sigma)}{m (\tau)} \right)^2$$
For every $\tau \in V_{k-1} (U_0,...,U_k)$ and every $\sigma V_k ( U_{0} ,..., U_k )$ such that $\tau \subset \sigma$, there are exactly $k$ elements $\tau ' \in  V_{k-1} (U_0,...,U_k)$ such that $\sigma = \tau \cup \tau '$. By definition of $E_{k-1} (U_0,...,U_k)$ we have that
\begin{dmath*}
\sum_{\tau \in V_{k-1} (U_0,...,U_k)}  m( \tau ) \left(\sum_{\sigma \in V_k ( U_{0} ,..., U_k ), \tau \subset \sigma} \dfrac{m(\sigma)}{m (\tau)} \right)^2 =
 \sum_{\tau \in V_{k-1} (U_0,...,U_k)}  m( \tau ) \left(\sum_{(\tau , \tau ') \in E_{k-1} (U_0,...,U_k)} \dfrac{1}{k} \dfrac{m(\tau \cup \tau')}{ m (\tau)} \right)^2 =
\sum_{\tau \in V_{k-1} (U_0,...,U_k)} \sum_{ (\tau , \tau_1) \in E_{k-1} (U_0,...,U_k)} \dfrac{m(\tau \cup \tau_1)}{k} \sum_{ (\tau , \tau_2) \in E_{k-1} (U_0,...,U_k)} \dfrac{1}{k}\dfrac{m(\tau \cup \tau_2 )}{m (\tau)} =
\sum_{\tau_1 \in V_{k-1} (U_0,...,U_k)} \sum_{ (\tau,  \tau_1) \in E_{k-1} (U_0,...,U_k)} \dfrac{m(\tau \cup \tau_1)}{k} \sum_{ (\tau , \tau_2) \in E_{k-1} (U_0,...,U_k)} \dfrac{1}{k}\dfrac{m(\tau \cup \tau_2 )}{m (\tau)} = 
 \sum_{\tau_1 \in V_{k-1} (U_0,...,U_k)} m (\tau_1) \sum_{ (\tau,  \tau_1) \in E_{k-1} (U_0,...,U_k)} \dfrac{1}{k}\dfrac{m(\tau \cup \tau_1)}{m(\tau_1)} \sum_{ (\tau , \tau_2) \in E_{k-1} (U_0,...,U_k)} \dfrac{1}{k}\dfrac{m(\tau \cup \tau_2 )}{m (\tau)} .
\end{dmath*}
Recall that 
$$\nu_{k-1} (\tau_1 ) = k m (\tau_1),\mu_{k-1} ((\tau_1, \tau)) =\frac{1}{k}\frac{m(\tau \cup \tau_1)}{m(\tau_1)}, \mu_{k-1} ((\tau, \tau_2)) =\frac{1}{k}\frac{m(\tau \cup \tau_2)}{m(\tau)},$$ 
and therefore we get
$$ \sum_{\tau_1 \in V_{k-1} (U_0,...,U_k)} \dfrac{1}{k} \nu_{k-1} (\tau_1) \sum_{ (\tau,  \tau_1) \in E_{k-1} (U_0,...,U_k)} \mu_{k-1} ((\tau_1, \tau)) \sum_{ (\tau , \tau_2) \in E_{k-1} (U_0,...,U_k)} \mu_{k-1} ((\tau, \tau_2)) = $$
$$ = \dfrac{1}{k} pathc_{k-1} (V_{k-1} (U_0,...,U_k), E_{k-1} (U_0,...,U_k), E_{k-1} (U_0,...,U_k)).$$

\end{proof}

As a corollary of the two above lemmas we get that:
\begin{corollary}
\label{h inner, h out as norms of d chi, delta chi}
For  $0 \leq k \leq n-1$ and any pairwise disjoint sets $U_0,...,U_k \subseteq X^{(0)}$ we have that:
$$h^k_{out} = \dfrac{\Vert d \chi_{U_0,...,U_k} \Vert^2 }{\Vert \chi_{U_0,...,U_k} \Vert^2 },$$
and that
$$h^k_{inner} =(k+1) \dfrac{\Vert \delta \chi_{U_0,...,U_k} \Vert^2 }{\Vert \chi_{U_0,...,U_k} \Vert^2 } .$$ 
\end{corollary}

\begin{proof}
The first equality is obvious from lemma \ref{chi U_0,...,U_k calc1}. For the second equality - the case $k=0$ is shown in the proof of proposition \ref{1-dim Cheeger}. For $k \geq 1$, recall that by corollary \ref{h inner with m} we have that
$$h_{inner}^k (U_0,...,U_k) = \dfrac{pathc_{k-1} (V_{k-1} (U_0,..,U_k), E_{k-1} (U_0,..,U_k), E_{k-1} (U_0,..,U_k))}{k(k+1)m(U_0,...,U_k)}.$$
Therefore the equality follows from lemmas \ref{chi U_0,...,U_k calc1}, \ref{chi U_0,...,U_k calc2}.
\end{proof}

\begin{theorem}
\label{high dim Cheeger ineq. 1}
Let $X$ be a pure $n$-dimensional weighted simplicial complex such that all the links of $X$ of dimension $>0$ are connected. For any $0 \leq k \leq n-1$, if there is $\varepsilon >0$ such that 
$$\bigcup_{\tau \in \Sigma (k-1)} Spec (\Delta_{\tau,0}^+) \setminus \lbrace 0 \rbrace \subseteq [\dfrac{k}{k+1} + \varepsilon, \infty),$$
then $h^k (X) \geq \varepsilon$.
\end{theorem}

\begin{proof}
The case $k=0$ was already proven in proposition \ref{1-dim Cheeger} (note that although the assumption in proposition \ref{1-dim Cheeger} was that $X$ is $1$-dimensional, the proof follows through in the $n$-dimensional case). Assume that $1 \leq k \leq n-1$. Let $U_0,...,U_k \subset X^{(0)}$ be non empty pairwise disjoint sets. 
By lemma \ref{SpecGapLocalToGlobal1} we have that
$$  \Vert d \chi_{U_0,...,U_k} \Vert^2  \geq (k+1) \Vert \chi_{U_0,...,U_k} \Vert^2  \varepsilon - \left( \dfrac{k}{k+1} + \varepsilon \right) \Vert \delta \phi \Vert^2.$$
Therefore
$$ \dfrac{1}{k+1} \left( \dfrac{k}{k+1} + \varepsilon \right) \dfrac{\Vert \delta \phi \Vert^2}{\Vert \chi_{U_0,...,U_k} \Vert^2}+\dfrac{1}{k+1}\dfrac{\Vert d \chi_{U_0,...,U_k} \Vert^2}{\Vert \chi_{U_0,...,U_k} \Vert^2}  \geq \varepsilon ,$$
By corollary \ref{h inner, h out as norms of d chi, delta chi} this gives
$$\left(\dfrac{k}{k+1} + \varepsilon \right) h_{inner}^k (U_0,...,U_k) +\dfrac{1}{k+1} h_{out}^k (U_0,...,U_k) \geq \varepsilon,$$
and since this is true for any $U_0,...,U_k \subset X^{(0)}$ be non empty pairwise disjoint sets, we get that $h^k (X) \geq \varepsilon$.
 
\end{proof}

Now we are ready to state exactly and prove theorem \ref{Cheeger inequality - section 2} stated above: 

\begin{corollary}
\label{high dim Cheeger ineq. 2}
Let $X$ be a pure $n$-dimensional weighted simplicial complex such that all the links of $X$ of dimension $>0$ are connected. If there is $\varepsilon >0$ such that 
$$\bigcup_{\tau \in \Sigma (n-2)} Spec (\Delta_{\tau,0}^+) \setminus \lbrace 0 \rbrace \subseteq [\dfrac{n-1}{n} + \varepsilon, \infty),$$
then for every $0 \leq k \leq n-1$, there is $\varepsilon_k (\varepsilon ) > 0$ such that $h^k (X) \geq \varepsilon_k$.
\end{corollary}

\begin{proof}
For $k= n-1$, take $\varepsilon_{n-1} = \varepsilon$ and apply theorem \ref{high dim Cheeger ineq. 1}. Assume $0 \leq k \leq n-2$.
Denote $f(x) = 2- \frac{1}{x}$. By corollary \ref{SpectralGapDescent2} for every $0 \leq k \leq n-2$ we have that
$$\bigcup_{\tau \in \Sigma (k-1)} Spec (\Delta_{\tau, 0}^+) \setminus \lbrace 0 \rbrace  \subseteq [ f^{n-k-1} (\dfrac{n-1}{n} + \varepsilon) , \infty).$$
Take $\varepsilon_k = f^{n-k-1} (\frac{n-1}{n} + \varepsilon) - \frac{k}{k+1} >0$ and apply theorem \ref{high dim Cheeger ineq. 1} to complete the proof.
\end{proof}

\section{Mixing}

This section owes its existence to the work done in \cite{P} studying mixing for simplicial complexes given spectral gaps of Laplacians. Our results differ from those of \cite{P} since the Laplacians we use are different (our Laplacians are normalized with respect to the weight function $m$). We also prove a mixing result for $(n+1)$-partite simplicial complexes based on local spectral expansion, that is very different in nature to the one proven in \cite{EGL}. \\ \\

\subsection{From Laplacians to coarse paths}

Let $X$ be a pure $n$-dimensional weighted simplicial complex such that all the links of $X$ of dimension $>0$ are connected. For any integer $0 \geq k$ and $U_0,...,U_k \subseteq X^{(0)}$ disjoint, non empty sets define a $k$-form $\chi_{U_0,...,U_k} \in C^{k} (X, \mathbb{R})$ as in the previous section:
$$\chi_{U_0,...,U_k} ((u_0,...,u_k)) = \begin{cases}
sgn (\pi ) & \exists \pi \in Sym (\lbrace 0,...,k \rbrace), u_{\pi (0)} \in U_0,...,u_{\pi (k)} \in U_k \\
0 & \text{otherwise}
\end{cases} .$$
Define a projection $\mathbb{P}_{U_0,...,U_k} : C^{k} (X, \mathbb{R}) \rightarrow C^{k} (X, \mathbb{R})$ as 
$$\mathbb{P}_{U_0,...,U_k} (\phi ) = \vert \chi_{U_0,...,U_k} \vert \phi, \forall \phi \in C^{k} (X,\mathbb{R}),$$ 
(the multiplication above is point-wise, i.e., for every $\tau \in \Sigma (k), (\vert \chi_{U_0,...,U_k} \vert \phi) (\tau) = \vert \chi_{U_0,...,U_k} (\tau ) \vert \phi (\tau)$).  \\
For any $k+1$-tuple $U_i,...,U_{i+k}$ define
$$\Sigma (k) (U_i,...,U_{i+k} ) = \Sigma (k) \cap \left( \bigcup_{\pi \in Sym (\lbrace i,...,i+k \rbrace)} U_{\pi (i)} \times ... \times U_{\pi (i+k)}. \right)$$
Define the sign function with respect to $U_0,...,U_l$ as
$$sgn (U_0,...,U_l) : \bigcup_{k=0}^{\min \lbrace l,n\rbrace} \bigcup_{i=0}^{l-k}  \Sigma (k) (U_i,...,U_{i+k} ) \rightarrow \lbrace 1, -1\rbrace,$$
As 
$$ \forall (u_0,...,u_k) \in \Sigma (k) (U_i,...,U_{i+k} ), sgn (U_0,...,U_l) ((u_0,...,u_k)) = sgn (\pi),$$
where $\pi \in Sym (\lbrace i,...,i+k \rbrace) $ such that $(u_0,...,u_k) \in U_{\pi (i)} \times ... \times U_{\pi (i+k)}$. Later we'll just write $sgn$ instead of $sgn (U_0,...,U_l) $ whenever it is clear what are $U_0,...,U_l$. \\

Before the next definition, we remark that in this section we allow some abuse of notation, referring to ordered simplices as unordered ones for random walk proposes:
\begin{remark}
Let  $0 \leq k \leq n-1$. We'll allow the following abuse of notation: 
\begin{enumerate}
\item Let $\tau , \tau ' \in \Sigma (k)$ such that $\tau, \tau'$ are not the same simplex up to reordering. If that there is $\sigma \in \Sigma (k+1)$ such that $\tau, \tau' \subset \sigma$, denote $m (\tau \cup \tau ' ) = m(\sigma)$.
\item  Let $\tau , \tau ' \in \Sigma (k)$ such that $\tau, \tau'$ are not the same simplex up to reordering, denote 
$$\mu_k (\tau , \tau ') = \begin{cases}
\dfrac{m(\tau \cup \tau ' )}{(k+1) m (\tau)} & \exists \sigma \in \Sigma (k+1), \tau, \tau' \subset \sigma \\
0 & \text{otherwise}
\end{cases}.$$
\item For $(u_0,...,u_k) = \tau \in \Sigma (k)$, and $E_1,...,E_l \subset E_k$, denote 
$$path \mu_k (\tau, E_1,...,E_l) =path \mu_k (\lbrace u_0,...,u_k \rbrace, E_ 1,...,E_l) .$$
\end{enumerate}
\end{remark}

\begin{definition}
Let $X$, $U_0,...,U_l$ as above. For $0 \leq k < \min \lbrace l,n\rbrace$ Define the $k$ random walk form $\Psi_k (U_0,...,U_l) \in C^k (X,\mathbb{R})$ as follows:
\begin{itemize}
\item For $\sigma \in \Sigma (k) (U_0,...,U_k) $, define \\ \\
$\Psi_k (U_0,...,U_l) (\sigma) =$
$$sgn(\sigma ) path \mu_k (\sigma , E_k (U_0,...,U_{k+1}),...,E_k (U_{l-k-1},...,U_{l})). $$
\item For $\sigma \notin \Sigma (k) (U_0,...,U_k) $, define $\Psi_k (U_0,...,U_l) (\sigma) = 0$.
\end{itemize}
Also define $\Psi_{-1} (U_0,...,U_l) \in C^{-1} (X, \mathbb{R})$ as
$$\Psi_{-1} (U_0,...,U_l) (\emptyset) = \dfrac{m(U_0) ... m(U_l)}{m (\emptyset)^{l+1}} .$$
\end{definition} 

\begin{proposition}
\label{concatenation of Psi}
Let $X$, $U_0,...,U_l$ as above. For $0 \leq k < \min \lbrace l,n\rbrace - 1$ we have that for every $\sigma \in \Sigma (k) (U_0,...,U_k)$ that
$$\Psi_k (U_0,...,U_l) (\sigma )= sgn (\sigma ) \sum_{\sigma' \in \Sigma (k) \cap \left( U_1 \times ... \times U_{k+1} \right)}  \mu_k (\sigma , \sigma ') \Psi_k (U_1,...,U_l) (\sigma ').$$
\end{proposition}

\begin{proof}
Note that for every $\sigma \in \Sigma (k) (U_0,...,U_k)$ we have that
\begin{dmath*}
\sum_{\sigma' \in \Sigma (k) \cap \left( U_1 \times ... \times U_{k+1} \right)}  \mu_k (\sigma , \sigma ') \Psi_k (U_1,...,U_l) (\sigma ') =
 \sum_{\sigma' \in \Sigma (k) \cap \left( U_1 \times ... \times U_{k+1} \right)}  \mu_k (\sigma , \sigma ') path \mu_k (\sigma ' , E_k (U_1,...U_{k+2}),..., E_k (U_{l-k-1},...U_{l}))  =
 path \mu_k (\sigma , E_k (U_0,...U_{k+1}),... ,E_k (U_{l-k-1},...U_{l})) ,
 \end{dmath*}
and the claim in the proposition follows.

\end{proof}

\begin{lemma}
\label{mixing lemma 1}
Let $X$ be a pure $n$-dimensional weighted simplicial complex such that all the links of $X$ of dimension $>0$ are connected. Let $0 \leq k \leq n-1$ and $k <l$. Then for any disjoint sets $U_0,...,U_l \subseteq X^{(0)}$, we have
$$\dfrac{(-1)^{(k+1)(l-k)}}{(k+1)^{l-k-1}} \left( \prod_{i =0}^{l-k-1} \mathbb{P}_{U_i,...,U_{k+i}} \Delta^+_k \right) \chi_{U_{l-k},...,U_l} = \Psi_k (U_0,...,U_l).$$
\end{lemma}

\begin{proof}
We'll prove using induction on $l$. For $l=k+1$ we have the form $\mathbb{P}_{U_0,...,U_k} \Delta^+_k \chi_{U_1,...,U_{k+1}}$. By the definition of $\mathbb{P}_{U_0,...,U_k}$ is is clear that 
$$ \forall \sigma \in \left( \Sigma (k) \setminus \Sigma (k) (U_0,...,U_k) \right), \mathbb{P}_{U_0,...,U_k} \Delta^+_k \chi_{U_1,...,U_{k+1}} (\sigma) = 0.$$
Since $\mathbb{P}_{U_0,...,U_k} \Delta^+_k \chi_{U_1,...,U_{k+1}}$ is antisymmetric, it is enough to prove that for every $\sigma \in \Sigma (k) \cap \left( U_{0} \times ... \times U_{k} \right)$, the following holds:
$$(-1)^{k+1} \Delta^+_k \chi_{U_1,...,U_{k+1}} = \Psi_k (U_0,...,U_{k+1} ) (\sigma).$$
Recall that 
$$\Delta^+_k \chi_{U_1,...,U_{k+1}} (\sigma) = \chi_{U_1,...,U_{k+1}} (\sigma)-\sum_{\begin{array}{c}
{\scriptstyle v \in\Sigma(0)}\\
{\scriptstyle v \sigma \in \Sigma (k+1)}\end{array} }\sum_{0\leq i\leq k}(-1)^{i}\frac{m(v\sigma)}{m(\sigma)} \chi_{U_1,...,U_{k+1}} (v\sigma_{i}) .$$
Since $\sigma \in \left( U_{0} \times ... \times U_{k} \right)$ we get that $\chi_{U_1,...,U_{k+1}} (\sigma) =0$. Also, since for every $i >0$, and every $v$, $v \sigma_i $ contains a vertex in $U_0$ we get that $\chi_{U_1,...,U_{k+1}} (v \sigma_i) =0$. For $\sigma_0$ and every $v \in \Sigma (0)$ such that $v \sigma_0 \in \Sigma (k)$ we have that
$$\chi_{U_1,...,U_{k+1}} (v \sigma_0) = \begin{cases}
(-1)^k & v \in U_{k+1} \\
0 & \text{otherwise}
\end{cases} .$$
Therefore
\begin{dmath*}
\Delta^+_k \chi_{U_1,...,U_{k+1}} (\sigma)  = - \sum_{\begin{array}{c}
{\scriptstyle v \in U_{k+1}}\\
{\scriptstyle v \sigma \in \Sigma (k+1)}\end{array} }  \frac{m(v\sigma)}{m(\sigma)} \chi_{U_1,...,U_{k+1}} (v\sigma_{0}) = 
 (-1)^{k+1} \sum_{\begin{array}{c}
{\scriptstyle v \in U_{k+1}}\\
{\scriptstyle  \sigma v \in \Sigma (k+1)}\end{array} }  \frac{m(\sigma v)}{m(\sigma)} = 
  (-1)^{k+1} \sum_{\begin{array}{c}
{\scriptstyle v \in U_{k+1}}\\
{\scriptstyle  \sigma v \in \Sigma (k+1)}\end{array} }  \sum_{i=0}^k \frac{m(\sigma \cup \sigma_i v)}{(k+1) m(\sigma)} = (-1)^{k+1} \sum_{\begin{array}{c}
{\scriptstyle v \in U_{k+1}}\\
{\scriptstyle  \sigma v \in \Sigma (k+1)}\end{array} }  \sum_{i=0}^k \mu_k (\sigma, \sigma_i v) .
\end{dmath*}
Denote $\sigma = (u_0,...,u_k)$ and note that $(\lbrace u_0,...,u_k \rbrace , \lbrace v_0,...,v_k \rbrace ) \in E_k (U_0,...,U_{k+1})$ if and only if there is some $v \in U_{k+1}$ and $0 \leq i \leq k$ such that 
$$\lbrace v_0,...,v_k \rbrace = \lbrace u_0,..., \widehat{u_i},...,u_k,v \rbrace.$$
Therefore we have that 
\begin{dmath*}
(-1)^{k+1} \sum_{\begin{array}{c}
{\scriptstyle v \in U_{k+1}}\\
{\scriptstyle  \sigma v \in \Sigma (k+1)}\end{array} }  \sum_{i=0}^k \mu_k (\sigma, \sigma_i v)   = (-1)^{k+1} path \mu_k (\sigma , E_k (U_0,...,U_{k+1} )) = (-1)^{k+1} \Psi_k (U_0,...,U_{k+1}) (\sigma ) .
\end{dmath*}
This finishes the case $l=k+1$. Assume now that the claim is true for  $l-1$. This assumption implies that:
$$\dfrac{(-1)^{(k+1)(l-1-k)}}{(k+1)^{l-1-k-1}} \left( \prod_{i =1}^{l-k-1} \mathbb{P}_{U_i,...,U_{k+i}} \Delta^+_k \right) \chi_{U_{l-k},...,U_l} = \Psi_k (U_1,...,U_l).$$
This yields
$$\dfrac{(-1)^{(k+1)(l-k)}}{(k+1)^{l-k-1}} \left( \prod_{i =0}^{l-k-1} \mathbb{P}_{U_i,...,U_{k+i}} \Delta^+_k \right) \chi_{U_{l-k},...,U_l} =\dfrac{(-1)^{k+1}}{k+1} \mathbb{P}_{U_0,...,U_{k}} \Delta^+_k \Psi_k (U_1,...,U_l).$$
Therefore we are left to prove that
$$\dfrac{(-1)^{k+1}}{k+1} \mathbb{P}_{U_0,...,U_{k}} \Delta^+_k \Psi_k (U_1,...,U_l) = \Psi_k (U_0,...,U_l).$$
Again, by the definition of $\mathbb{P}_{U_0,...,U_k}$ is is clear that 
$$ \forall \sigma \in \left( \Sigma (k) \setminus \Sigma (k) (U_0,...,U_k) \right), \mathbb{P}_{U_0,...,U_k} \Delta^+_k \Psi_k (U_1,...,U_l) (\sigma) = 0.$$
Therefore, it is enough to prove that for every $\sigma \in \Sigma (k) \cap \left( U_{0} \times ... \times U_{k} \right)$, the following holds:
$$\dfrac{(-1)^{k+1}}{k+1} \Delta^+_k \Psi_k (U_1,...,U_l ) (\sigma ) = \Psi_k (U_0,...,U_{l} ) (\sigma).$$
By the same considerations of the $l=k+1$ case, we get that
\begin{dmath*}
 \Delta^+_k \Psi_k (U_1,...,U_l ) (\sigma )  = - \sum_{\begin{array}{c}
{\scriptstyle v \in U_{k+1}}\\
{\scriptstyle  v \sigma \in \Sigma (k+1)}\end{array} }   \frac{m(v\sigma)}{m(\sigma)} \Psi_k (U_1,...,U_l ) (v \sigma_0) = 
 (-1)^{k+1} \sum_{\begin{array}{c}
{\scriptstyle v \in U_{k+1}}\\
{\scriptstyle   \sigma v \in \Sigma (k+1)}\end{array} }   \frac{m(\sigma v)}{m(\sigma)} \Psi_k (U_1,...,U_l ) (\sigma_0 v) = 
  (-1)^{k+1} (k+1) \sum_{\begin{array}{c}
{\scriptstyle v \in U_{k+1}}\\
{\scriptstyle   \sigma v \in \Sigma (k+1)}\end{array} }   \frac{m(\sigma v)}{(k+1) m(\sigma)} \Psi_k (U_1,...,U_l ) (\sigma_0 v) =
 (-1)^{k+1} (k+1) \sum_{ \sigma' \in \Sigma (k) \cap \left( U_1 \times ... \times U_{k+1} \right)}   \frac{m(\sigma \cup \sigma ')}{(k+1) m(\sigma)} \Psi_k (U_1,...,U_l ) (\sigma ') =
 (-1)^{k+1} (k+1) \sum_{ \sigma' \in \Sigma (k) \cap \left( U_1 \times ... \times U_{k+1} \right)}   \mu_k (\sigma, \sigma ') \Psi_k (U_1,...,U_l ) (\sigma ') =
 (-1)^{k+1} (k+1) \Psi_k (U_0,...,U_l ) (\sigma ).
 \end{dmath*}
Where the last equality is due to proposition \ref{concatenation of Psi}. Therefore 
$$ \Delta^+_k \Psi_k (U_1,...,U_l ) (\sigma )  = (-1)^{k+1} (k+1) \Psi_k (U_0,...,U_l ) (\sigma ),$$
and we are done.

\end{proof}

\begin{lemma}
\label{mixing lemma 2}
Let $X$ be a pure $n$-dimensional weighted simplicial complex such that all the links of $X$ are of dimension $>0$ connected. Let $1 \leq k \leq n-1$ and $k <l$. Then for any disjoint, non empty sets $U_0,...,U_l \subseteq X^{(0)}$, we have for
$$\left( \prod_{i =0}^{l-k-1} \mathbb{P}_{U_i,...,U_{k+i}} \Delta^-_k \right) \chi_{U_{l-k},...,U_l} \in C^{k} (X,\mathbb{R}), $$
that
$$ \forall \sigma \in \left( \Sigma (k) \setminus \Sigma (k)(U_0,...,U_k) \right), \left( \prod_{i =0}^{l-k-1} \mathbb{P}_{U_i,...,U_{k+i}} \Delta^-_k \right) \chi_{U_{l-k},...,U_l} (\sigma ) = 0,$$
and that for every $\sigma \in \Sigma (k) \cap \left( U_0 \times ... \times U_k \right)$, 
$$ \dfrac{(-1)^{(l-k)k}}{k^{l-k-1}} \left( \prod_{i =0}^{l-k-1} \mathbb{P}_{U_i,...,U_{k+i}} \Delta^-_k \right) \chi_{U_{l-k},...,U_l}  (\sigma ) = \Psi_{k-1} (U_1,...,U_l) (\sigma_0 ).$$

\end{lemma}

\begin{proof}
The proof is very similar to the proof of the former lemma, therefore will omit some details in the proof. The proof is by induction. Start with $l=k+1$.From the definition of $\mathbb{P}_{U_0,...,U_{k}}$ it is clear that 
$$ \forall \sigma \in \left( \Sigma (k) \setminus \Sigma (k)(U_0,...,U_k) \right),  \mathbb{P}_{U_0,...,U_{k}} \Delta^-_k \chi_{U_{1},...,U_{k+1}} (\sigma ) = 0.$$
Therefore, we are left to prove that for every $\sigma \in \Sigma (k) \cap \left( U_0 \times ... \times U_k \right)$, 
$$ (-1)^{k} \Delta^-_k \chi_{U_{1},...,U_{k+1}} (\sigma ) = \Psi_{k-1} (U_1,...,U_l) (\sigma_0 ).$$
Recall that 
$$\Delta^-_k \chi_{U_{1},...,U_{k+1}} (\sigma )  =  \sum_{i=0}^k (-1)^i  \sum_{ v \in\Sigma(0),  v\sigma_i \in\Sigma(k) }\dfrac{m(v\sigma_i)}{m(\sigma_i)}\chi_{U_{1},...,U_{k+1}} (v\sigma_i). $$
Therefore we get 
$$\Delta^-_k \chi_{U_{1},...,U_{k+1}} (\sigma ) = (-1)^{k} \sum_{v \in U_{k+1}, \sigma_0 v \in \Sigma (k)} \dfrac{m(\sigma_0 v)}{m(\sigma_0)}.$$
For $1 \leq i \leq k, \sigma = (u_0,...,u_k)$ denote $\sigma_{0i} = (u_1,...,\widehat{u_i},...,u_k)$. By this notation we have 
\begin{dmath*}
(-1)^{k} \sum_{v \in U_{k+1}, \sigma_0 v \in \Sigma (k)} \dfrac{m(\sigma_0 v)}{m(\sigma_0)} = (-1)^{k} \sum_{v \in U_{k+1}, \sigma_0 v \in \Sigma (k)} \sum_{i=1}^k \dfrac{m(\sigma_0 \cup \sigma_{0i} v)}{km(\sigma_0)} = 
  (-1)^{k} \sum_{v \in U_{k+1}, \sigma_0 v \in \Sigma (k)} \sum_{i=1}^k \mu_{k-1} (\sigma_0 , \sigma_{0i} v ) =
 (-1)^{k} \Psi_{k-1} (U_1,...,U_{k+1} ) (\sigma_0).
 \end{dmath*}
Assume that the claim is true for $l-1$. Denote 
$$\Phi =  \dfrac{(-1)^{(l-1-k)k}}{k^{l-1-k}} \left( \prod_{i =1}^{l-k-1} \mathbb{P}_{U_i,...,U_{k+i}} \Delta^-_k \right) \chi_{U_{l-k},...,U_l} .$$
Again, it is clear that
$$\mathbb{P}_{U_0,...,U_{k}} \Delta^-_k \Phi$$
vanishes outside $\Sigma (k) (U_0,...,U_k)$. Let $\sigma \in \Sigma (k) \cap \left( U_0 \times ... \times U_k \right)$. We are left to prove that
$$\dfrac{(-1)^{k}}{k} \Delta^-_k \Phi (\sigma ) = \Psi_{k-1} (U_1,...,U_l) (\sigma_0).$$
Recall that 
$$ \Delta^-_k \Phi (\sigma ) = \sum_{i=0}^k (-1)^i  \sum_{ v \in\Sigma(0),  v\sigma_i \in\Sigma(k) }\dfrac{m(v\sigma_i)}{m(\sigma_i)} \Phi (v\sigma_i). $$
By the induction assumption $\Phi$ vanishes outside $\Sigma (k) (U_1,...,U_{k+1})$ and therefore we get
$$ \Delta^-_k \Phi (\sigma ) = (-1)^k \sum_{ v \in U_{k+1} ,  \sigma_0 v \in\Sigma(k) }\dfrac{m(\sigma_0 v)}{m(\sigma_0 v)} \Phi (\sigma_0 v). $$
By the induction assumption $\Phi (\sigma_0 v) = \Psi_{k-1} (U_2,...,U_l) (\sigma_{01}  v)$ and therefore
\begin{dmath*}
\Delta^-_k \Phi (\sigma ) = (-1)^k \sum_{ v \in U_{k+1} ,  \sigma_0 v \in\Sigma(k) }\dfrac{m(\sigma_0 v)}{m(\sigma_0 v)} \Psi_{k-1} (U_2,...,U_l) (\sigma_{01}  v) = 
 (-1)^k k \sum_{ v \in U_{k+1} ,  \sigma_0 v \in\Sigma(k) }\dfrac{m(\sigma_0 \cup \sigma_{01} v)}{k m(\sigma_0 v)} \Psi_{k-1} (U_2,...,U_l) (\sigma_{01}  v) = 
 (-1)^k k \sum_{ v \in U_{k+1} ,  \sigma_0 v \in\Sigma(k) }\mu_{k-1} ( \sigma_0, \sigma_{01} v ) \Psi_{k-1} (U_2,...,U_l) (\sigma_{01}  v) =
 (-1)^k k \sum_{ \sigma' \in \Sigma (k-1) \cap \left( U_2 \times ... \times U_{k+1} \right) }\mu_{k-1} (\sigma_0, \sigma' ) \Psi_{k-1} (U_2,...,U_l) (\sigma ') =
 (-1)^k k \Psi_{k-1} (U_1,...,U_l) (\sigma_0),
 \end{dmath*}
and we are done.
\end{proof}

By the two above lemmas we get

\begin{corollary}
\label{pathc_k as inner product} 
Let $X$ be a pure $n$-dimensional weighted simplicial complex such that all the links of $X$ are of dimension $>0$ connected. Let $0 \leq k \leq n-1$ and $k <l$. For any disjoint sets $U_0,...,U_l \subseteq X^{(0)}$,  denote
$$pathc_k (U_0,...,U_l) = pathc_k (V_k (U_0,...,U_k), E_k (U_0,...,U_{k+1}) ,...,E_k (U_{l-k-1},...,U_{l})).$$
For any such $k,l$, we have that
$$\left\vert \left\langle \chi_{U_0,...,U_k}, \left( \prod_{i =0}^{l-k-1} \mathbb{P}_{U_i,...,U_{k+i}} \Delta^+_k \right) \chi_{U_{l-k},...,U_l} \right\rangle \right\vert =(k+1)^{l-1 - (k+1)} pathc_k (U_0,...,U_l) .$$
In the case $k \geq 1$, we also have
$$\left\vert \left\langle \chi_{U_0,...,U_k}, \left( \prod_{i =0}^{l-k-1} \mathbb{P}_{U_i,...,U_{k+i}} \Delta^-_k \right) \chi_{U_{l-k},...,U_l} \right\rangle \right\vert =k^{l-1-k} pathc_{k-1} (U_0,...,U_l) .$$
\end{corollary}

\begin{proof}
For $0 \leq k \leq n-1$ and $U_0,...,U_l \subseteq X^{(0)}$ as above.
Note that $\chi_{U_0,...,U_k}$ is supported on $\Sigma (k) (U_0,...,U_k)$ and that for every $k$-form $\phi \in C^{k} (X, \mathbb{R} )$ we have by antisymmetry of $k$-forms that
$$\left\langle \chi_{U_0,...,U_k} , \phi \right\rangle = \sum_{\sigma \in \Sigma (k) \cap \left( U_0 \times ... \times U_{k} \right) } m(\sigma) \chi_{U_0,...,U_k} (\sigma ) \phi (\sigma ) = $$
$$ = \sum_{\sigma \in \Sigma (k) \cap \left( U_0 \times ... \times U_{k} \right) } m(\sigma) \phi (\sigma ) .$$
Therefore, by lemma \ref{mixing lemma 1} 
\begin{dmath*}
\left\vert \left\langle \chi_{U_0,...,U_k}, \left( \prod_{i =0}^{l-k-1} \mathbb{P}_{U_i,...,U_{k+i}} \Delta^+_k \right) \chi_{U_{l-k},...,U_l} \right\rangle \right\vert =  
 \left\vert (-1)^{(k+1)(l-k)} (k+1)^{l-k-1} \sum_{\sigma \in \Sigma (k) \cap \left( U_0 \times ... \times U_{k} \right) } m(\sigma) \Psi_k (U_0,...,U_l) (\sigma ) \right\vert =
 \left\vert  (k+1)^{l-k-2} \sum_{\sigma \in \Sigma (k) \cap \left( U_0 \times ... \times U_{k} \right) } (k+1) m(\sigma) \Psi_k (U_0,...,U_l) (\sigma ) \right\vert = 
 \left\vert  {(k+1)^{l-k-2}} \sum_{\sigma \in \Sigma (k) \cap \left( U_0 \times ... \times U_{k} \right) } \nu_k (\sigma ) \Psi_k (U_0,...,U_l) (\sigma ) \right\vert  =(k+1)^{l-k-2}  pathc_k (U_0,...,U_l).
\end{dmath*}
Assume that $k \geq 1$, then by lemma \ref{mixing lemma 2} we have that
\begin{dmath*}
\left\vert \left\langle \chi_{U_0,...,U_k}, \left( \prod_{i =0}^{l-k-1} \mathbb{P}_{U_i,...,U_{k+i}} \Delta^-_k \right) \chi_{U_{l-k},...,U_l} \right\rangle \right\vert = 
 \left\vert (-1)^{(l-k)k} k^{l-k-1} \sum_{\sigma \in \Sigma (k) \cap \left( U_0 \times ... \times U_{k} \right) } m(\sigma) \Psi_{k-1} (U_1,...,U_l) (\sigma_0 ) \right\vert =
\left\vert k^{l-k-1} \sum_{\sigma \in \Sigma (k) \cap \left( U_0 \times ... \times U_{k} \right) } k m(\sigma_k) \dfrac{m(\sigma_k \cup \sigma_0)}{k m(\sigma_0)}  \Psi_{k-1} (U_1,...,U_l) (\sigma_0 ) \right\vert =
\left\vert k^{l-k-1} \sum_{\sigma \in \Sigma (k) \cap \left( U_0 \times ... \times U_{k} \right) } \nu_{k-1} (\sigma_k) \mu_{k-1} (\sigma_k, \sigma_0) \Psi_{k-1} (U_1,...,U_l) (\sigma_0 ) \right\vert = 
\left\vert k^{l-k-1} \sum_{\tau \in \Sigma (k-1) \cap \left( U_0 \times ... \times U_{k-1} \right) } \nu_{k-1} (\tau) \sum_{\tau' \in \Sigma (k-1) \cap \left( U_1 \times ... \times U_{k} \right)} \mu_{k-1} (\tau, \tau ') \Psi_{k-1} (U_1,...,U_l) (\sigma_0 ) \right\vert = 
k^{l-k-1} \sum_{\tau \in \Sigma (k-1) \cap \left( U_0 \times ... \times U_{k-1} \right) } \nu_{k-1} (\tau) \Psi_{k-1} (U_0,...,U_l) (\tau )=
k^{l-k-1} pathc_{k-1} (U_0,...,U_l).
\end{dmath*}
\end{proof}

\begin{lemma}
\label{mixing lemma 3}
Let $X$ be a pure $n$-dimensional weighted simplicial complex such that all the links of $X$ of dimension $>0$ are connected. Let $0 \leq l$. For any disjoint sets $U_0,...,U_l \subseteq X^{(0)}$ we have that
$$\left( \prod_{i =0}^{l-1} \mathbb{P}_{U_i} \Delta^-_0 \right) \chi_{U_l}  = 
\dfrac{m(U_1) ... m(U_l)}{m(X^{(0)} )^{l}} \chi_{U_0},$$
and
$$ \left\langle \chi_{U_0}, \left( \prod_{i =0}^{l-1} \mathbb{P}_{U_i} \Delta^-_0 \right) \chi_{U_l} \right\rangle  =\dfrac{m(U_0) ... m (U_l)}{m(X^{(0)} )^{l}} = pathc_{-1} (U_0,...,U_l).$$
\end{lemma}

\begin{proof}
Recall that for every $\phi \in C^{0} (X,\mathbb{R})$,  $\Delta^-_0 \phi$ is the constant function
$$\Delta^-_0 \phi \equiv \dfrac{\sum_{u \in \Sigma (0)} m(u) \phi (u)}{m (X^{(0)})} .$$
Therefore for every non empty set $U \subseteq X^{(0)}$, 
$$\Delta^-_0 \chi_U \equiv \dfrac{m(U)}{m (X^{(0)})} .$$
Since for every $i$, projection the $\mathbb{P}_{U_i}$ is multiplying by $\chi_{U_i}$ we get that
\begin{dmath*}
\left( \prod_{i =0}^{l-1} \mathbb{P}_{U_i} \Delta^-_0 \right) \chi_{U_l}  = \left( \prod_{i =0}^{l-2} \mathbb{P}_{U_i} \Delta^-_0 \right) \mathbb{P}_{U_i} \Delta^-_0 \chi_{U_l}  = \left( \prod_{i =0}^{l-2} \mathbb{P}_{U_i} \Delta^-_0 \right) \dfrac{m(U_l)}{m (X^{(0)})} \chi_{U_{l-1}} = 
\dfrac{m(U_l)}{m (X^{(0)})}  \left( \prod_{i =0}^{l-2} \mathbb{P}_{U_i} \Delta^-_0 \right) \chi_{U_{l-1}} = ... = \dfrac{m(U_1) ... m(U_l)}{m(X^{(0)} )^{l}} \chi_{U_0}.
\end{dmath*}
For the second equality - the first equality combined with $\left\langle \chi_{U_0} , \chi_{U_0} \right\rangle = m(U_0)$ gives
$$ \left\langle \chi_{U_0}, \left( \prod_{i =0}^{l-1} \mathbb{P}_{U_i} \Delta^-_0 \right) \chi_{U_l} \right\rangle  =\dfrac{m(U_0) ... m (U_l)}{m(X^{(0)} )^{l}},$$
and by proposition \ref{case(-1) For pathc}, we get that
$$\dfrac{m(U_0) ... m (U_l)}{m(X^{(0)} )^{l}} = pathc_{-1} (U_0,...,U_l).$$ 
\end{proof}

\subsection{Mixing for two-sided local spectral expansion}

Combining the above results with further assumptions of the spectra of the Laplacians in the links gives the following:
\begin{lemma}
\label{mixing descent lemma}
Let $X$ be a pure $n$-dimensional weighted simplicial complex such that all the links of $X$ of dimension $>0$ are connected. Let $0 \leq k \leq n-1$. Assume that there are $\kappa \geq \lambda > \frac{k}{k+1}$ such that
$$\bigcup_{\tau \in \Sigma (k-1)} Spec (\Delta_{\tau,0}^+) \setminus \lbrace 0 \rbrace \subseteq [\lambda, \kappa],$$

Then for any $k <l$ and any disjoint sets $U_0,...,U_l \subseteq X^{(0)}$ we have that:
\begin{enumerate}
\item For $k=0$, 
\begin{dmath*}
\left\vert pathc_0 (U_0,...,U_l) - \left(\dfrac{\lambda + \kappa}{2} \right)^l pathc_{-1} (U_0,...,U_l) \right\vert \leq  l \left(\dfrac{\kappa}{2} \right)^{l-1} \left(\dfrac{\kappa - \lambda}{2} \right)  \sqrt{m (U_0) m(U_l) }.
\end{dmath*}
\item For $1 \leq k \leq n-1$, denote 
\begin{dmath*}
\left\vert (k+1)^{l-1 - (k+1)} pathc_k (U_0,...,U_l) -  \left(\dfrac{\lambda + \kappa}{2} \right)^{l-k} k^{l-1-k} pathc_{k-1} (U_0,...,U_l) \right\vert \leq \\
(l-k) (k+1) \dfrac{\kappa - \lambda}{2}    \left( \dfrac{(k+1)\kappa -k}{2} \right)^{l-k-1} \sqrt{m(U_0,...,U_k) m(U_{l-k},...,U_l )}.
\end{dmath*}
\end{enumerate}
\end{lemma}

\begin{proof}
The proof is very similar in the both cases - $k=0$ and $k \geq 1$. We'll write a detailed proof for the case $k=0$ and in the case $k \geq 1$, we'll sometimes omit some explanations.
\begin{enumerate}
\item First, notice that by definition $\Delta_0^- \phi$ is always a constant function and $\Delta_0^- \chi_{X^{(0)}} \equiv 1$. Therefore the spectrum of $\Delta_0^-$ is always $\lbrace 0 ,1\rbrace$.  Also, for $k=0$, we get that 
$$\bigcup_{\tau \in \Sigma (-1)} Spec (\Delta_{\tau,0}^+) \setminus \lbrace 0 \rbrace = Spec (\Delta_{0}^+) \setminus \lbrace 0 \rbrace, $$
therefore $ Spec (\Delta_{0}^+) \setminus \lbrace 0 \rbrace \subseteq [ \lambda, \kappa]$. \\
By corollary \ref{pathc_k as inner product} we have that 
$$ pathc_0 (U_0,...,U_l) =\left\vert \left\langle \chi_{U_0}, \prod_{i =0}^{l-1} \left( \mathbb{P}_{U_i} \Delta^+_0 \right) \chi_{U_l} \right\rangle \right\vert  .$$
By lemma \ref{mixing lemma 3} we have that
$$pathc_{-1} (U_0,...,U_l) = \left\vert \left\langle \chi_{U_0}, \prod_{i =0}^{l-1} \left( \mathbb{P}_{U_i} \Delta^-_0 \right) \chi_{U_l} \right\rangle \right\vert = \left\vert \left\langle \chi_{U_0}, \prod_{i =0}^{l-1} \left( \mathbb{P}_{U_i} (-\Delta^-_0) \right) \chi_{U_l} \right\rangle \right\vert .$$
Therefore 
\begin{dmath*}
\left\vert pathc_0 (U_0,...,U_l) - \left(\dfrac{\lambda + \kappa}{2} \right)^l pathc_{-1} (U_0,...,U_l) \right\vert \leq \left\vert \left\langle \chi_{U_0}, \prod_{i =0}^{l-1} \left( \mathbb{P}_{U_i}  \Delta^+_0 \right) \chi_{U_l} \right\rangle - \left(\dfrac{\lambda + \kappa}{2} \right)^l \left\langle \chi_{U_0}, \prod_{i =0}^{l-1} \left( \mathbb{P}_{U_i} (-\Delta^-_0) \right) \chi_{U_l} \right\rangle \right\vert \leq 
\end{dmath*}
\begin{equation}
\label{ineq1}
\sum_{j=0}^{l-1}  \left(\dfrac{\lambda + \kappa}{2} \right)^j \left\vert \left\langle \chi_{U_0}, \prod_{i =0}^{j-1} \left( \mathbb{P}_{U_i}  (-\Delta^-_0) \right) \left( \mathbb{P}_{U_j}  \left( \Delta^+_0 +  \left(\dfrac{\lambda + \kappa}{2} \right) \Delta_0^- \right) \right) \prod_{i =j+1}^{l-1} \left( \mathbb{P}_{U_i}  \Delta^+_0 \right) \chi_{U_l} \right\rangle \right\vert .
\end{equation}
Next, note that for every constant $\alpha \in \mathbb{R}$ we have that
$$\forall 0 \leq i \leq l-2, \mathbb{P}_{U_{i}} (\alpha I) \mathbb{P}_{U_{i+1}} = \alpha \mathbb{P}_{U_{i}}  \mathbb{P}_{U_{i+1}}  = 0.$$
Therefore we have
\begin{dmath*}
\prod_{i =0}^{j-1} \left( \mathbb{P}_{U_i}  (-\Delta^-_0) \right) \left( \mathbb{P}_{U_j}  \left( \Delta^+_0 + \left(\dfrac{\lambda + \kappa}{2} \right) \Delta_0^- \right) \right) \prod_{i =j+1}^{l-1} \left( \mathbb{P}_{U_i}  \Delta^+_0 \right) \chi_{U_l} = 
\prod_{i =0}^{j-1} \left( \mathbb{P}_{U_i}  (\dfrac{1}{2} I -\Delta^-_0) \right) \left( \mathbb{P}_{U_j}  \left( \Delta^+_0 + \left(\dfrac{\lambda + \kappa}{2} \right) \Delta_0^-  - \dfrac{\lambda + \kappa}{2} I \right) \right) \prod_{i =j+1}^{l-1} \left( \mathbb{P}_{U_i}  ( \Delta^+_0 - \dfrac{\kappa}{2} I ) \right) \chi_{U_l}  .
\end{dmath*}
By the information we have on the spectrum of $\Delta_0^+, \Delta_0^-$ we get the following bounds on the operator norms:
$$\Vert \dfrac{1}{2} I -\Delta^-_0 \Vert \leq \dfrac{1}{2}, \Vert \Delta^+_0 - \dfrac{\kappa}{2} I  \Vert \leq \dfrac{\kappa}{2} .$$
By corollary \ref{norm bound - local to global} for the case $k=0$, we have that
$$\left\Vert \Delta^+_0 + \left(\dfrac{\lambda + \kappa}{2} \right) \Delta_0^-  - \dfrac{\lambda + \kappa}{2} I \right\Vert \leq \dfrac{\kappa - \lambda}{2} .$$
Therefore 
\begin{dmath*}
\left\Vert \prod_{i =0}^{j-1} \left( \mathbb{P}_{U_i}  (\dfrac{1}{2} I -\Delta^-_0) \right) \left( \mathbb{P}_{U_j}  \left( \Delta^+_0 + \left(\dfrac{\lambda + \kappa}{2} \right) \Delta_0^-  - \left(\dfrac{\lambda + \kappa}{2} \right) I \right) \right) \prod_{i =j+1}^{l-1} \left( \mathbb{P}_{U_i}  ( \Delta^+_0 - \dfrac{\kappa}{2} I ) \right) \right\Vert \leq  
 \left(\dfrac{1}{2} \right)^j   \left(\dfrac{\kappa - \lambda}{2} \right)  \left(\dfrac{\kappa}{2} \right)^{l-1-j} 
=  \left(\dfrac{\kappa}{2} \right)^{l-1} \left(\dfrac{\kappa - \lambda}{2} \right)  \left( \dfrac{1}{\kappa} \right)^{j}.
\end{dmath*}
This yields the following bound on \eqref{ineq1}
\begin{dmath*}
\left(\dfrac{\kappa}{2} \right)^{l-1} \left(\dfrac{\kappa - \lambda}{2} \right) \sum_{j=0}^{l-1} \left(\dfrac{\lambda + \kappa}{2 \kappa}  \right)^{j} \Vert \chi_{U_0} \Vert \Vert \chi_{U_l} \Vert \leq
l \left(\dfrac{\kappa}{2} \right)^{l-1} \left(\dfrac{\kappa - \lambda}{2} \right)  \sqrt{m (U_0) m(U_l) }.
\end{dmath*}
\item By the same considerations as in the $k=0$ case we have for $1 \leq k \leq n-1$ that

\begin{dmath*}
{\left\vert (k+1)^{l-1-(k+1)} pathc_k (U_0,...,U_l) -  \left(\dfrac{\kappa + \lambda}{2} \right)^{l-k} k^{l-1-k} pathc_{k-1} (U_0,...,U_l) \right\vert  =} \\
\left\vert \left\langle \chi_{U_0,...,U_k}, \left( \prod_{i =0}^{l-k-1} \mathbb{P}_{U_i,...,U_{k+i}} \Delta^+_k \right) \chi_{U_{l-k},...,U_l} \right\rangle   
- \left(\dfrac{\kappa + \lambda}{2} \right)^{l-k} \left\langle \chi_{U_0,...,U_k}, \left( \prod_{i =0}^{l-k-1} \mathbb{P}_{U_i,...,U_{k+i}} (-\Delta^-_k) \right) \chi_{U_{l-k},...,U_l} \right\rangle \right\vert \leq \\
\sum_{j=0}^{l-1-k} \left(\dfrac{\kappa + \lambda}{2} \right)^{j} \\ \left\vert \left\langle \chi_{U_0,...,U_k}, \left( \prod_{i =0}^{j-1} \mathbb{P}_{U_i,...,U_{k+i}} (-\Delta^-_k) \right) \\ \left( \mathbb{P}_{U_j,...,U_{k+j}} \left( \Delta^+_k  +\left(\dfrac{\kappa + \lambda}{2} \right) \Delta_k^- \right) \right) \left( \prod_{i =j+1}^{l-k-1} \mathbb{P}_{U_i,...,U_{k+i}} \Delta^+_k \right) \chi_{U_{l-k},...,U_l} \right\rangle \right\vert
\end{dmath*}
As before, we can translate by $\alpha I$ for suitable $\alpha$'s in order to get
\begin{dmath*}
\sum_{j=0}^{l-1-k} \left(\dfrac{\kappa + \lambda}{2} \right)^{j} \\ \left\vert \left\langle \chi_{U_0,...,U_k}, \left( \prod_{i =0}^{j-1} \mathbb{P}_{U_i,...,U_{k+i}}  (\dfrac{(k+1)\kappa -k}{2 \kappa} I -\Delta^-_k) \right) \\ \left( \mathbb{P}_{U_j,...,U_{k+j}} \left( \Delta^+_k  +\left(\dfrac{\kappa + \lambda}{2} \right) \Delta_k^- - (k+1) (\dfrac{\lambda + \kappa}{2} - \dfrac{k}{k+1}) I\right) \right) \\ \left( \prod_{i =j+1}^{l-k-1} \mathbb{P}_{U_i,...,U_{k+i}} (\Delta^+_k -  \dfrac{(k+1)\kappa -k}{2} I ) \right) \chi_{U_{l-k},...,U_l} \right\rangle \right\vert
\end{dmath*}
Recall that
$$Spec (\Delta_{k}^+) \setminus \lbrace 0 \rbrace \subseteq [(k+1)\lambda - k, (k+1)\kappa - k] ,$$
$$Spec (\Delta_{k}^-) \setminus \lbrace 0 \rbrace \subseteq [k(2 - \dfrac{1}{\lambda}) - (k-1), k(2 - \dfrac{1}{\kappa}) - (k-1)] = [\dfrac{(k+1) \lambda - k}{\lambda}, \dfrac{(k+1) \kappa - k}{\kappa}].$$
Therefore
$$\left\Vert  \Delta^+_k -  \dfrac{(k+1)\kappa -k}{2} I \right\Vert \leq  \dfrac{(k+1)\kappa -k}{2}, $$
$$\left\Vert  \dfrac{(k+1)\kappa -k}{2 \kappa} I -\Delta^-_k \right\Vert \leq   \dfrac{(k+1)\kappa -k}{2 \kappa}. $$
By corollary \ref{norm bound - local to global}, we get that
$$\left\Vert  \Delta^+_k  +\left(\dfrac{\kappa + \lambda}{2} \right) \Delta_k^- - (k+1) (\dfrac{\lambda + \kappa}{2} - \dfrac{k}{k+1}) I \right\Vert \leq (k+1) \dfrac{\kappa - \lambda}{2} .$$
Therefore we have the following upper bound on the sum above
\begin{dmath*}
{(k+1) \dfrac{\kappa - \lambda}{2}  \sqrt{m(U_0,...,U_k) m(U_{l-k},...,U_l )} } \\ \sum_{j=0}^{l-1-k} \left( \dfrac{\lambda + \kappa}{2} \right)^j \left( \dfrac{(k+1)\kappa -k}{2 \kappa} \right)^j \left( \dfrac{(k+1)\kappa -k}{2} \right)^{l-k-1-j} = {(k+1) \dfrac{\kappa - \lambda}{2}  \left( \dfrac{(k+1)\kappa -k}{2} \right)^{l-k-1} \sqrt{m(U_0,...,U_k) m(U_{l-k},...,U_l )}   } \\ {\sum_{j=0}^{l-1-k} \left( \dfrac{\lambda + \kappa}{2 \kappa} \right)^j \leq} \\ {(l-k) (k+1) \dfrac{\kappa - \lambda}{2}    \left( \dfrac{(k+1)\kappa -k}{2} \right)^{l-k-1} \sqrt{m(U_0,...,U_k) m(U_{l-k},...,U_l )}} .
\end{dmath*}
\end{enumerate}
\end{proof}

Recall that by corollary \ref{SpectralGapDescent2} bounds on the non trivial spectrum of the $1$-dimensional links yielded bounds of the non trivial spectrum of the links of all dimensions larger than $1$. By this we have the following mixing theorem:

\begin{theorem}
\label{mixing theorem}
Let $X$ be a pure $n$-dimensional weighted simplicial complex such that all the links of $X$ of dimension $>0$ are connected.  Denote $f(x) = 2-\frac{1}{x}$ and $f^j$ to be the composition of $f$ with itself $j$ times (where $f^0$ is defined as $f^0 (x) = x$).
Assume there are $\kappa \geq \lambda > \frac{n-1}{n}$ such that
$$\bigcup_{\tau \in \Sigma (n-2)} Spec (\Delta_{\tau, 0}^+) \setminus \lbrace 0 \rbrace \subseteq [\lambda, \kappa].$$
For every $0 \leq j \leq n-1$, denote 
$$\lambda_j =  f^{n-1-j} (\lambda ) , \kappa_j =  f^{n-1-j} (\kappa ),$$
$$r_j = \dfrac{\lambda_j  + \kappa_j}{2}, \varepsilon_j = (l-j) (j+1)   \left( \dfrac{(j+1)\kappa_j -j}{2} \right)^{l-j-1} \dfrac{\kappa_j - \lambda_j}{2}.$$
Then for every $0 \leq k \leq n-1$ and for every $k <l$ and any disjoint sets $U_0,...,U_l \subseteq X^{(0)}$ we have that:
\begin{dmath*}
\left\vert \left( k+1 \right)^{l-1 - (k+1)} pathc_k (U_0,...,U_l) -   \left( \prod_{j=0}^k r_j^{l-j} \right) \dfrac{m(U_0) ... m (U_l)}{m (X^{(0)})^{l} } \right\vert \leq 
\sum_{i=0}^k  \varepsilon_i  \left( \prod_{j=i+1}^{k} r_j^{l-j} \right)  \sqrt{m(U_0,...,U_i) m(U_{l-i},...,U_l )}.
\end{dmath*}
\end{theorem}

\begin{proof}
We'll prove the theorem by induction on $k$. For $k=0$, recall that
$$\dfrac{m(U_0) ... m (U_l)}{m(X^{(0)} )^{l}} = pathc_{-1} (U_0,...,U_l).$$
By corollary \ref{SpectralGapDescent2}, we have that
$$\bigcup_{\tau \in \Sigma (-1)} Spec (\Delta_{\tau,0}^+) \setminus \lbrace 0 \rbrace \subseteq [f^{n-1} (\lambda),f^{n-1} ( \kappa)].$$
Therefore, by the lemma \ref{mixing descent lemma}, we get that 
\begin{dmath*}
{ \left\vert pathc_0 (U_0,...,U_l) -  r_0^l pathc_{-1} (U_0,...,U_l) \right\vert = }  \\ \left\vert pathc_0 (U_0,...,U_l) - \left(\dfrac{f^{n-1} (\lambda) + f^{n-1} (\kappa) }{2} \right)^l pathc_{-1} (U_0,...,U_l) \right\vert \leq  \\
{ l \left(\dfrac{f^{n-1} (\kappa)}{2} \right)^{l-1} \left(\dfrac{ f^{n-1} (\kappa) -  f^{n-1} (\lambda) }{2} \right)  \sqrt{m (U_0) m(U_l) } = } \\ \varepsilon_0   \sqrt{m (U_0) m(U_l) },
\end{dmath*}
and we are done. Next, assume that the theorem holds for $k-1$, then we have
\begin{dmath*}
\left\vert (k+1)^{l-1 - (k+1)} pathc_k (U_0,...,U_l) -  \left( \prod_{j=0}^k r_j^{l-j} \right) \dfrac{m(U_0) ... m (U_l)}{m (X^{(0)})^{l} } \right\vert \leq 
\left\vert (k+1)^{l-1 - (k+1)} pathc_{k} (U_0,...,U_l)  - r_k^{l-k} k^{l-1-k} pathc_{k-1} (U_0,...,U_l) \right\vert + r_k^{l-k}  \left\vert  k^{l-1-k} pathc_{k-1} (U_0,...,U_l) - \left( \prod_{j=0}^{k-1} r_j^{l-j} \right) \dfrac{m(U_0) ... m (U_l)}{m (X^{(0)})^{l} } \right\vert \leq 
\left\vert (k+1)^{l-1 - (k+1)} pathc_{k} (U_0,...,U_l)  - r_k^{l-k} k^{l-1-k} pathc_{k-1} (U_0,...,U_l) \right\vert + r_k^{l-k} \sum_{i=0}^{k-1} \varepsilon_i  \left( \prod_{j=i+1}^{k-1} r_j^{l-j} \right) \sqrt{m(U_0,...,U_i) m(U_{l-i},...,U_l )}
\end{dmath*}

By corollary \ref{SpectralGapDescent2}, 
$$\bigcup_{\tau \in \Sigma (k-1)} Spec (\Delta_{\tau,0}^+) \setminus \lbrace 0 \rbrace \subseteq [f^{n-1-k} (\lambda),f^{n-1-k} ( \kappa)] = [ \lambda_k, \kappa_k].$$
And we finish by applying lemma \ref{mixing descent lemma}.

\end{proof}

Now we are ready to give the exact statement and proof of theorem \ref{Mixing in the general case - section 2} stated above:

\begin{corollary}
\label{mixing m(U_0,...,U_l)}
Let $X$ be a pure $n$-dimensional weighted simplicial complex such that all the links of $X$ of dimension $>0$ are connected. If there are $\kappa \geq \lambda > \frac{n-1}{n}$ such that
$$\bigcup_{\tau \in \Sigma (n-2)} Spec (\Delta_{\tau, 0}^+) \setminus \lbrace 0 \rbrace \subseteq [\lambda, \kappa].$$
Then for every $1 \leq l \leq n$ there are continuous functions $\mathcal{E}_l (\lambda, \kappa), \mathcal{A}_l (\lambda, \kappa)$ such that
$$\lim_{(\lambda,\kappa) \rightarrow (1,1)} \mathcal{E}_l (\lambda, \kappa) = 0, \lim_{(\lambda,\kappa) \rightarrow (1,1)} \mathcal{A}_l (\lambda, \kappa) = 1, $$
and such that  every non empty disjoint sets $U_0,...,U_l \subseteq X^{(0)}$ the following inequalities holds:
\begin{dmath*}
\left\vert m(U_0,...,U_l) -  \mathcal{A}_l (\lambda, \kappa) \dfrac{m(U_0) ... m (U_l)}{m (X^{(0)})^{l} } \right\vert \leq 
\mathcal{E}_l (\lambda, \kappa) \min_{0 \leq i < j \leq l} \sqrt{m(U_i) m (U_j)},
\end{dmath*}
and
\begin{dmath*}
\left\vert m(U_0,...,U_l) -  \mathcal{A}_l (\lambda, \kappa) \dfrac{m(U_0) ... m (U_l)}{m (X^{(0)})^{l} } \right\vert \leq 
\mathcal{E}_l (\lambda, \kappa) \left( m (U_0)...m (U_l) \right)^{\frac{1}{l+1}} .
\end{dmath*}
\end{corollary}

\begin{proof}
Recall that for every $1 \leq l \leq n-1$ we have by proposition \ref{pathc_k for U_0,...,U_(k+1)} that 
$$   m(U_0,...,U_l) = \dfrac{pathc_{l-1} (U_0,...,U_l)}{l+1} .$$
Then by theorem \ref{mixing theorem} with $l, k=l-1$ we get
\begin{dmath*}
\left\vert m(U_0,...,U_l)  -   \left( \prod_{j=0}^{l-1} r_j^{l-j}  \right) \dfrac{m(U_0) ... m (U_l)}{m (X^{(0)})^{l} } \right\vert \leq \sum_{i=0}^{l-1}  \varepsilon_i  \left( \prod_{j=i+1}^{l-1} r_j^{l-j} \right)  \sqrt{m(U_0,...,U_i) m(U_{l-i},...,U_l )}\leq
\sqrt{m(U_0) m (U_l)} \sum_{i=0}^{l-1}  \varepsilon_i  \left( \prod_{j=i+1}^{l-1} r_j^{l-j} \right) \sqrt{\dfrac{m(U_0,...,U_i)}{m (U_0)} \dfrac{m(U_{l-i},...,U_l )}{m (U_l)}} \leq \sqrt{m(U_0) m (U_l)} \sum_{i=0}^{l-1}  \varepsilon_i  \left( \prod_{j=i+1}^{l-1} r_j^{l-j} \right),
\end{dmath*}
where the last inequality is due to the fact that for every $i$,
$$m(U_0,...,U_i) \leq m (U_0), m(U_{l-i},...,U_l) \leq m (U_l).$$
Note that for every $i$,
$$\lim_{(\lambda, \kappa) \rightarrow (1,1)} r_i =1, \lim_{(\lambda, \kappa) \rightarrow (1,1)} \varepsilon_i = 0 .$$
Therefore, denoting 
$$\mathcal{A}_l (\lambda, \kappa) = \prod_{j=0}^{l-1} r_j^{l-j} ,$$
$$\mathcal{E}_l (\lambda, \kappa) = \sum_{i=0}^{l-1}  \varepsilon_i  \left( \prod_{j=i+1}^{l-1} r_j^{l-j} \right)  ,$$
yields the inequality
\begin{dmath*}
\left\vert m(U_0,...,U_l) -  \mathcal{A}_l (\lambda, \kappa) \dfrac{m(U_0) ... m (U_l)}{m (X^{(0)})^{l} } \right\vert \leq 
\mathcal{E}_l (\lambda, \kappa) \sqrt{m(U_0) m (U_l)},
\end{dmath*}
with 
$$\lim_{(\lambda,\kappa) \rightarrow (1,1)} \mathcal{E}_l (\lambda, \kappa) = 0, \lim_{(\lambda,\kappa) \rightarrow (1,1)} \mathcal{A}_l (\lambda, \kappa) = 1.$$
Note that by definition of $m(U_0,...,U_l)$, we have any reordering $\pi \in Sym (\lbrace 0,...,l \rbrace)$ that
$$m(U_0,...,U_l) = m(U_{\pi (0)},...,U_{\pi (l)} )$$
Therefore, for every $0 \leq i < j \leq l$, we have 
\begin{dmath*}
\left\vert m(U_0,...,U_l) -  \mathcal{A}_l (\lambda, \kappa) \dfrac{m(U_0) ... m (U_l)}{m (X^{(0)})^{l} } \right\vert \leq 
\mathcal{E}_l (\lambda, \kappa) \sqrt{m(U_i) m (U_j)}.
\end{dmath*}
Taking $i,j$ such that $m(U_i) m(U_j)$ is minimal yields the first inequality stated above. Next, multiply the inequalities for all the different choices of $0 \leq i<j \leq l$:
\begin{dmath*}
\left( \left\vert m(U_0,...,U_l) - \mathcal{A}_l (\lambda, \kappa) \dfrac{m(U_0) ... m (U_l)}{m (X^{(0)})^{l} } \right\vert \right)^{\frac{l(l+1)}{2}} \leq 
\left( \mathcal{E}_l (\lambda, \kappa) \right)^{\frac{l(l+1)}{2}} \sqrt{m(U_0)^l... m (U_l)^l}.
\end{dmath*}
Taking both sides to the power $\frac{2}{l(l+1)}$ yields the second inequality stated above.
\end{proof}

\subsection{Mixing for partite simplicial complexes}
The above mixing result dealt with general simplicial complexes. Here we deal with mixing for $(n+1)$-partite simplicial complexes, which is an analogue to mixing in bipartite graphs. The proof of the mixing results is very similar to the proofs of the mixing in the general case, but relays on corollary \ref{norm bound - n+1 partite case} instead of corollary \ref{norm bound - local to global} that was used in the general case. For the convenience of the reader, we shall repeat all the arguments in the proofs even where there is a substantial overlap with the arguments given in the proofs of the mixing in the general case.

\begin{lemma}
\label{mixing descent lemma -partite case}
Let $X$ be a pure $n$-dimensional, weighted, $(n+1)$-partite simplicial complex such that all the links of $X$ of dimension $>0$ are connected. Denote by $S_0,...,S_n$ the sides of $X$. Let $0 \leq k \leq n-1$. Assume that there is $ \lambda > \frac{k}{k+1}$ such that
$$\bigcup_{\tau \in \Sigma (k-1)} Spec (\Delta_{\tau,0}^+) \setminus \lbrace 0 \rbrace \subseteq [\lambda, \infty),$$

Then for any $k <l \leq n$ and any disjoint sets $U_0 \subseteq S_0,...,U_l \subseteq S_l$ we have that:
\begin{enumerate}
\item For $k=0$, 
\begin{dmath*}
\left\vert pathc_0 (U_0,...,U_l) - \left(\dfrac{n+1}{n} \right)^l pathc_{-1} (U_0,...,U_l) \right\vert \leq  l \left(\dfrac{n+1}{2n} \right)^{l-1} \left(\dfrac{(n+1)(1- \lambda)}{2} \right)   \sqrt{m (U_0) m(U_l) }.
\end{dmath*}
\item For $1 \leq k \leq n-1$, denote 
\begin{dmath*}
\left\vert (k+1)^{l-1 - (k+1)} pathc_k (U_0,...,U_l) -  \left(\frac{n+1-k}{n-k}  \right)^{l-k} k^{l-1-k} pathc_{k-1} (U_0,...,U_l) \right\vert \leq \\
{(l-k)(k+1)(n+1-k) \dfrac{1 - \lambda}{2} \left( \dfrac{n+1}{2(n-k)} \right)^{l-k}\sqrt{m(U_0,...,U_k) m(U_{l-k},...,U_l )} }.
\end{dmath*}
\end{enumerate}
\end{lemma}

\begin{proof}
The proof is very similar in the both cases - $k=0$ and $k \geq 1$. We'll write a detailed proof for the case $k=0$ and in the case $k \geq 1$, we'll sometimes omit some explanations.
\begin{enumerate}
\item First, notice that by definition $\Delta_0^- \phi$ is always a constant function and $\Delta_0^- \chi_{X^{(0)}} \equiv 1$. Therefore the spectrum of $\Delta_0^-$ is always $\lbrace 0 ,1\rbrace$.  Also, for $k=0$, we get that 
$$\bigcup_{\tau \in \Sigma (-1)} Spec (\Delta_{\tau,0}^+) \setminus \lbrace 0 \rbrace = Spec (\Delta_{0}^+) \setminus \lbrace 0 \rbrace, $$
therefore $ Spec (\Delta_{0}^+) \setminus \lbrace 0 \rbrace \subseteq [ \lambda, \frac{n+1}{n}] $, where the upper bound is due to corollary \ref{Laplacian norm bounds}.\\
By corollary \ref{pathc_k as inner product} we have that 
$$ pathc_0 (U_0,...,U_l) =\left\vert \left\langle \chi_{U_0}, \prod_{i =0}^{l-1} \left( \mathbb{P}_{U_i} \Delta^+_0 \right) \chi_{U_l} \right\rangle \right\vert  .$$
By lemma \ref{mixing lemma 3} we have that
$$pathc_{-1} (U_0,...,U_l) = \left\vert \left\langle \chi_{U_0}, \prod_{i =0}^{l-1} \left( \mathbb{P}_{U_i} \Delta^-_0 \right) \chi_{U_l} \right\rangle \right\vert = \left\vert \left\langle \chi_{U_0}, \prod_{i =0}^{l-1} \left( \mathbb{P}_{U_i} (-\Delta^-_0) \right) \chi_{U_l} \right\rangle \right\vert .$$
Therefore 
\begin{dmath*}
\left\vert pathc_0 (U_0,...,U_l) - \left(\dfrac{n+1}{n} \right)^l pathc_{-1} (U_0,...,U_l) \right\vert \leq \left\vert \left\langle \chi_{U_0}, \prod_{i =0}^{l-1} \left( \mathbb{P}_{U_i}  \Delta^+_0 \right) \chi_{U_l} \right\rangle - \left(\dfrac{n+1}{n} \right)^l \left\langle \chi_{U_0}, \prod_{i =0}^{l-1} \left( \mathbb{P}_{U_i} (-\Delta^-_0) \right) \chi_{U_l} \right\rangle \right\vert \leq 
\end{dmath*}
\begin{equation}
\label{ineq1-partite}
\sum_{j=0}^{l-1}  \left(\dfrac{n+1}{n} \right)^j \left\vert \left\langle \chi_{U_0}, \prod_{i =0}^{j-1} \left( \mathbb{P}_{U_i}  (-\Delta^-_0) \right) \left( \mathbb{P}_{U_j}  \left( \Delta^+_0 +  \left(\dfrac{n+1}{n} \right) \Delta_0^- \right) \right) \prod_{i =j+1}^{l-1} \left( \mathbb{P}_{U_i}  \Delta^+_0 \right) \chi_{U_l} \right\rangle \right\vert .
\end{equation}
Next, note for any $\alpha \in \mathbb{R}$ we have that
$$\forall 0 \leq i \leq l-2, \mathbb{P}_{U_{i}} (\alpha I) \mathbb{P}_{U_{i+1}} = \alpha \mathbb{P}_{U_{i}}  \mathbb{P}_{U_{i+1}}  = 0.$$
Also, for any constant $\alpha \in \mathbb{R}$ and for any $0 \leq j \leq n$, we have by corollary \ref{side average operators} that
$$\forall 0 \leq i \leq l-2, \mathbb{P}_{U_{i}} (\alpha \Delta^-_{(0,j)}) \mathbb{P}_{U_{i+1}} = \alpha \mathbb{P}_{U_{i}}  \Delta^-_{(0,j)} \mathbb{P}_{U_{i+1}}  = 0,$$
(note that here we are using the assumption that $U_0 \subseteq S_0,...,U_l \subseteq S_l$). \\
Therefore we have
\begin{dmath*}
\prod_{i =0}^{j-1} \left( \mathbb{P}_{U_i}  (-\Delta^-_0) \right) \left( \mathbb{P}_{U_j}  \left( \Delta^+_0 + \left(\dfrac{n+1}{n} \right) \Delta_0^- \right) \right) \prod_{i =j+1}^{l-1} \left( \mathbb{P}_{U_i}  \Delta^+_0 \right) \chi_{U_l} = 
{\prod_{i =0}^{j-1} \left( \mathbb{P}_{U_i}  (\dfrac{1}{2} I -\Delta^-_0) \right)}  \\ {\left( \mathbb{P}_{U_j}  \left( \Delta^+_{0} +\frac{n+1}{n}   \Delta^-_{0} -  \dfrac{2+n(1-\lambda)}{2}  I  - ( \frac{(n+1)^2}{n} - (n+1)^2 \dfrac{2+n(1-\lambda)}{2} ) \sum_{j=0}^n \Delta^-_{(0,j)} \right) \right) } \\ \prod_{i =j+1}^{l-1} \left( \mathbb{P}_{U_i}  ( \Delta^+_0 - \dfrac{n+1}{2n} I ) \right) \chi_{U_l}  .
\end{dmath*}
By the information we have on the spectrum of $\Delta_0^+, \Delta_0^-$ we get the following bounds on the operator norms:
$$\Vert \dfrac{1}{2} I -\Delta^-_0 \Vert \leq \dfrac{1}{2}, \Vert \Delta^+_0 - \dfrac{n+1}{2n} I  \Vert \leq \dfrac{n+1}{2n} .$$
By corollary \ref{norm bound - n+1 partite case} for the case $k=0$, we have that
\begin{dmath*}
\left\Vert  \Delta^+_{0} +\frac{n+1}{n}   \Delta^-_{0} -  \dfrac{2+n(1-\lambda)}{2}  I  - ( \frac{(n+1)^2}{n} - (n+1)^2 \dfrac{2+n(1-\lambda)}{2} ) \sum_{j=0}^n \Delta^-_{(0,j)} \right\Vert \leq \dfrac{(n+1)(1- \lambda)}{2} .
\end{dmath*}
Therefore 
\begin{dmath*}
\left\Vert {\prod_{i =0}^{j-1} \left( \mathbb{P}_{U_i}  (\dfrac{1}{2} I -\Delta^-_0) \right)}  \\ {\left( \mathbb{P}_{U_j}  \left( \Delta^+_{0} +\frac{n+1}{n}   \Delta^-_{0} -  \dfrac{2+n(1-\lambda)}{2}  I  - ( \frac{(n+1)^2}{n} - (n+1)^2 \dfrac{2+n(1-\lambda)}{2} ) \sum_{j=0}^n \Delta^-_{(0,j)} \right) \right) } \\ \prod_{i =j+1}^{l-1} \left( \mathbb{P}_{U_i}  ( \Delta^+_0 - \dfrac{n+1}{2n} I ) \right) \right\Vert \leq  
 \left(\dfrac{1}{2} \right)^j   \left(\dfrac{(n+1)(1- \lambda)}{2} \right)  \left(\dfrac{n+1}{2n} \right)^{l-1-j} 
=  \left(\dfrac{n+1}{2n} \right)^{l-1} \left(\dfrac{(n+1)(1- \lambda)}{2} \right)  \left( \dfrac{n}{n+1} \right)^{j}.
\end{dmath*}
This yields the following bound on \eqref{ineq1-partite}
\begin{dmath*}
\left(\dfrac{n+1}{2n} \right)^{l-1} \left(\dfrac{(n+1)(1- \lambda)}{2} \right)    \sum_{j=0}^{l-1} \left(\dfrac{n+1}{n} \right)^j \left( \dfrac{n}{n+1} \right)^{j} \Vert \chi_{U_0} \Vert \Vert \chi_{U_l} \Vert = \\
l \left(\dfrac{n+1}{2n} \right)^{l-1}  \left(\dfrac{(n+1)(1- \lambda)}{2} \right)   \sqrt{m (U_0) m(U_l) }.
\end{dmath*}
\item By the same considerations as in the $k=0$ case we have for $1 \leq k \leq n-1$ that

\begin{dmath*}
{\left\vert (k+1)^{l-1-(k+1)} pathc_k (U_0,...,U_l) -  \left(\frac{n+1-k}{n-k}  \right)^{l-k} k^{l-1-k} pathc_{k-1} (U_0,...,U_l) \right\vert  =} \\
\left\vert \left\langle \chi_{U_0,...,U_k}, \left( \prod_{i =0}^{l-k-1} \mathbb{P}_{U_i,...,U_{k+i}} \Delta^+_k \right) \chi_{U_{l-k},...,U_l} \right\rangle   
- \left(\frac{n+1-k}{n-k}  \right)^{l-k} \left\langle \chi_{U_0,...,U_k}, \left( \prod_{i =0}^{l-k-1} \mathbb{P}_{U_i,...,U_{k+i}} (-\Delta^-_k) \right) \chi_{U_{l-k},...,U_l} \right\rangle \right\vert \leq \\
\sum_{j=0}^{l-1-k} \left(\frac{n+1-k}{n-k}  \right)^{j} \\ \left\vert \left\langle \chi_{U_0,...,U_k}, \left( \prod_{i =0}^{j-1} \mathbb{P}_{U_i,...,U_{k+i}} (-\Delta^-_k) \right) \\ \left( \mathbb{P}_{U_j,...,U_{k+j}} \left( \Delta^+_k  +\left(\frac{n+1-k}{n-k} \right) \Delta_k^- \right) \right) \left( \prod_{i =j+1}^{l-k-1} \mathbb{P}_{U_i,...,U_{k+i}} \Delta^+_k \right) \chi_{U_{l-k},...,U_l} \right\rangle \right\vert
\end{dmath*}
As before, note for any $\alpha \in \mathbb{R}$ we have that
$$\forall 0 \leq i \leq l-k-1, \mathbb{P}_{U_{i},...,U_{k+i}} (\alpha I) \mathbb{P}_{U_{i+1},...,U_{k+i+1}} = \alpha \mathbb{P}_{U_{i},...,U_{k+i}} \mathbb{P}_{U_{i+1},...,U_{k+i+1}} = 0.$$
Also, for any $0 \leq j \leq n$, we have by corollary \ref{side average operators} that
$$\forall 0 \leq i \leq l-2, \forall 0 \leq i \leq l-k-1, \mathbb{P}_{U_{i},...,U_{k+i}} (\alpha \Delta^-_{(0,j)}) \mathbb{P}_{U_{i+1},...,U_{k+i+1}}   = 0,$$
(note that here we are using the assumption that $U_0 \subseteq S_0,...,U_l \subseteq S_l$). \\
Therefore we have
\begin{dmath*}
\sum_{j=0}^{l-1-k} \left(\frac{n+1-k}{n-k} \right)^{j} \\ \left\vert \left\langle \chi_{U_0,...,U_k}, \left( \prod_{i =0}^{j-1} \mathbb{P}_{U_i,...,U_{k+i}}  ( \frac{n+1}{2(n+1-k)} I -\Delta^-_k) \right)  \left( \mathbb{P}_{U_j,...,U_{k+j}} \left( \Delta^+_{k} +\frac{n+1-k}{n-k}   \Delta^-_{k} \\ - ( \dfrac{2+(n-k)(1-\lambda)}{2})  I  - ( \frac{(n+1-k)^2}{n-k} - (n+1-k)^2 \dfrac{2+(n-k)(1-\lambda)}{2} ) \sum_{j=0}^n \Delta^-_{(k,j)}  \right) \right) \\ \left( \prod_{i =j+1}^{l-k-1} \mathbb{P}_{U_i,...,U_{k+i}} (\Delta^+_k -  \dfrac{n+1}{2(n-k)} I ) \right) \chi_{U_{l-k},...,U_l} \right\rangle \right\vert
\end{dmath*}
Recall that
$$Spec (\Delta_{k}^+) \subseteq \left[0, \dfrac{n+1}{n-k}\right] ,$$
$$Spec (\Delta_{k}^-) \subseteq \left[0, \dfrac{n+1}{n+1-k} \right].$$
Therefore
$$\left\Vert  \Delta^+_k -  \dfrac{n+1}{2(n-k)} I \right\Vert \leq \dfrac{n+1}{2(n-k)}, $$
$$\left\Vert  \frac{n+1}{2(n+1-k)} I -\Delta^-_k \right\Vert \leq   \frac{n+1}{2(n+1-k)}. $$
By corollary \ref{norm bound - n+1 partite case}, we get that
\begin{dmath*}
\left\Vert \Delta^+_{k} +\frac{n+1-k}{n-k}   \Delta^-_{k} - (  \dfrac{2+(n-k)(1-\lambda)}{2})  I  - ( \frac{(n+1-k)^2}{n-k} - (n+1-k)^2 \dfrac{2+(n-k)(1-\lambda)}{2} ) \sum_{j=0}^n \Delta^-_{(k,j)}  \right\Vert \leq (k+1)(n+1-k) \dfrac{1 - \lambda}{2} .
\end{dmath*}
Therefore we have the following upper bound on the sum above
\begin{dmath*}
{(k+1)(n+1-k) \dfrac{1 - \lambda}{2} \sqrt{m(U_0,...,U_k) m(U_{l-k},...,U_l )} } \\ \sum_{j=0}^{l-1-k} \left(\frac{n+1-k}{n-k} \right)^{j} \left( \frac{n+1}{2(n+1-k)} \right)^j \left( \dfrac{n+1}{2(n-k)} \right)^{l-k-1-j} = {(l-k)(k+1)(n+1-k) \dfrac{1 - \lambda}{2} \left( \dfrac{n+1}{2(n-k)} \right)^{l-k}\sqrt{m(U_0,...,U_k) m(U_{l-k},...,U_l )} }  .
\end{dmath*}
\end{enumerate}
\end{proof}

Using corollary \ref{SpectralGapDescent2} we can show the following mixing theorem:

\begin{theorem}
\label{mixing theorem - partite case}
Let $X$ be a pure $n$-dimensional, weighted,  $(n+1)$-partite  simplicial complex such that all the links of $X$ of dimension $>0$ are connected.  Denote by $S_0,...,S_n$ the sides of $X$. Denote $f(x) = 2-\frac{1}{x}$ and $f^j$ to be the composition of $f$ with itself $j$ times (where $f^0$ is defined as $f^0 (x) = x$).
Assume there is $\lambda > \frac{n-1}{n}$ such that
$$\bigcup_{\tau \in \Sigma (n-2)} Spec (\Delta_{\tau, 0}^+) \setminus \lbrace 0 \rbrace \subseteq [\lambda, \infty) .$$
For every $0 \leq j \leq n-1$, denote 
$$\lambda_j =  f^{n-1-j} (\lambda ) , \varepsilon_j = (l-j)   \left( \dfrac{n+1}{2(n-j)} \right)^{l-j} (j+1)(n+1-j)\dfrac{1-\lambda_j}{2},$$
$$r_j = \left( \dfrac{n+1-j}{n-j} \right)^{l-j}. $$
Then for every $0 \leq k \leq n-1$ and for every $k <l \leq n$ and any disjoint sets $U_0 \subseteq S_0,...,U_l \subseteq S_l$ we have that:
\begin{dmath*}
\left\vert \left( k+1 \right)^{l-1 - (k+1)} pathc_k (U_0,...,U_l) -  \left( \prod_{j=0}^k r_j^{l-j} \right) \dfrac{m(U_0) ... m (U_l)}{m (X^{(0)})^{l} } \right\vert \leq 
\sum_{i=0}^k  \varepsilon_i \left( \prod_{j=i+1}^{k} r_j^{l-j} \right) \sqrt{m(U_0,...,U_i) m(U_{l-i},...,U_l )}.
\end{dmath*}
\end{theorem}

\begin{proof}
We'll prove the theorem by induction on $k$. For $k=0$, recall that
$$\dfrac{m(U_0) ... m (U_l)}{m(X^{(0)} )^{l}} = pathc_{-1} (U_0,...,U_l).$$
By corollary \ref{SpectralGapDescent2}, we have that
$$\bigcup_{\tau \in \Sigma (-1)} Spec (\Delta_{\tau,0}^+) \setminus \lbrace 0 \rbrace \subseteq [f^{n-1} (\lambda), \infty ).$$
Therefore, by the lemma \ref{mixing descent lemma -partite case}, we get that 
\begin{dmath*}
\left\vert pathc_0 (U_0,...,U_l) - \left(\dfrac{n+1}{n} \right)^l pathc_{-1} (U_0,...,U_l) \right\vert \leq  l \left(\dfrac{n+1}{2n} \right)^{l} \left(\dfrac{(n+1)(1- f^{n-1} (\lambda))}{2} \right)   \sqrt{m (U_0) m(U_l) } =  \varepsilon_0  \sqrt{m (U_0) m(U_l) },
\end{dmath*}
and we are done. Next, assume that the theorem holds for $k-1$, then we have
\begin{dmath*}
\left\vert (k+1)^{l-1 - (k+1)} pathc_k (U_0,...,U_l) -  \left( \prod_{j=0}^k r_j^{l-j} \right) \dfrac{m(U_0) ... m (U_l)}{m (X^{(0)})^{l} } \right\vert \leq 
\left\vert (k+1)^{l-1 - (k+1)} pathc_{k} (U_0,...,U_l)  - r_k^{l-k} k^{l-1-k} pathc_{k-1} (U_0,...,U_l) \right\vert + r_k^{l-k}  \left\vert  k^{l-1-k} pathc_{k-1} (U_0,...,U_l) - \left( \prod_{j=0}^{k-1} r_j^{l-j} \right) \dfrac{m(U_0) ... m (U_l)}{m (X^{(0)})^{l} } \right\vert \leq 
\left\vert (k+1)^{l-1 - (k+1)} pathc_{k} (U_0,...,U_l)  - r_k^{l-k} k^{l-1-k} pathc_{k-1} (U_0,...,U_l) \right\vert + r_k^{l-k} \sum_{i=0}^{k-1} \varepsilon_i  \left( \prod_{j=i+1}^{k-1} r_j^{l-j} \right) \sqrt{m(U_0,...,U_i) m(U_{l-i},...,U_l )}.
\end{dmath*}

By corollary \ref{SpectralGapDescent2}, 
$$\bigcup_{\tau \in \Sigma (k-1)} Spec (\Delta_{\tau,0}^+) \setminus \lbrace 0 \rbrace \subseteq [f^{n-1-k} (\lambda), \infty) = [ \lambda_k, \infty).$$
And we finish by applying lemma \ref{mixing descent lemma -partite case}.

\end{proof}

Now we are ready to give the exact statement and proof of theorem \ref{Mixing for partite case - section 2} stated above:

\begin{corollary}
\label{mixing m(U_0,...,U_l) - partite}
Let $X$ be a pure $n$-dimensional, weighted, $(n+1)$-partite simplicial complex such that all the links of $X$ of dimension $>0$ are connected. Denote by $S_0,...,S_n$ the sides of $X$. If there is $\lambda > \frac{n-1}{n}$ such that
$$\bigcup_{\tau \in \Sigma (n-2)} Spec (\Delta_{\tau, 0}^+) \setminus \lbrace 0 \rbrace \subseteq [\lambda, \infty).$$
Then for every $1 \leq l \leq n$, there is a continuous function $\mathcal{E}_l (\lambda)$ such that
$$\lim_{\lambda \rightarrow 1} \mathcal{E}_l (\lambda) = 0 , $$
and such that  every non empty disjoint sets $U_0 \subseteq S_0,...,U_l \subseteq S_l$ the following inequalities holds:
\begin{dmath*}
\left\vert \dfrac{m(U_0,...,U_l)}{m(X^{(0)})} -  \dfrac{1}{(n+1)n(n-1)...(n-l+1)} \dfrac{m(U_0) ... m (U_l)}{m(S_0)...m(S_l)} \right\vert \leq 
\mathcal{E}_l (\lambda ) \min_{0 \leq i < j \leq l} \sqrt{\dfrac{m(U_i) m(U_j)}{m(S_i) m(S_j)}},
\end{dmath*}
and
\begin{dmath*}
\left\vert \dfrac{m(U_0,...,U_l)}{m(X^{(0)})} -  \dfrac{1}{(n+1)n(n-1)...(n-l+1)} \dfrac{m(U_0) ... m (U_l)}{m(S_0)...m(S_l)} \right\vert \leq 
\mathcal{E}_l (\lambda) \left( \dfrac{ m (U_0)...m (U_l) }{m(S_0)...m(S_l)}) \right)^{\frac{1}{l+1}} .
\end{dmath*}
\end{corollary}

\begin{proof}
Recall that for every $1 \leq l \leq n-1$ we have by proposition \ref{pathc_k for U_0,...,U_(k+1)} that 
$$   m(U_0,...,U_l) = \dfrac{pathc_{l-1} (U_0,...,U_l)}{l+1} .$$
Then by theorem \ref{mixing theorem - partite case} with $l, k=l-1$ we get
\begin{dmath*}
\left\vert m(U_0,...,U_l)  -   \left( \prod_{j=0}^{l-1} r_j^{l-j}  \right) \dfrac{m(U_0) ... m (U_l)}{m (X^{(0)})^{l} } \right\vert \leq \sum_{i=0}^{l-1}  \varepsilon_i  \left( \prod_{j=i+1}^{l-1} r_j^{l-j} \right)  \sqrt{m(U_0,...,U_i) m(U_{l-i},...,U_l )}\leq
\sqrt{m(U_0) m (U_l)} \sum_{i=0}^{l-1}  \varepsilon_i  \left( \prod_{j=i+1}^{l-1} r_j^{l-j} \right) \sqrt{\dfrac{m(U_0,...,U_i)}{m (U_0)} \dfrac{m(U_{l-i},...,U_l )}{m (U_l)}} \leq \sqrt{m(U_0) m (U_l)} \sum_{i=0}^{l-1}  \varepsilon_i  \left( \prod_{j=i+1}^{l-1} r_j^{l-j} \right),
\end{dmath*}
where the last inequality is due to the fact that for every $i$,
$$m(U_0,...,U_i) \leq m (U_0), m(U_{l-i},...,U_l) \leq m (U_l).$$
Dividing the above inequality by $m(X^{(0)})$ yields
\begin{dmath*}
\left\vert \dfrac{m(U_0,...,U_l)}{m(X^{(0)})}  -   \left( \prod_{j=0}^{l-1} r_j^{l-j}  \right) \dfrac{m(U_0) ... m (U_l)}{m (X^{(0)})^{l+1} } \right\vert \leq \dfrac{\sqrt{m(U_0) m (U_l)} }{m(X^{(0)})}\sum_{i=0}^{l-1}  \varepsilon_i  \left( \prod_{j=i+1}^{l-1} r_j^{l-j} \right).
\end{dmath*}
Note that
\begin{dmath*}
\prod_{j=0}^{l-1} r_j^{l-j}  = \left( \dfrac{n+1}{n} \right)^l \left( \dfrac{n}{n-1} \right)^{l-1} ... \left( \dfrac{n-l +2}{n-l+1} \right) = \dfrac{(n+1)^l}{n(n-1)...(n-l+1)}.
\end{dmath*}
Also note that by the proposition \ref{weight in n dim simplices} we have for every $0 \leq i \leq n$ that
$$m(S_i) = \dfrac{m(X^{(0)})}{n+1}.$$
Therefore we can write the above inequality as
\begin{dmath*}
\left\vert \dfrac{m(U_0,...,U_l)}{m(X^{(0)})}  -   \dfrac{1}{(n+1)n(n-1)...(n-l+1)} \dfrac{m(U_0) ... m (U_l)}{m (S_0)...m(S_l)} \right\vert \leq \dfrac{\sqrt{m(U_0) m (U_l)} }{m(X^{(0)})}\sum_{i=0}^{l-1}  \varepsilon_i  \left( \prod_{j=i+1}^{l-1} r_j^{l-j} \right) =\sqrt{\dfrac{m(U_0) m(U_l)}{m(S_0) m(S_l)}} (n+1) \sum_{i=0}^{l-1}  \varepsilon_i  \left( \prod_{j=i+1}^{l-1} r_j^{l-j} \right) .
\end{dmath*}
Note that for every $i$,
$$\lim_{\lambda \rightarrow 1} \varepsilon_i = 0 .$$
Therefore, denoting 
$$\mathcal{E}_l (\lambda) = (n+1) \sum_{i=0}^{l-1}  \varepsilon_i  \left( \prod_{j=i+1}^{l-1} r_j^{l-j} \right)  ,$$
yields the inequality
\begin{dmath*}
\left\vert \dfrac{m(U_0,...,U_l)}{m(X^{(0)})}  -   \dfrac{1}{(n+1)n(n-1)...(n-l+1)} \dfrac{m(U_0) ... m (U_l)}{m (S_0)...m(S_l)} \right\vert \leq \mathcal{E}_l (\lambda)  \sqrt{\dfrac{m(U_0) m(U_l)}{m(S_0) m(S_l)}}.
\end{dmath*}
with 
$$\lim_{\lambda \rightarrow 1} \mathcal{E}_l (\lambda) = 0 .$$
Note that by definition of $m(U_0,...,U_l)$, we have any reordering $\pi \in Sym (\lbrace 0,...,l \rbrace)$ that
$$m(U_0,...,U_l) = m(U_{\pi (0)},...,U_{\pi (l)} )$$
Therefore, for every $0 \leq i < j \leq l$, we have 
\begin{dmath*}
\left\vert \dfrac{m(U_0,...,U_l)}{m(X^{(0)})}  -   \dfrac{1}{(n+1)n(n-1)...(n-l+1)} \dfrac{m(U_0) ... m (U_l)}{m (S_0)...m(S_l)} \right\vert \leq \mathcal{E}_l (\lambda)  \sqrt{\dfrac{m(U_i) m(U_j)}{m(S_i) m(S_j)}}.
\end{dmath*}
Taking $i,j$ such that $m(U_i) m(U_j)$ is minimal yields the first inequality stated above. Next, multiply the inequalities for all the different choices of $0 \leq i<j \leq l$:
\begin{dmath*}
\left\vert \dfrac{m(U_0,...,U_l)}{m(X^{(0)})}  -   \dfrac{1}{(n+1)n(n-1)...(n-l+1)} \dfrac{m(U_0) ... m (U_l)}{m (S_0)...m(S_l)} \right\vert^{\frac{l(l+1)}{2}} \leq \mathcal{E}_l (\lambda)^{\frac{l(l+1)}{2}} \left( \sqrt{\dfrac{m(U_0)... m(U_l)}{m(S_0)... m(S_l)}} \right)^l.
\end{dmath*}
Taking both sides to the power $\frac{2}{l(l+1)}$ yields the second inequality stated above.
\end{proof}

\section{Geometric overlap property}

In \cite{Grom}, Gromov defined the geometric overlap property for complexes. We'll define a weighted analogue of this property. We shall need the following definition first:
\begin{definition}
Let $X$ be an $n$-dimensional simplicial complex and let $\phi : X^{(0)} \rightarrow \mathbb{R}^n$ be a map. The geometric extension of $\phi$ is the unique map $\widetilde{\phi} : X \rightarrow \mathbb{R}^n$ that extends $\phi$ affinely, i.e., for every $0 \leq l \leq n$ and every $\lbrace u_0,...,u_l \rbrace \in X^{(l)}$, $\widetilde{\phi}$ maps $\lbrace u_0,...,u_l \rbrace$ to the simplex in $\mathbb{R}^n$ spanned by $\phi (u_0),...,\phi (u_l)$.
\end{definition}

Using the above definition, the geometrical overlap property is defined as follows:
\begin{definition}
\label{geometric overlap definition}
Let $X$ be a $n$-dimensional simplicial complex and let $\varepsilon >0$. We shall say that $X$ has the $\varepsilon$-geometric overlap if for every map $\phi :  X^{(0)} \rightarrow \mathbb{R}^n$ and for the geometric  extension $\widetilde{\phi}$ of $\phi$, there is a point $O \in \mathbb{R}^n$ such that 
$$\vert \lbrace \sigma \in X^{(n)} : O \in \widetilde{\phi} (\sigma)   \rbrace \vert \geq \varepsilon \vert  X^{(n)} \vert .$$ 
\end{definition}
Generalizing to the weighted setting, the weighted geometrical overlap property is defined as follows:
\begin{definition}
Let $X$ be a weighted $n$-dimensional simplicial complex with a weight function $m$ and let $\varepsilon >0$. We shall say that $X$ has the weighted $\varepsilon$-geometric overlap if for every map $\phi :  X^{(0)} \rightarrow \mathbb{R}^n$ and for the geometric  extension $\widetilde{\phi}$ of $\phi$, there is a point $O \in \mathbb{R}^n$ such that 
$$m ( \lbrace \sigma \in X^{(n)} : O \in \widetilde{\phi} (\sigma)   \rbrace ) \geq \varepsilon m( X^{(n)} ) .$$ 
\end{definition}

\begin{remark}
When $m$ is the homogeneous weight, we have that $m(X^{(n)}) = \vert X^{(n)} \vert$ and the weighted definition coincides with the non-weighted definition.
\end{remark}

The aim of this section is to show that mixing results above imply weighted geometric overlap, both for the general case of mixing and for $(n+1)$-partite simplicial complexes.

\subsection{Geometric overlap from mixing for general simplicial complexes} 

We shall start by quoting the following result from the appendix:
\begin{theorem*}[\ref{weighted version of Pach - corollary}]

Let $V$ be a finite set and $m : V \rightarrow \mathbb{R}^+$ be some fixed map.
Then for $n \in \mathbb{N}$, there are constants $\omega(n) >0, c(n) >0$ such that for every $\phi : V \rightarrow \mathbb{R}^n$, one of the following holds:
\begin{enumerate}
\item There is $u \in V$ such that
$$m (u ) \geq \omega (n) \dfrac{1}{2(n+1)} m (V) .$$
\item There are pairwise disjoint sets $Q_0,...,Q_n \subset V$ such that for every $0 \leq i \leq n$,
$$ m(Q_i ) \geq c(n) \dfrac{1}{2(n+1)} m(V),$$
and
$$\bigcap_{(u_0,...,u_n) \in Q_0 \times ... \times Q_n } \overline{conv} (\phi (u_0),..., \phi (u_n) ) \neq \emptyset ,$$
where $\overline{conv} (u_0,...,u_n )$ is the closure of the convex hull of $\phi (u_0),..., \phi (u_n)$ (i.e., the closed simplex spanned by $\phi (u_0),...,\phi (u_n)$ in $\mathbb{R}^n$). 
\end{enumerate}
\end{theorem*}

Based on the above, we shall prove the following:

\begin{theorem}
\label{mixing implies geom-overlap}
Let $X$ be a weighted $n$-dimensional simplicial complex with a weight function $m$. Assume that are constants $0<\mathcal{A}_n, 0 \leq \mathcal{E}_n$ such that for every pairwise disjoint, non empty sets $U_0,...,U_n \subset X^{(0)}$ the following holds:
\begin{dmath*}
\left\vert m(U_0,...,U_n) -  \mathcal{A}_n \dfrac{m(U_0) ... m (U_n)}{m (X^{(0)})^{n} } \right\vert \leq  \mathcal{E}_n \left( m (U_0)...m (U_n) \right)^{\frac{1}{n+1}} .
\end{dmath*}
Assume farther that 
$$\dfrac{\mathcal{E}_n}{\mathcal{A}_n} < \left( \frac{c(n)}{2(n+1)} \right)^n ,$$
where $c(n)$ is the constant mentioned in theorem \ref{weighted version of Pach - corollary} above. 
Then $X$ has weighted $\varepsilon$-geometric overlap with
$$\varepsilon = \min \left\lbrace  \dfrac{\omega (n)}{2(n+1)^2}, \mathcal{A}_n \frac{n! c(n)}{2}   \left( \left(  \frac{c(n)}{2(n+1)} \right)^n - \dfrac{\mathcal{E}_n}{\mathcal{A}_n}  \right)\right\rbrace ,$$
where $\omega(n), c(n)$ are the constants mentioned in theorem \ref{weighted version of Pach - corollary} above. 
\end{theorem}

\begin{proof}
Fix $\phi : X^{(0)} \rightarrow \mathbb{R}^n$. If there is $u \in X^{(0)}$ such that 
$$m(u) \geq \omega (n) \dfrac{1}{2(n+1)} m (X^{(0)}) .$$
Then for $O= \phi (u)$ we have that (using proposition \ref{weight in n dim simplices}) 
\begin{dmath*}
m(\lbrace \sigma \in X^{(n)} : O \in \widetilde{\phi} (\sigma) \rbrace) \geq  {m(\lbrace \sigma \in X^{(n)} : u \in \sigma \rbrace)} = \dfrac{1}{n!} m(u) \geq  \dfrac{1}{n!}\omega (n) \dfrac{1}{2(n+1)} m (X^{(0)}).
\end{dmath*} 
Using proposition  \ref{weight in n dim simplices} again to deduce
$$m(X^{(0)} ) = (n+1)! m(X^{(n)}),$$
we get that
\begin{dmath*}
m(\lbrace \sigma \in X^{(n)} : O \in \widetilde{\phi} (\sigma) \rbrace) \geq \omega (n) \dfrac{1}{2(n+1)^2} m (X^{(n)}) \geq \varepsilon m (X^{(n)}).
\end{dmath*}
Assume now that for every $u \in X^{(0)}$, we have that
$$m(u) < \omega (n) \dfrac{1}{2(n+1)} m (X^{(0)}) .$$
By theorem \ref{weighted version of Pach - corollary} stated above we have that there pairwise disjoints sets $Q_0,...,Q_n$ such that for every $0 \leq i \leq n$,
$$ m(Q_i ) \geq c(n) \dfrac{1}{2(n+1)} m(X^{(0)}),$$
and
$$\bigcap_{(u_0,...,u_n) \in Q_0 \times ... \times Q_n } \overline{conv} (\phi (u_0),..., \phi (u_n) ) \neq \emptyset .$$
By our assumption on mixing, we also have 
\begin{dmath*}
\left\vert m(Q_0,...,Q_n) -  \mathcal{A}_l \dfrac{m(Q_0) ... m (Q_n)}{m (X^{(0)})^{n} } \right\vert \leq  \mathcal{E}_n \left( m (Q_0)...m (Q_n) \right)^{\frac{1}{n+1}} .
\end{dmath*}
which yields
\begin{dmath*}
m(Q_0,...,Q_n) \geq  \mathcal{A}_n \left( \dfrac{m(Q_0) ... m (Q_n)}{m (X^{(0)})^{n} } -
 \dfrac{\mathcal{E}_n}{\mathcal{A}_n} \left( m (Q_0)...m (Q_n) \right)^{\frac{1}{n+1}} \right) = \mathcal{A}_n \left( m (Q_0)...m (Q_n) \right)^{\frac{1}{n+1}} \left( \left( \dfrac{(m(Q_0) ... m (Q_n))^{\frac{1}{n+1}}}{m (X^{(0)}) } \right)^n - \dfrac{\mathcal{E}_n}{\mathcal{A}_n}  \right) \geq \mathcal{A}_n \frac{c(n)}{2(n+1)}  m (X^{(0)}) \left( \left(  \frac{c(n)}{2(n+1)} \right)^n - \dfrac{\mathcal{E}_n}{\mathcal{A}_n}  \right) = \mathcal{A}_n \frac{n! c(n)}{2}   \left( \left(  \frac{c(n)}{2(n+1)} \right)^n - \dfrac{\mathcal{E}_n}{\mathcal{A}_n}  \right) m (X^{(n)}) \geq \varepsilon m(X^{(n)}).
 \end{dmath*}
(Again, we used the fact that $m(X^{(0)}) = (n+1)! m(X^{(n)}) $)

\end{proof}
Now we are ready to give the exact statement and the proof of theorem \ref{Geometric overlap in the general case - section 2} stated above:
\begin{corollary}
\label{geometric overlap}
Let $X$ be a pure $n$-dimensional weighted simplicial complex such that all the links of $X$ of dimension $>0$ are connected. There is a continuous function $\varepsilon (\lambda , \kappa) : [0,1] \times [1,2] \rightarrow \mathbb{R}$ such that:
\begin{itemize}
\item We have that
$$\lim_{(\lambda , \kappa ) \rightarrow (1,1)}  \varepsilon (\lambda , \kappa) = \min \left\lbrace  \dfrac{\omega (n)}{2(n+1)^2}, \frac{n! c(n)}{2}  \left(  \frac{c(n)}{2(n+1)} \right)^n \right\rbrace >0 ,$$
when $c(n), \omega (n)$ is the constants as in the above theorem.
\item For a given $\kappa, \lambda$ if $\lambda > \frac{n-1}{n}$, $\varepsilon (\lambda , \kappa) >0$, and 
$$\bigcup_{\tau \in \Sigma (n-2)} Spec (\Delta_{\tau, 0}^+) \setminus \lbrace 0 \rbrace \subseteq [\lambda, \kappa],$$
then $X$ has $\varepsilon (\lambda , \kappa)$-geometric overlap.

\end{itemize} 
\end{corollary}

\begin{proof}
Combine the above theorem with corollary \ref{mixing m(U_0,...,U_l)}.
\end{proof}

\subsection{Geometric overlap from mixing for partite complexes}

The proof of geometric overlap as a consequence of mixing in the $(n+1)$-partite case is almost the same as in the general case. For the convenience of the reader we will repeat all the arguments. First, let us recall the weighted version of Pach's theorem for the $(n+1)$-partite case proven in the appendix:
\begin{theorem*}[\ref{weighted version of Pach}]

Let $V$ be a finite set and $m : V \rightarrow \mathbb{R}^+$ be some fixed map. For $U \subseteq V$, denote 
$$m(U) = \sum_{u \in U} m(u).$$
Then for $n \in \mathbb{N}$, there are constants $0 < \omega (n) \leq 1, c(n) >0$ such that for every $\phi : V \rightarrow \mathbb{R}^n$ and every disjoint partition of $V$, $S_0,...,S_n$, one of the following holds:
\begin{enumerate}
\item There is a vertex $u \in V$ such that 
$$m(u) \geq \omega (n) \min \lbrace m(S_0), ..., m(S_n) \rbrace.$$
\item There are sets $Q_0 \subseteq S_0,...,Q_n \subseteq S_n$ such that for every $0 \leq i \leq n$,
$$ m(Q_i ) \geq c(n)  m(S_i),$$
and
$$\bigcap_{(u_0,...,u_n) \in Q_0 \times ... \times Q_n } \overline{conv} (\phi (u_0),..., \phi (u_n) ) \neq \emptyset ,$$
where $\overline{conv} (\phi(u_0),...,\phi(u_n) )$ is the closed convex hull of $\phi (u_0),..., \phi (u_n)$ (i.e., the closure of the simplex spanned by $\phi (u_0),...,\phi (u_n)$ in $\mathbb{R}^n$). 
\end{enumerate}
\end{theorem*}

Based on the above, we shall prove the following:

\begin{theorem}
\label{mixing implies geom-overlap}
Let $X$ be a weighted, $(n+1)$-partite, pure $n$-dimensional simplicial complex with a weight function $m$. Denote the sides of $X$ by $S_0,...,S_n$. Assume that there is a constant  $0 \leq \mathcal{E}_n$ such that for every non empty sets $U_0 \subseteq S_0,...,U_n \subseteq S_n$ the following holds:
\begin{dmath*}
\left\vert \dfrac{m(U_0,...,U_n)}{m(X^{(0)})} -  \dfrac{1}{(n+1)!} \dfrac{m(U_0) ... m (U_n)}{m(S_0)...m(S_n)} \right\vert \leq 
\mathcal{E}_n \left( \dfrac{ m (U_0)...m (U_n) }{m(S_0)...m(S_n)} \right)^{\frac{1}{n+1}} .
\end{dmath*}
Assume further that 
$$\mathcal{E}_n <\dfrac{c(n)^n}{n!} $$
Then $X$ has weighted $\varepsilon$-geometric overlap  with
$$\varepsilon = \min \left\lbrace\dfrac{\omega (n)}{(n+1)^2}, c(n) \left( c(n)^n  - (n+1)! \mathcal{E}_n  \right) \right\rbrace .$$

\end{theorem}

\begin{proof}
Fix $\phi : X^{(0)} \rightarrow \mathbb{R}^n$. Note that by definition we have for every $i$ that
$$m(S_i) = \dfrac{m(X^{(0)})}{n+1} .$$
If there is $u \in X^{(0)}$ such that 
$$m(u) \geq \omega (n) m(S_i) = \omega (n) \dfrac{m(X^{(0)})}{n+1}  .$$
Then for $O= \phi (u)$ we have that (using proposition \ref{weight in n dim simplices}) 
\begin{dmath*}
m(\lbrace \sigma \in X^{(n)} : O \in \widetilde{\phi} (\sigma) \rbrace) \geq  {m(\lbrace \sigma \in X^{(n)} : u \in \sigma \rbrace)} = \dfrac{1}{n!} m(u) \geq  \dfrac{1}{n!}\omega (n) \dfrac{1}{n+1} m (X^{(0)}).
\end{dmath*} 
Using proposition  \ref{weight in n dim simplices} again to deduce
$$m(X^{(0)} ) = (n+1)! m(X^{(n)}),$$
we get that
\begin{dmath*}
m(\lbrace \sigma \in X^{(n)} : O \in \widetilde{\phi} (\sigma) \rbrace) \geq \omega (n) \dfrac{1}{(n+1)^2} m (X^{(n)}) \geq \varepsilon m (X^{(n)}).
\end{dmath*}
Assume now that for every $u \in X^{(0)}$, we have that
$$m(u) < \omega (n) \dfrac{1}{n+1} m (X^{(0)}) =  \omega (n) \min \lbrace m(S_0),...,m(S_n) \rbrace .$$
By theorem \ref{weighted version of Pach} stated above we have that there pairwise disjoints sets $Q_0 \subseteq S_0,...,Q_n \subseteq S_n$ such that for every $0 \leq i \leq n$,
$$ m(Q_i ) \geq c(n) m(S_i),$$
and
$$\bigcap_{(u_0,...,u_n) \in Q_0 \times ... \times Q_n } \overline{conv} (\phi (u_0),..., \phi (u_n) ) \neq \emptyset .$$
By our assumption on mixing, we also have 
\begin{dmath*}
\left\vert \dfrac{m(Q_0,...,Q_n)}{m(X^{(0)})} -  \dfrac{1}{(n+1)!} \dfrac{m(Q_0) ... m (Q_n)}{m(S_0)...m(S_n) } \right\vert \leq  \mathcal{E}_n \left( \dfrac{m (Q_0)...m (Q_n)}{m(S_0)...m(S_n)} \right)^{\frac{1}{n+1}} .
\end{dmath*}
which yields
\begin{dmath*}
m(Q_0,...,Q_n) \geq  m(X^{(0)}) \left( \dfrac{1}{(n+1)!} \dfrac{m(Q_0) ... m (Q_n)}{m(S_0)...m(S_n) } -
 \mathcal{E}_n \left( \dfrac{m (Q_0)...m (Q_n)}{m(S_0)...m(S_n)} \right)^{\frac{1}{n+1}} \right) = m(X^{(0)}) \left(\dfrac{m (Q_0)...m (Q_n)}{m(S_0)...m(S_n)} \right)^{\frac{1}{n+1}} \left( \dfrac{1}{(n+1)!} \left(  \dfrac{m(Q_0) ... m (Q_n)}{m(S_0)...m(S_n) } \right)^{\frac{n}{n+1}} - \mathcal{E}_n  \right) \geq m(X^{(0)}) c(n) \left( \dfrac{c(n)^n}{(n+1)!}  - \mathcal{E}_n  \right)  =  m(X^{(n)}) (n+1)! c(n) \left( \dfrac{c(n)^n}{(n+1)!}  - \mathcal{E}_n  \right) = m(X^{(n)}) c(n) \left( c(n)^n -(n+1)! \mathcal{E}_n  \right)
  \geq \varepsilon m(X^{(n)}).
 \end{dmath*}
(Again, we used the fact that $m(X^{(0)}) = (n+1)! m(X^{(n)}) $)

\end{proof}
Now we are ready to give the exact statement and the proof of theorem \ref{Geometric overlap for partite case - section 2} stated above:
\begin{corollary}
\label{geometric overlap for partite case}
Let $X$ be a pure $n$-dimensional, weighted, $(n+1)$-partite simplicial complex such that all the links of $X$ of dimension $>0$ are connected.  There is a continuous function $\varepsilon (\lambda) : [0,1]  \rightarrow \mathbb{R}$ such that:
\begin{itemize}
\item We have that
$$\lim_{\lambda  \rightarrow 1}  \varepsilon (\lambda)  = \min \left\lbrace\dfrac{\omega (n)}{(n+1)^2}, c(n)^{n+1}  \right\rbrace  >0,$$
when $\omega(n), c(n)$ are the constants as in the above theorem.
\item For a given $\lambda$ if $\lambda > \frac{n-1}{n}$, $\varepsilon (\lambda) >0$, and 
$$\bigcup_{\tau \in \Sigma (n-2)} Spec (\Delta_{\tau, 0}^+) \setminus \lbrace 0 \rbrace \subseteq [\lambda, \infty),$$
then $X$ has $\varepsilon (\lambda)$-geometric overlap.

\end{itemize}

\end{corollary}

\begin{proof}
Combine the above theorem with corollary \ref{mixing m(U_0,...,U_l) - partite}.
\end{proof}

\section{Examples}

\subsection{Groups acting on simplicial complexes}
Let $\Gamma$ be a discrete group acting simplicially and cocompactly on a connected infinite simplicial complex $\widetilde{X}$. Denote by $d$ the distance on $\widetilde{X}$. Assume that
$$\min \lbrace d(v,g.v) : g \in \Gamma \setminus \lbrace e \rbrace, v \in \widetilde{X}^{(0)} \rbrace \geq 3.$$
This assumption implies that every link in $X = \widetilde{X} / \Gamma$ which is not the trivial link, $X_\emptyset = X$, is isomorphic to a link in $\widetilde{X}$. This leads to the following observation: under the above assumptions, if $\widetilde{X}$ above is pure $n$-dimensional such that all the links of $\widetilde{X}$ are connected (apart from the $0$-dimensional links) and such that there is $\lambda > \frac{n-1}{n}$ such that
$$\inf_{\tau \in  \widetilde{X}^{(n-2)}} \lambda ( \widetilde{X}_\tau) \geq \lambda, $$
then $X = \widetilde{X} / \Gamma$ has $\lambda$-local spectral expansion. If in addition there is $\kappa <2$ such that 
$$\sup_{\tau \in  \widetilde{X}^{(n-2)}} \kappa ( \widetilde{X}_\tau) \leq \kappa, $$
then $X = \widetilde{X} / \Gamma$ has two sided $(\lambda,\kappa)$-local spectral expansion. \\
Examples of this sort appear in the explicit construction of Ramanujan complexes in \cite{LSV}. Specifically, it is shown there that given any prime $q$ and any $r \in \mathbb{N}$ ($r >1$ if $q=2$), there is an affine building $\widetilde{X}$ of type $\widetilde{A}_n$ and thickness $q^r +1$ and a sequence of $\Gamma_i$ acting simplicially and cocompactly on $\widetilde{X}$ such that the quotients $X_i = \widetilde{X} / \Gamma_i $ are arbitrarily large, i.e.,
$$\lim_{i \rightarrow \infty} \left\vert \left( \widetilde{X} / \Gamma_i \right)^{(0)} \right\vert = \infty,$$
and such that for every $i$,
$$\min \lbrace d(v,g.v) : g \in \Gamma_i \setminus \lbrace e \rbrace, v \in \widetilde{X}^{(0)} \rbrace \geq 3.$$
We will not review the theory of Tits buildings here, but only recall that an affine building of type $\widetilde{A}_n$ is a connected, pure $n$-dimensional simplicial complex with connected links (apart from the $0$-dimensional links) such that for every $\tau \in \widetilde{X}^{(n-2)}$ we have one of the following options:
\begin{enumerate}
\item $\widetilde{X}_\tau$ is a complete bipartite graph. In this case 
$$\lambda (\widetilde{X}_\tau ) =1, \kappa (\widetilde{X}_\tau ) =2.$$ 
\item $\widetilde{X}_\tau$ is a spherical building of type $A_2$ and thickness $q^r+1$. In this case (see \cite{FH}), 
$$\lambda (\widetilde{X}_\tau ) =1 - \dfrac{\sqrt{q^r}}{q^r+1}, \kappa (\widetilde{X}_\tau ) =2.$$
\end{enumerate}
It is clear that if $q,r$ are large enough with respect to $n$ (for instance, if $q^r > n^2$), then for every $\tau \in \widetilde{X}^{(n-2)}$ we have 
$$\lambda (\widetilde{X}_\tau ) \geq  1 - \dfrac{\sqrt{q^r}}{q^r+1} >  \dfrac{n-1}{n}$$ 
and $X_i = \widetilde{X} / \Gamma_i $ will have $(1 - \frac{\sqrt{q^r}}{q^r+1})$-local spectral expansion for every $i$. In this example, achieving mixing in the general case is hopeless because the $1$-dimensional links are bipartite graphs (therefore one should not expect two-sided local spectral expansion). However, Choosing the right $\Gamma_i$'s, one can make sure that $\widetilde{X} / \Gamma_i$ is a $(n+1)$-partite simplicial complex and therefore has the mixing result for the $(n+1)$-partite case. From this using corollary \ref{geometric overlap for partite case}, one can construct an infinite family of simplicial complexes with the same $\varepsilon$-geometric overlap (given the $q$ is chosen to be large enough). A similar construction was given in \cite{FGLNP} but the arguments used there to show mixing (and therefore geometric overlap) where completely different. \\ 
More generally, groups acting on simplicial Tits-building under the conditions mentioned above are likely to provide examples of simplicial complexes with local spectral expansion, provided that the thickness of the building is large enough (we recall that the spectral gaps of all $1$-dimensional links that appear in a simplicial Tits-building where calculated explicitly in \cite{FH}). Also, if can take quotients such that the simplicial complexes are $(n+1)$-partite, then one can get mixing and geometric overlap for large enough thickness. The fact that quotients of any affine building yields a simplicial complexes with geometric overlap was conjectured in \cite{LubHighDim}, but as far as we know, we are the first to prove it.

\subsection{Random complexes}

First let us describe the model $X(N,p)$ for random complexes. $X \sim X(N,p)$ is randomly chosen in the following way: 
\begin{enumerate}
\item $X$ has $N$ vertices $\lbrace v_1,...,v_N \rbrace$.
\item For every two vertices $v_i,v_j$, there is an edge between $v_i,v_j$ with probability $p$.
\item After all the edges are randomly chosen, $X$ is completed to be a flag complex, i.e., for every set $\lbrace v_{i_1},...,v_{i_k} \rbrace$ we have:
$$\lbrace v_{i_1},...,v_{i_k} \rbrace \in X^{(k-1)} \Leftrightarrow \forall 1 \leq j < l \leq k, \lbrace v_{i_j} ,v_{i_l}  \rbrace \in X^{(1)} .$$
\end{enumerate}
In random complex theory, one is usually interested in asymptotic properties when $N \rightarrow \infty$. We shall say that $X \sim X(N,p)$ has some property $\mathcal{P}$ with high probability if
$$\lim_{N \rightarrow \infty} \mathbb{P} ( X \text{ has property } \mathcal{P} ) =1 .$$
In \cite{Kahle} it is shown that for every $n \in \mathbb{N}$, if $X \sim X(N,p)$ such that there is some $\varepsilon >0$ such that
$$p \geq \left( \dfrac{\left(\frac{n}{2} + 1 + \varepsilon \right) \log N}{N} \right)^{\frac{1}{n+1}} ,$$
then (the $n$-skeleton of) $X$ has the following properties with high probability:
\begin{enumerate}
\item $X$ is pure $n$-dimensional (\cite{Kahle}[Lemma 3.1]).
\item $X$ is connected (this is simply due to \textit{Erd\H{o}s-R\'{e}nyi theorem}).
\item All the $1,...,(n-1)$-dimensional links of are connected (\cite{Kahle}[Proof of part (1) of Theorem 1.1]).
\item For every $\alpha >0$ and every $\sigma \in X^{(n-2)}$, $\lambda (X_{\sigma} ) \geq 1 - \alpha$ and  $\kappa (X_{\sigma} ) \leq 1 + \alpha$.
(\cite{Kahle}[Proof of part (1) of Theorem 1.1] only proves this result only for $\lambda (X_{\sigma} )$, but using \cite{HKP}[Theorem 1.1], the result for $\kappa (X_{\sigma} )$ also follows).
\end{enumerate}

Therefore, we get that for suitable $p (N,n)$, $X \sim X(N,p)$ is with high probability (the $n$-skeleton of ) $X$ is pure $n$-dimensional simplicial complex that has for every $\alpha >0$ a two-sided $(1-\alpha, 1+\alpha)$-local spectral expansion. Therefore, if $\alpha$ is small enough, by corollary \ref{geometric overlap} these complexes will have $\varepsilon (1-\alpha , 1+\alpha)$-geometric overlap, where $\varepsilon (1-\alpha , 1+\alpha) >0$.

\appendix
\section{A weighted version of a result by Pach}

The aim of this appendix is to prove the following version of a theorem by Pach proven in \cite{Pach} (all the ideas of this proof appear in \cite{Pach}, the aim of the appendix is just to adapt the ideas to the weighted setting).

\begin{theorem}
\label{weighted version of Pach}
Let $V$ be a finite set and $m : V \rightarrow \mathbb{R}^+$ be some fixed map. For $U \subseteq V$, denote 
$$m(U) = \sum_{u \in U} m(u).$$
Then for $n \in \mathbb{N}$, there are constants $0 < \omega (n) \leq 1, c(n) >0$ such that for every $\phi : V \rightarrow \mathbb{R}^n$ and every disjoint partition of $V$, $S_0,...,S_n$, one of the following holds:
\begin{enumerate}
\item There is a vertex $u \in V$ such that 
$$m(u) \geq \omega (n) \min \lbrace m(S_0), ..., m(S_n) \rbrace.$$
\item There are sets $Q_0 \subseteq S_0,...,Q_n \subseteq S_n$ such that for every $0 \leq i \leq n$,
$$ m(Q_i ) \geq c(n)  m(S_i),$$
and
$$\bigcap_{(u_0,...,u_n) \in Q_0 \times ... \times Q_n } \overline{conv} (\phi (u_0),..., \phi (u_n) ) \neq \emptyset ,$$
where $\overline{conv} (\phi(u_0),...,\phi(u_n) )$ is the closed convex hull of $\phi (u_0),..., \phi (u_n)$ (i.e., the closure of the simplex spanned by $\phi (u_0),...,\phi (u_n)$ in $\mathbb{R}^n$). 
\end{enumerate}
\end{theorem}

To prove the theorem we shall need a few preliminary results.
\begin{lemma}
\label{T_i lemma}
Let $V$ be a finite set, $S_0,...,S_n$ be a disjoint partition of $V$ and $m : V \rightarrow \mathbb{R}^+$ be some fixed map. For a set $A \subseteq S_0 \times ... \times S_n$, denote 
$$e (A) = \sum_{(u_0,...,u_n) \in A} m(u_0) m(u_1) ... m(u_n) .$$
Assume there is $\beta >0$ and a set $A \subseteq S_0 \times ... \times S_n$ such that 
$$e (A) \geq \beta m(S_0) ... m(S_n) .$$
Let $0 < \varepsilon_1 \leq \varepsilon_2  < 1$, such that 
$$\dfrac{1-\varepsilon_1}{1-\varepsilon_2} (1- \varepsilon_2^{n+1}) <1.$$
Then there are non empty subsets $T_0 \subseteq S_0,...,T_n \subseteq S_n$ and a constant $ \alpha = \alpha (n, \varepsilon_1, \varepsilon_2)$, $0 < \alpha < 1$, such that the following holds:
\begin{enumerate}
\item For any $0 \leq i \leq n$, we have that
$$ m (T_i) \geq \beta^{\frac{1}{\alpha}} m (S_i) .$$
\item 
$$e (A \cap (T_0 \times ... \times T_n)) \geq \beta m(T_0) ... m(T_n) .$$
\item For every subsets $Q_0 \subseteq T_0, ... ,Q_n \subseteq T_n$ with $\varepsilon_1 \leq \frac{m(Q_i)}{m(T_i)} \leq \varepsilon_2$ for every $0 \leq i \leq n$, we have that $e(A \cap (Q_0 \times ... \times Q_n) ) >0$.
\end{enumerate}

\end{lemma}

\begin{proof}
Fix $0 < \varepsilon_1 \leq \varepsilon_2  < 1$, such that 
$$\dfrac{1-\varepsilon_1}{1-\varepsilon_2} (1- \varepsilon_2^{n+1}) <1.$$
Define
$$g(x) = \sum_{i=0}^n (1-\varepsilon_1)^{1-x} \varepsilon_2^{i(1-x)}= (1-\varepsilon_1)^{1-x} \dfrac{1-\varepsilon_2^{n+1-(n+1)x}}{1-\varepsilon_2^{1-x}} .$$
It is easy to see that $g(x)$ is continuous on the interval $[0,1)$ and that 
$$g(0) = \dfrac{1-\varepsilon_1}{1-\varepsilon_2} (1- \varepsilon_2^{n+1}) <1 .$$
Choose some $\alpha >0$ such that $g( \alpha ) <1$ (as noted above the value of such $\alpha$ depends on $\varepsilon_1, \varepsilon_2$ and $n$). 
Next, choose $T_0 \subseteq S_0,...,T_n \subseteq S_n$ such that 
$$\dfrac{e (A \cap (T_0 \times ... \times T_n)) }{\left( m(T_0) ... m(T_n) \right)^{1- \alpha}} $$
is maximal. By maximality we get that
\begin{dmath*}
\dfrac{e (A \cap (T_0 \times ... \times T_n)) }{\left( m(T_0) ... m(T_n) \right)^{1- \alpha}} \geq \dfrac{e (A ) }{\left( m(S_0) ... m(S_n) \right)^{1- \alpha}} = \dfrac{e (A ) }{\left( m(S_0) ... m(S_n) \right)} \left( m(S_0) ... m(S_n) \right)^{\alpha} \geq \beta \left( m(S_0) ... m(S_n) \right)^{\alpha} \geq \beta \left( m(T_0) ... m(T_n) \right)^{\alpha} .
\end{dmath*}
which yields 
$$e (A \cap (T_0 \times ... \times T_n)) \geq \beta m(T_0) ... m(T_n) .$$
From the same computation, combined with the inequality
$$\left( m(T_0) ... m(T_n) \right)^{\alpha} \geq \dfrac{e (A \cap (T_0 \times ... \times T_n)) }{\left( m(T_0) ... m(T_n) \right)^{1- \alpha}} ,$$
we get that
$$\left( m(T_0) ... m(T_n) \right)^{\alpha} \geq \beta \left( m(S_0) ... m(S_n) \right)^{\alpha}.$$
Therefore,
$$ m(T_0) ... m(T_n)  \geq \beta^{\frac{1}{\alpha}} m(S_0) ... m(S_n) .$$
Which yields for every $i$:
$$ m (T_i) \geq \beta^{\frac{1}{\alpha}} m (S_i) \dfrac{m(S_0)}{m(T_0)} ... \dfrac{m(S_{i-1})}{m(T_{i-1})} \dfrac{m(S_{i+1})}{m(T_{i+1})} ... \dfrac{m(S_n)}{m(T_n)}  \geq  \beta^{\frac{1}{\alpha}} m (S_i) .$$
Next, let $Q_0 \subseteq T_0, ... ,Q_n \subseteq T_n$ with $\varepsilon_1 \leq \frac{m(Q_i)}{m(T_i)} \leq \varepsilon_2$ for every $0 \leq i \leq n$. Then 
\begin{dmath*}
e(A \cap (Q_0 \times ... \times Q_n) ) = e(A \cap (T_0 \times ... \times T_n) ) - e(A \cap ((T_0 \setminus Q_0 )\times T_1 \times  ... \times T_n) ) - e(A \cap ( Q_0 \times  (T_1 \setminus Q_1)  \times  T_2 \times ... \times T_n) ) - ... - e(A \cap (Q_0\times Q_1 \times ... \times Q_{l-1} \times (T_n \setminus Q_n )) ) .
\end{dmath*}
Note that 
\begin{dmath*}
\dfrac{e (A \cap ((T_0 \setminus Q_0 )\times T_1 \times  ... \times T_n) ) }{\left( m((T_0 \setminus Q_0 )) m(T_1) ... m(T_n) \right)^{1-\alpha}} \leq \dfrac{e (A \cap (T_0 \times T_1 \times  ... \times T_n) ) }{\left( m((T_0 ) m(T_1) ... m(T_n) \right)^{1-\alpha}}.
\end{dmath*} 
Therefore
\begin{dmath*}
e (A \cap ((T_0 \setminus Q_0 )\times T_1 \times  ... \times T_n)  \leq \left( \dfrac{m((T_0 \setminus Q_0 ))}{m(T_0)} \right)^{1-\alpha} e (A \cap (T_0 \times T_1 \times  ... \times T_n)) \leq (1-\varepsilon_1)^{1- \alpha}  e(A \cap (T_0 \times T_1 \times  ... \times T_n)).
\end{dmath*}
In the same manner, for every $0 \leq i \leq n$ we have that 
\begin{dmath*}
e (A \cap ( Q_0 \times Q_1 \times  ... \times Q_{i-1} \times (T_i \setminus Q_i ) \times T_{i+1} \times ... \times T_n)  \leq \\ 
 (1-\varepsilon_1)^{1- \alpha} \varepsilon_2^{i(1-\alpha)} e (A \cap (T_0 \times T_1 \times  ... \times T_n)).
\end{dmath*}
Therefore
\begin{dmath*}
e(A \cap (Q_0 \times ... \times Q_n) ) \geq  e(A \cap (T_0 \times ... \times T_n) )\left( 1 - \sum_{i=0}^n  (1-\varepsilon_1)^{1- \alpha} \varepsilon_2^{i(1-\alpha)} \right) =   e(A \cap (T_0 \times ... \times T_n) ) ( 1- g( \alpha ) ) > 0 .
\end{dmath*}
\end{proof}

\begin{remark}
The condition 
$$\dfrac{1-\varepsilon_1}{1-\varepsilon_2} (1- \varepsilon_2^{n+1}) <1$$
obviously holds when $\varepsilon_1=\varepsilon_2$. Therefore for every $0< \varepsilon_1$, if one takes $\varepsilon_2$ such that $\varepsilon_2 - \varepsilon_1$ is small enough, then the condition above holds. Explicitly, one can always take $\varepsilon_2 = \varepsilon_1 + (1-\varepsilon_1) \varepsilon_1^{n+1}$:
\begin{dmath*}
\dfrac{1-\varepsilon_1}{1-\varepsilon_2} (1- \varepsilon_2^{n+1}) = \dfrac{1-\varepsilon_1}{1-\varepsilon_1 - (1-\varepsilon_1) \varepsilon_1^{n+1}} (1- (\varepsilon_1 + (1-\varepsilon_1) \varepsilon_1^{n+1} )^{n+1})  < \dfrac{1-\varepsilon_1}{1-\varepsilon_1 - (1-\varepsilon_1) \varepsilon_1^{n+1}} (1 - \varepsilon_1^{n+1} ) = 1.
\end{dmath*}
\end{remark}

Another result we'll need is a Boros-F\"{u}redi type theorem for the weighted case taken from Karasev \cite{Kara}:
\begin{theorem}
\label{Karasev result}
Let $\mu_0,...,\mu_n$ be discrete probability measures on $\mathbb{R}^n$. A random $n$-simplex is a simplex spanned by $x_0,...,x_n \in \mathbb{R}^n$ where for every $i$, $x_i$ is distributed according to the measure $\mu_i$. Then for any choice of $\mu_0,...,\mu_n$ there is a point $O \in \mathbb{R}^n$ such that the probability of a random $n$-simplex to contain $O$ is $\geq \frac{1}{(n+1)!}$.
\end{theorem}

\begin{remark}
The result stated in \cite{Kara}[Theorem 1] is for absolutely continuous probability measures, but in the remark after \cite{Kara}[Theorem 2] it is explained how to pass from continuous probability measures to discrete measures.
\end{remark}

Next, recall that a $(n+1)$-tuple of convex sets in $\mathbb{R}^n$ is called separated any $j$ of them can be strictly separated from the other $n+1-j$ by a hyperplane. A family of convex sets in $\mathbb{R}^n$ is called separated, if any $(n+1)$-tuple of the family is separated. The following theorem taken from \cite{GPW} gives a nice characterization of separated families:
\begin{theorem}
A family of convex sets in $\mathbb{R}^n$ is separated if and only if no $n+1$ of its members 
can be intersected by a hyperplane. 
\end{theorem}

\begin{corollary}
\label{intersection}
Let $C_0,...,C_n, \lbrace O \rbrace$ be a separated family of convex sets in $\mathbb{R}^n$ ($O \in \mathbb{R}^l$ is a single point). If there are $p_0 \in C_0 ,..., p_n \in C_n$ such that $O \in conv (p_0,...,p_l)$, then for every $q_0 \in C_0,...,q_n \in C_n$, we have that $O \in conv (q_0,...,q_n)$. 
\end{corollary}

\begin{proof}
Assume towards contradiction that there are points $q_0 \in C_0,...,q_n \in C_n$, such that $O \notin conv (q_0,...,q_n)$.  For $0 \leq \gamma \leq 1$ and $0 \leq i \leq n$, denote $\gamma p_i + (1-\gamma) q_i \in C_i$ the corresponding point on the interval connecting $p_i$ and $q_i$. From continuity we get that there is some $\gamma'$ such that $O$ is in on an $n-1$ face of the simplex spanned by  $\gamma' p_0,...,\gamma' p_n$, but this is in contradiction (by the above theorem) to the fact the $C_0,...,C_n, \lbrace O \rbrace$ is a separated family.
\end{proof}

Last, we'll need the following separation with respect to weight result:
\begin{lemma}
\label{Q_i generation}
Let $V$ a finite set, $m : V \rightarrow \mathbb{R}^+$, $S_0,...,S_n$ a disjoint partition of $V$, $T_0 \subseteq S_0,...,T_n \subseteq S_n$ non empty sets and $\phi : V \rightarrow \mathbb{R}^n$ be a map that sends $V$ to points in general position in $\mathbb{R}^n$. \\
Assume that for every $u \in V$, we have that
$$m(u) \leq \frac{\min \lbrace \frac{m(T_0)}{m (S_0)} ,...,  \frac{m(T_n)}{m (S_n)} \rbrace}{2^{2 + n2^{n}}} \dfrac{1}{n} \min \lbrace m(S_0),..., m(S_n) \rbrace .$$
Then for every point $O \in \mathbb{R}^n \setminus (\phi (T_0) \cup ... \cup \phi (T_n))$, there are non empty sets $Q_0 \subseteq T_0,..., Q_n \subseteq T_n$ such that 
$$\forall 0 \leq i \leq n, m(Q_i) \geq \dfrac{1}{1+2^{n2^{n}}} m (T_i) ,$$
and $\overline{conv} (\phi (Q_0)), ..., \overline{conv} (\phi (Q_n)), \lbrace O \rbrace$ is a separated family in $\mathbb{R}^n$.
\end{lemma}

\begin{proof}
Denote 
$$\Omega =  \frac{\min \lbrace \frac{m(T_0)}{m (S_0)} ,...,  \frac{m(T_n)}{m (S_n)} \rbrace}{2^{2 + n2^{n}}} .$$
For every hyperplane $H$ in $\mathbb{R}^n$ denote by $H^+, H^-$ the two open half spaces "above" and "below" $H$. We'll start by choosing $T_0' \subseteq T_0,..., T_n' \subseteq T_n$ such that $\overline{conv} (\phi (T_0')),...,  \overline{conv} (\phi (T_n'))$ is a separated $(n+1)$-tuple in $\mathbb{R}^n$ with $m (T_0 ' ),...,m (T_n ')$ that are not "too small". This is done using the discrete ham sandwich theorem (see for instance \cite{Hill}). Start with $T_0,...,T_n$, if $\overline{conv} (\phi (T_0)),...,  \overline{conv} (\phi (T_n))$ is a separated $n+1$ tuple, we are done. \\
Otherwise, say that $\overline{conv} (\phi (T_0)) ,..., \overline{conv} (\phi (T_j))$ are not separated by a hyperplane from $\overline{conv} (\phi (T_{j+1})) ,..., \overline{conv} (\phi (T_n))$. By the ham sandwich theorem there is a hyperplane $H$ such that for every $0 \leq i \leq n-1$, we have 
$$m (T_i \cap \phi^{-1} (H^+ \cup H) ) \geq \dfrac{1}{2} m (T_i ) \text{ and }  m (T_i \cap \phi^{-1} (H^- \cup H) ) \geq \dfrac{1}{2} m (T_i ) .$$
Note that by the assumption that $\phi (V)$ is in general position in $\mathbb{R}^n$, we have that for every $i$, there are at most $n$ vertices $v_1,...,v_n \in T_i$ such that
$$\phi (v_1),...,\phi (v_n) \in H.$$
Therefore we have for every $i$ that
$$m (T_i \cap \phi^{-1} (H)) \leq n  \dfrac{\Omega}{n} m(S_i) =\Omega m(S_i) .$$
Therefore 
$$m (T_i \cap \phi^{-1} (H^+) ) \geq \dfrac{1}{2} m (T_i ) - \Omega m(S_i) \text{ and }  m (T_i \cap \phi^{-1} (H^-) ) \geq \dfrac{1}{2} m (T_i )  - \Omega m(S_i)  .$$
Without loss of generality, we also have 
$$m (T_n \cap \phi^{-1} (H^+) ) \geq \dfrac{1}{2} m (T_n ) -  \Omega m(S_n) .$$
Then define new sets 
$$\forall 0 \leq i \leq j, T_i^{(1)} =T_i \cap \phi^{-1} (H^-) ,$$
$$\forall j+1 \leq i \leq n, T_i^{(1)} =T_i \cap \phi^{-1} (H^+) .$$
If $\overline{conv} (\phi (T_0^{(1)})),...,  \overline{conv} (\phi (T_n^{(1)}))$ is a separated $(n+1)$-tuple, we are done. Otherwise repeat the above process. Notice that after at most $2^n$ steps, we get a separated $(n+1)$-tuple $\overline{conv} (\phi (T_0')),...,  \overline{conv} (\phi (T_l'))$ with 
$$\forall 0 \leq i \leq n, m(T_i' ) \geq \dfrac{1}{2^{2^n}} m(T_i) -  \sum_{k=0}^{2^{n}-1} \dfrac{1}{2^k} \Omega m (S_i) .$$ 
In the same manner, we can have, for example,  $T''_0 \subseteq T_0',...,T''_{n-1} \subseteq T'_{n-1}, T_n'' = T_n'$ such that $\lbrace O \rbrace, \overline{conv} (\phi (T''_0)),..., \overline{conv} (\phi (T''_{n-1}))$ is a separated $(n+1)$-tuple (the only difference is that when applying the above process, we always keep the point $O$ even in the cases where $H$ passes through it). After that, we'll apply the same process to get $T_0''' \subseteq T_0'',...,T_{n-1}''' = T_{n-1}'',  T_{n}''' \subseteq T_{n}''$ such that $\lbrace O \rbrace, \overline{conv} (\phi (T''_0)),..., \overline{conv} (\phi (T''_{n-2})), \overline{conv} (\phi (T''_{n}))$ is a separated $(n+1)$-tuple and so on. Therefore, at the end we'll have sets $Q_0 \subseteq T_0,..., Q_n \subseteq T_n$ such that $\overline{conv} (\phi (Q_0)), ..., \overline{conv} (\phi (Q_n)), \lbrace O \rbrace$ is a separated family and for every $0 \leq i \leq n$, we have that:
\begin{dmath*}
m(Q_i) \geq \dfrac{1}{2^{n2^n}} m(T_i) -  \sum_{k=0}^{n 2^{n}-1} \dfrac{1}{2^k} \Omega m (S_i)  \geq \dfrac{1}{2^{n2^n}} m(T_i)  - 2 \Omega m(S_i) \geq \dfrac{1}{2^{n2^n}} m(T_i)  - 2 \dfrac{1}{2^{2+n2^n}} \dfrac{m(T_i)}{m(S_i)} m(S_i) \geq \dfrac{1}{2^{1+n2^n}} m(T_i).
\end{dmath*}

\end{proof}

Now we are finally ready to prove theorem \ref{weighted version of Pach}:
\begin{proof}
Let $V$ and $m : V \rightarrow \mathbb{R}^+$ as in the theorem. Denote $\varepsilon_1 = \frac{1}{2^{1+n2^n}}$ and take $\varepsilon_2$ such that $\varepsilon_2 > \varepsilon_1$ and such that the condition in lemma \ref{T_i lemma} holds (this can be done for instance, by choosing $\varepsilon_2$ as in the remark after lemma \ref{T_i lemma}). Denote $\alpha = \alpha (n , \varepsilon_1 , \varepsilon_2 )$ as the constant from lemma \ref{T_i lemma}. Choose 
$$\omega (n) = \min \left\lbrace \dfrac{1}{n 2^{2+n2^n}},  \varepsilon_2 - \varepsilon_1 \right\rbrace  \left( \dfrac{1}{(n+1)!} \right)^{\frac{1}{\alpha}}  .$$
Fix a disjoint partition $S_0,...,S_l$. We'll split the proof into two cases: \\
\textbf{Case 1 (assuming general position):} \\
Assume that $\phi : V \rightarrow \mathbb{R}^n$ such that $V$ to points in general position in $\mathbb{R}^n$.  \\
We'll use theorem \ref{Karasev result}. For every $0 \leq i \leq n$ define a measure
$$\mu_i = \sum_{u \in S_i} \dfrac{m(u)}{m(S_i)} \delta_{u} ,$$
where $\delta_{u}$ is the delta measure at $u$. By theorem \ref{Karasev result}, there is a point $O \in \mathbb{R}^n$, such that for the set 
$$A = \lbrace (u_0,...,u_n) \in S_0 \times ... \times S_n : O \in \overline{conv} (u_0,...,u_n) \rbrace,$$
we have the following inequality:
$$\sum_{(u_0,...,u_l) \in A } \mu_0 (u_0) ... \mu_n (u_n) \geq \dfrac{1}{(n+1)!} .$$
Note that this inequality can be rewritten (in the notation of lemma \ref{T_i lemma}) as:
$$e (A) \geq \dfrac{1}{(n+1)!} m(S_0) ... m(S_n) .$$
Note that for all $i$, we have that $O \notin \phi (S_i)$. Indeed, if, for example, $O \in \phi (S_0)$, then we denote $s_0 = \phi^{-1} (\lbrace O \rbrace) \in S_0$ (this is a single vertex, from the general position assumption) and without loss of generality, we can assume $A = \lbrace s_0 \rbrace \times S_1 \times ... \times S_n $. Therefore,
$$e (A) = m (s_0) m(S_1) ... m (S_n) \geq  \dfrac{1}{(n+1)!} m (S_0) m(S_1) ... m (S_n) .$$
This yields that 
$$m(s_0) \geq \dfrac{m(S_0)}{(n+1)!} > \omega(n) \min \lbrace m(S_0),...,m(S_n) \rbrace.$$
in contradiction to the choice of $\omega (n)$.    \\
By lemma \ref{T_i lemma} with $\varepsilon_1, \varepsilon_2, \alpha$ as above, there are sets $T_0 \subseteq S_0,..., T_n \subseteq S_n$ with
$$\forall 0 \leq i \leq n, m(T_i) \geq \left( \dfrac{1}{(n+1)!} \right)^{\frac{1}{\alpha}} m(S_i) .$$
Note that for every $0 \leq  i  \leq n$, we have that
$$\dfrac{m (T_i)}{m(S_i)} \dfrac{1}{n 2^{2+n2^n}} \geq  \left( \dfrac{1}{(n+1)!} \right)^{\frac{1}{\alpha}} \dfrac{1}{n 2^{2+n2^n}} \geq \omega (n).$$
Therefore, we can apply lemma \ref{Q_i generation} and get $Q_0 \subseteq T_0,..., Q_n \subseteq T_n$ such that  $\overline{conv} (\phi (Q_0)),..., \overline{conv} (\phi (Q_n)), \lbrace O \rbrace$ is a separated family in $\mathbb{R}^n$ and 
$$ \forall 0 \leq i \leq n, \dfrac{1}{2^{1+n2^n}} = \varepsilon_1 \leq \dfrac{m(Q_i)}{m (T_i)}.$$
Note that from the definition of $\omega (n)$ we have that for every  $u \in S_i$ that 
$$\dfrac{m(u)}{m(S_i)} \leq \omega (n),$$
and therefore for every $i$ and for every $u \in T_i$,
\begin{dmath*}
\dfrac{m(u)}{m(T_i)} = \dfrac{m(u)}{m(S_i)} \dfrac{m(S_i)}{m(T_i)} \leq \omega (n) \dfrac{m(S_i)}{m(T_i)} \leq \omega (n) \dfrac{1}{\left( \dfrac{1}{(n+1)!} \right)^{\frac{1}{\alpha}} } \leq \varepsilon_2 - \varepsilon_1 . 
\end{dmath*}
Therefore, by deleting elements from the $Q_i$'s, if necessary, we can make sure that  
$$ \forall 0 \leq i \leq n, \varepsilon_1 \leq \dfrac{m(Q_i)}{m (T_i)} \leq \varepsilon_2.$$
By lemma \ref{T_i lemma}, we have that 
$$e (A \cap (Q_0 \times ... \times Q_n ) ) >0.$$
This implies there is $(q_0,...,q_n) \in Q_0 \times ... \times Q_n$ such that $O \in conv (\phi (q_0),...,\phi (q_l))$. By our choice of $Q_0,...,Q_n$, $\overline{conv} (\phi (Q_0)) ,..., \overline{conv} (\phi (Q_n)), \lbrace O \rbrace$ is a separated family, and therefore by corollary \ref{intersection}, we have that 
$$O \in \bigcap_{(u_0,...,u_n ) \in Q_0 \times ... \times Q_n}  conv (\phi (u_0),...,\phi (u_n)) .$$
Notice that for every $0 \leq i \leq n$, 
$$m(Q_i) \geq \varepsilon_1 m(T_i) \geq \varepsilon_1 \left( \dfrac{1}{(n+1)!} \right)^{\frac{1}{\alpha}}  m(S_i) = \dfrac{1}{2^{1+n2^n}} \left( \dfrac{1}{(n+1)!} \right)^{\frac{1}{\alpha}} m(S_i) ,$$
therefore we can take 
$$c(n) = \dfrac{1}{2^{1+n2^n}} \left( \dfrac{1}{(n+1)!} \right)^{\frac{1}{\alpha}} .$$
\textbf{Case 2:} \\
Assume now that $\phi : V \rightarrow \mathbb{R}^n$ is arbitrary ($\phi (V)$ is not necessarily in general position in $\mathbb{R}^n$). By taking arbitrarily small perturbations of  $\phi$ we generate a sequence $\phi_j : V \rightarrow \mathbb{R}^n$ such that for each $j$, $\phi_j (V)$ is in general position in $\mathbb{R}^n$ and such that
$$\lim_{j \rightarrow \infty} \sup_{u \in V} \vert \phi (u) - \phi_j (u) \vert =0.$$
From case 1, we have that for every such $\phi_j$, there are sets $Q_0^j \subseteq S_0,...,Q_n^j \subseteq S_n$ and a point $O_j \in \mathbb{R}^n$, such that
$$ m(Q_i ) \geq c(n)  m(S_i),$$
and
$$O_j \in \bigcap_{(u_0,...,u_n) \in Q_0 \times ... \times Q_n } \overline{conv} (\phi (u_0),..., \phi (u_n) ) .$$
After passing to a subsequence, we can assume that there are sets $Q_0 \subseteq S_0,..., Q_n \subseteq S_n$ such that for every $j$,  
$$Q_0 = Q_0^j,..., Q_n = Q_n^j.$$
Also, up to passing to a subsequence, we can assume that the sequence $O_j$ is convergent in $\mathbb{R}^n$ and denote 
$$O = \lim_{j \rightarrow \infty} O_j.$$
Therefore we get that 
$$ m(Q_i ) \geq c(n)  m(S_i).$$
Also, for every $(u_0,...,u_n) \in Q_0 \times ... \times Q_n$ we have that
$$O_j \in \overline{conv} (\phi_j (u_0),...,\phi_j (u_n) ) ,$$
$$\lim_{j \rightarrow \infty} \phi_j (u_0) = \phi (u_0),..., \lim_{j \rightarrow \infty} \phi_j (u_n) = \phi (u_n), \lim_{j \rightarrow \infty} O_j = O .$$
This implies that for every $(u_0,...,u_n) \in Q_0 \times ... \times Q_n$,
$$O \in \overline{conv} (\phi (u_0),...,\phi (u_n) ),$$
and we are done.
\end{proof}

\begin{corollary}
\label{weighted version of Pach - corollary}
Let $V$ be a finite set and $m : V \rightarrow \mathbb{R}^+$ be some fixed map.
Then for $n \in \mathbb{N}$, we have that for every $\phi : V \rightarrow \mathbb{R}^n$, one of the following holds:
\begin{enumerate}
\item There is $u \in V$ such that
$$m (u ) \geq \omega (n) \dfrac{1}{2(n+1)} m (V) ,$$
where $\omega (n)$ is the constant in theorem \ref{weighted version of Pach}.
\item There are pairwise disjoint sets $Q_0,...,Q_n \subset V$ such that for every $0 \leq i \leq n$,
$$ m(Q_i ) \geq c(n) \dfrac{1}{2(n+1)} m(V),$$
where $c (n)$ is the constant in theorem \ref{weighted version of Pach}, and
$$\bigcap_{(u_0,...,u_n) \in Q_0 \times ... \times Q_n } \overline{conv} (\phi (u_0),..., \phi (u_n) ) \neq \emptyset ,$$
where $\overline{conv} (u_0,...,u_n )$ is the closure of the convex hull of $\phi (u_0),..., \phi (u_n)$ (i.e., the closed simplex spanned by $\phi (u_0),...,\phi (u_n)$ in $\mathbb{R}^n$). 
\end{enumerate}
\end{corollary}

\begin{proof}
Fix $\phi : V \rightarrow \mathbb{R}^n$. By the choice of $\omega(n)$ in the proof of theorem \ref{weighted version of Pach}, it is clear that $\omega (n) < 1$. Therefore, we can assume that for every $u \in V$, 
$$m(u) < \frac{1}{2(n+1)} m(V).$$ 
Consider now the following partitioning algorithm: order the elements of $V$ as $u_1,...,u_{\vert V \vert}$ such that for all $j$,
$m(u_j) \geq m(u_{j+1})$. Partition accolading to the following algorithm:
\begin{enumerate}
\item $S^{0}_0 = \emptyset,...,S^0_n =\emptyset$.
\item For $1 \leq j \leq \vert V \vert$, choose $S^{j-1}_{i}$ such that 
$$m(S^{j-1}_i ) = \min \lbrace m(S^{j-1}_0) ,..., m(S^{j-1}_n) \rbrace,$$
(for this algorithm, if $S^{j-1}_i$ is empty, then $m(S^{j-1}_i) =0$). 
Set 
$$S^{j}_0 = S^{j-1}_0,...,S^{j}_{i-1} = S^{j-1}_{i-1}, S^{j}_{i+1} = S^{j-1}_{i+1},...,S^{j}_n = S^{j-1}_n ,$$
and
$$S^{j}_i = S^{j-1}_i \cup \lbrace u_i \rbrace .$$
\item Denote $S_0 = S^{\vert V \vert}_0,..., S_n = S^{\vert V \vert}_n $.
\end{enumerate}
Following this algorithm, it is easy to see that for all $i$, $m(S_i) > \frac{1}{2(n+1)} m(V)$. Therefore for every $u \in V$, we have that 
$$m(u) < \omega (n) \frac{1}{2(n+1)} m(V) < \omega (n) \min \lbrace m(S_0),...,m(S_n) \rbrace.$$
By theorem \ref{weighted version of Pach} there are sets $Q_0 \subseteq S_0,...,Q_n \subseteq S_n$ such that for every $0 \leq i \leq n$,
$$ m(Q_i ) \geq c(n)  m(S_i) \geq c(n) \frac{1}{2(n+1)} m(V) ,$$
and
$$\bigcap_{(u_0,...,u_n) \in Q_0 \times ... \times Q_n } \overline{conv} (\phi (u_0),..., \phi (u_n) ) \neq \emptyset .$$

\end{proof}

\begin{remark}
The reader should note that throughout this appendix, we did not optimize our arguments to get the best constants. 
\end{remark}

\bibliographystyle{alpha}
\bibliography{bibl}
\Addresses
\end{document}